\documentclass[11pt,fleqn]{article}

\title{Simplicial presheaves of\\Green complexes and twisting cochains}
\author{Timothy Hosgood \and Mahmoud Zeinalian}
\date{}

\usepackage[style=alphabetic,backend=bibtex,maxnames=99,maxalphanames=4]{biblatex}
\DeclareFieldFormat{eprint:HAL}{\textsc{hal}: \href{https://tel.archives-ouvertes.fr/#1}{\texttt{#1}}}
\renewbibmacro{in:}{%
  \ifboolexpr{%
     test {\ifentrytype{article}}%
  }{}{\printtext{\bibstring{in}\intitlepunct}}%
}

\usepackage{amsmath,amssymb}
\usepackage[shortlabels]{enumitem}
\usepackage{hyperref}
\usepackage{xcolor}
\hypersetup{colorlinks,linkcolor={blue!50!black},citecolor={blue!50!black},urlcolor={blue!50!black}}
\usepackage[nameinlink,noabbrev,capitalize]{cleveref}
\usepackage{tikz}
\usetikzlibrary{decorations.pathmorphing}
\usepackage{tikz-cd}
\usetikzlibrary{arrows,arrows.meta}
\tikzcdset{arrow style=tikz, diagrams={>=latex}}
\usepackage{booktabs}
\usepackage{mathtools}

\usepackage{amsthm}
\usepackage{thmtools}

\declaretheorem[name=Theorem,numberwithin=subsection]{theorem}
\declaretheorem[name=Corollary,numberlike=theorem]{corollary}
\declaretheorem[name=Lemma,numberlike=theorem]{lemma}
\declaretheorem[name=Conjecture,numberlike=theorem,refname={Conjecture,Conjectures},Refname={Conjecture,Conjectures}]{conjecture}

\declaretheorem[name=Definition,style=definition,numberlike=theorem]{definition}
\declaretheorem[name=Remark,style=definition,numberlike=theorem]{remark}

\usepackage{etoolbox}
\AtEndEnvironment{definition}{\null\hfill$\lrcorner$}
\AtEndEnvironment{definition*}{\null\hfill$\lrcorner$}
\AtEndEnvironment{remark}{\null\hfill$\lrcorner$}
\AtEndEnvironment{remark*}{\null\hfill$\lrcorner$}
\AtEndEnvironment{notation}{\null\hfill$\lrcorner$}
\AtEndEnvironment{notation*}{\null\hfill$\lrcorner$}

\newlist{theoremenum}{enumerate}{1}
\setlist[theoremenum]{label=(\alph*), ref=\thetheorem~(\alph*)}
\crefalias{theoremenumi}{theorem}

\usepackage[margin=1.55in,top=1in,bottom=1.2in]{geometry}

\renewcommand{\leq}{\leqslant}
\renewcommand{\geq}{\geqslant}
\renewcommand{\cal}[1]{{\mathcal{#1}}}
\newcommand{\scr}[1]{{\mathscr{#1}}}

\newcommand{\loccit}{\emph{loc.\@ cit.\@}}

\newcommand{\OO}{\scr{O}}

\newcommand{\RingSpace}{\mathsf{RingSp}}
\newcommand{\ConnRingSpace}{\mathsf{RingSp}_\mathrm{conn}}

\newcommand{\Set}{\mathsf{Set}}
\newcommand{\Space}{\mathsf{Space}}

\newcommand{\sSet}{\mathsf{sSet}}
\newcommand{\csSet}{\mathsf{csSet}}
\newcommand{\dgCat}{\mathsf{dg}\text{-}\mathsf{Cat}}
\newcommand{\Kan}{\mathsf{Kan}}
\newcommand{\QuasiCat}{\mathsf{Quasi}\text{-}\mathsf{Cat}}
\newcommand{\Mod}[1]{\OO_{#1}\text{-}\mathsf{Mod}}
\newcommand{\LieGroup}{\mathsf{LieGroup}}
\newcommand{\cover}{\cal{U}}
\newcommand{\anothercover}{\cal{V}}
\newcommand{\Smth}{\mathsf{Man}}
\newcommand{\SmthC}{\Smth_{\cover}}
\newcommand{\nerve}{\cal{N}}
\newcommand{\delcech}{\hat{\cal{C}}}
\newcommand{\cechnerve}{\check{\nerve}}
\newcommand{\dgnerve}{\cal{N}^\mathrm{dg}}
\newcommand{\bigfunctor}{\mathbf{N}}
\newcommand{\op}{\mathrm{op}}
\newcommand{\id}{\mathrm{id}}
\newcommand{\GL}{\mathrm{GL}}
\newcommand{\anotherbullet}{\star}

\newcommand{\bary}[1]{#1_{\mathrm{bary}}}
\newcommand{\pair}[1]{#1_{\mathrm{pair}}}
\newcommand{\simto}{\xrightarrow{\raisebox{-0.7ex}[0ex][0ex]{$\sim$}}}
\newcommand{\from}{\leftarrow}
\newcommand{\simfrom}{\xleftarrow{\raisebox{-0.7ex}[0ex][0ex]{$\sim$}}}

\newcommand{\simplexpath}[1]{\left[\begin{smallmatrix}#1\end{smallmatrix}\right]}
\usepackage{stmaryrd}
\newcommand{\core}[1]{\llbracket#1\rrbracket}

\newcommand{\Green}{\scr{Green}}
\newcommand{\Twist}{\scr{Twist}}
\newcommand{\sTwist}{\scr{sTwist}}
\newcommand{\sGreen}{\scr{sGreen}}

\DeclareMathOperator{\Free}{\mathsf{Free}}

\DeclareMathOperator{\Bun}{\scr{Bun}_{\GL_n(\mathbb{R})}}

\DeclareMathOperator{\Tot}{Tot}
\DeclareMathOperator{\holim}{holim}
\DeclareMathOperator{\Hom}{Hom}

\DeclareMathOperator{\eq}{eq}
\DeclareMathOperator{\ver}{ver}
\DeclareMathOperator{\codim}{codim}

\DeclareMathOperator*{\colim}{colim}

\usepackage{newunicodechar}
\newunicodechar{よ}{\text{\usefont{U}{min}{m}{n}\symbol{'210}}}
\DeclareFontFamily{U}{min}{}
\DeclareFontShape{U}{min}{m}{n}{<-> udmj30}{}
\newcommand{\yon}{よ}

\usepackage[en-GB]{datetime2}
\DTMlangsetup[en-GB]{ord=raise,daymonthsep={\space of\space},monthyearsep={,\space}}

\usepackage{fouriernc}
\usepackage[scr=esstix]{mathalfa}

\usepackage[labelsep=period]{caption}

\crefname{section}{Section}{Sections}
\crefname{equation}{}{}
\counterwithin{figure}{subsection}
\renewcommand{\thefigure}{\arabic{section}.\arabic{subsection}.\roman{figure}}

\usepackage{tocloft}
\cftpagenumbersoff{part}

\usepackage{xurl}
\hypersetup{breaklinks=true}

\usetikzlibrary{arrows.meta}
\usetikzlibrary{patterns}

\tikzset{every picture/.style={thick}}

\tikzstyle{vertex}[black] = [draw=black,solid,fill=#1,circle,inner sep=0pt,minimum size=1.2mm]
\tikzstyle{edge} = [thick]
\tikzstyle{hatched}[gray] = [pattern=crosshatch dots, pattern color=#1]
\tikzstyle{filled}[gray] = [fill=#1,fill opacity=0.5]

\begin{document}

\maketitle

\begin{abstract}
  We construct three simplicial presheaves on the site of ringed spaces, and in particular on that of complex manifolds.
  The descent objects for these simplicial presheaves yield Toledo--Tong's twisting cochains, simplicial twisting cochains, and complexes that appear in Green's thesis on Chern classes for coherent analytic sheaves, respectively.
  We thus extend the aforementioned constructions to the equivariant setting, and more generally to stacks.
  Although this is the first step in achieving push-forwards in K-theory and Riemann--Roch theorems for appropriate stacks, continuing a programme of Toledo and Tong, and O'Brian and Green, this paper is one of higher category theory in the language of simplicial sets and dg-categories, and thus does not involve any technicalities of complex geometry.
\end{abstract}

{\small\tableofcontents}

\section{Introduction}
\label{section:introduction}

\subsection{History and motivation}

The problem of resolving a coherent sheaf by locally free sheaves is fundamental in geometry.
Indeed, one of the main tools in proving the Hirzebruch--Riemann--Roch theorem for holomorphic bundles on smooth projective complex varieties is a resolution of the pushforward along the diagonal of the structure sheaf by a bounded complex of locally free sheaves.

To prove the analogous statement in the non-algebraic setting, various tools from differential geometry, such as heat kernels, are used.
These tools rely on the choice of a metric, which, outside the context of Kähler manifolds within complex geometry, is unnatural, preventing us from generalising to the equivariant setting and to that of stacks.
To resolve a coherent sheaf on a compact complex manifold by vector bundles, it suffices to have a positive line bundle.
Such a line bundle, readily available in the algebraic setting by the canonical line bundle, does not exist in general.
As such, outside the algebraic setting, coherent analytic sheaves cannot always be resolved by locally free sheaves.

Nevertheless, in a series of papers (\cite{TT1976,TT1978,OTT1981a,OTT1981b,OTT1985,TT1986}), O'Brian, Toledo, and Tong showed that on arbitrary complex manifolds, coherent analytic sheaves can be locally resolved by free sheaves, and these local sheaves glue together using transition functions that satisfy cocycle conditions up to an infinitely-coherent system of homotopies.
They called the ensemble of transition functions and associated homotopies a \emph{twisting cochain}.
In modern language, these objects would be described in terms of the $\infty$-stackification of the presheaf of perfect complexes on the site of complex manifolds (as in \cite{GMTZ2022b}).
They used these homotopic methods to give a proof of the Hirzebruch--Riemann--Roch theorem for coherent analytic sheaves (\cite{OTT1981c}) and extended it to a proof of the more general Grothendieck--Riemann--Roch theorem (\cite{OTT1985}).

Another consequence of their work was to answer to an appeal of Bott, in \cite{Bot1973}, amongst other places, for \emph{``the construction of characteristic classes of bundles in terms of transition functions''}.
Indeed, when working with characteristic classes of bundles on foliations, or quotients by group actions, such an approach is necessary.

In 1980, a philosophically related but technically different approach to resolving coherent sheaves appeared in the thesis \cite{Gre1980} of Green, a student of O'Brian and Eels, in which Chern classes in de Rham cohomology of coherent analytic sheaves were constructed from local free resolutions that globally clutch together via a simplicial system of strictly invertible chain maps and inclusions, now referred to as \emph{Green complexes}.
Green's key insight was turning holomorphic twisting cochains into simplicial objects satisfying strict identities on the nose.

To relate these different approaches, Toledo and Tong in \cite{TT1986} gave a reformulation of Green's simplicial resolution in terms of objects that simultaneously generalise twisting cochains and the complexes of sheaves on the Čech nerve arising in \cite{Gre1980}, namely \emph{simplicial twisting cochains}.
As mentioned in \cite{HS2001}, the work of O'Brian, Toledo, and Tong responds to a question posed in \cite{BGI1971} concerning Riemann--Roch formulas using Čech calculations that are an example of descent for complexes of vector bundles, but \emph{``a better general framework for these calculations could contribute to our understanding of Riemann--Roch formulas''}.

Since then, there has been important work on better understanding the homotopy theory of twisting cochains (such as \cite{Wei2016,BHW2017,Wei2021,GMTZ2022a,GMTZ2022b}), but the full story of how this applies to various open problems in complex-analytic geometry is one that has yet to be fully told.
Even the abstract foundations provide need for further study: as mentioned in \cite[Remark~2.16]{Wei2016}, the connections between twisting cochains and the dg-nerve should be further explored.

\subsection{Purpose}
\label{subsection:purpose}

The fundamental idea of this paper is the following: to construct a simplicial presheaf $\sTwist$ of simplicial twisting cochains that recovers, via some analogue of sheafification, the simplicial twisting cochains of \cite{TT1986} in the complex-analytic setting.
As special cases, we will also obtain two more simplicial presheaves, $\Twist$ and $\Green$, that recover twisting cochains and Green complexes, respectively.
We give a motivation of these constructions via perfectness conditions of sheaves of $\OO_X$-modules, and ``homotopical weakening'', in \cref{subsection:narrative}.
These simplicial presheaves are defined on the category of connected ringed spaces.

We show (\cref{theorem:tot0-all-three}) that, if one picks the specific ringed space corresponding to some complex-analytic manifold, then this aforementioned analogue of sheafification (which we call \emph{Čech totalisation}) recovers the classical definitions that we would expect from the three simplicial presheaves.
Although we do not discuss what happens in the case of other geometries (such as locally Noetherian schemes with affine covers), the formal machinery that we provide can immediately be extended to these settings.

\bigskip

One sees that the dg-nerve and twisting cochains should be related to one another, since both are given by the Maurer--Cartan condition (see e.g. \cite[Remark~2.16]{Wei2016}).
In constructing the simplicial presheaf for twisting cochains, we show how the defining equations of the dg-nerve translate exactly to those for twisting cochains, via intermediary results concerning Maurer--Cartan elements in Čech bicomplexes (such as \cref{theorem:dg-nerve-and-MC-for-presheaves}) which can be thought of as upgrades of certain technical lemmas from \cite{GMTZ2022a} to the case of \emph{presheaves} of dg-categories.
Furthermore, we show that not only twisting cochains, but also the weak equivalences between them (as defined in \cite[Definition~2.27]{Wei2016}), arise from the dg-nerve (\cref{theorem:1-simplices-in-complex-analytic-twist}).

As mentioned above, we endow the simplicial presheaves with geometry via Čech totalisation (\cref{subsection:cech-totalisation}), which consists of evaluating a simplicial presheaf on the Čech nerve of some fixed cover and then taking the totalisation of the resulting cosimplicial simplicial set, which we show computes the homotopy limit.
In this way, we are providing the space analogue of \cite[Proposition~4.9]{BHW2017}, showing that twisting cochains arise as a homotopy limit of bounded complexes of free modules evaluated on the Čech nerve; we similarly characterise Green complexes and simplicial twisting cochains as homotopy colimits of other cosimplicial simplicial sets.
We make this analogy precise via a comparison result (\cref{lemma:comparsion-of-tot-dg-and-tot-sset}) for presheaves of dg-categories that preserve finite limits.

The results of this present article concerning only twisting cochains can thus be seen as a sort of synthesis of \cite{GMTZ2022a}, \cite{BHW2017}, and \cite{Wei2016}.

\bigskip

Although weak equivalences between Green complexes were defined in \cite{Hos2020a} as level-wise quasi-isomorphisms, the $1$-simplices of the Čech totalisation of $\Green$ here provide a seemingly more fitting notion (\cref{subsection:1-simplices-in-complex-green}).
We conjecture that the description of Green's resolution given in \cite{TT1986} actually describes a morphism of simplicial presheaves $\sTwist\to\Green$ (\cref{conjecture:TTs-green-gives-morphism-stwist-to-green}), though to prove this would require a refinement of the constructions given here, as we justify in \cref{subsection:greens-resolution}.

Given that we construct three simplicial presheaves, it is natural to ask how they relate to one another in the category of simplicial presheaves, which can be endowed with a model structure (namely, the classical Kan--Quillen model structure, as constructed in e.g. \cite[\S1]{GJ2009}).
By construction, $\Twist$ is globally fibrant, and will thus give a space after Čech totalisation (\cref{lemma:tot-of-cech-of-kan-is-kan}), but neither $\Green$ nor $\sTwist$ are immediately seen to be globally fibrant.
However, we provide some partial results (\cref{subsection:horn-filling-conditions}) which ensure that the simplicial $\pi_0$ are well defined for all three presheaves, and then show that they are all equivalent, noting a particularly nice application of Green's resolution in strictifying quasi-isomorphisms of complexes of free modules to isomorphisms (\cref{remark:green-lets-you-turn-quasi-isos-into-isos}).
As a consequence of global fibrancy, this implies that the $\pi_0$ of their Čech totalisations are also all equivalent (\cref{corollary:pi0-all-three-equivalent}).

\bigskip

In an attempt to make this paper as useful a reference as possible, we try not to leave any proofs as exercises for the reader: in \cref{appendix:example} and \cref{appendix:proof-of-1-simplices-in-complex-analytic-twist} we provide explicit descriptions and calculations of $1$-simplices in the Čech totalisation of a simplicial presheaf.

\subsection{Overview}
\label{subsection:overview}

\begin{itemize}
  \item \cref{section:introduction}: Historical context and brief summary of main results of this paper.
  \item \cref{section:preliminaries}: Preliminary notation and definitions, as well as some general technical lemmas. Mostly classical, but some new or folklore results as well, especially concerning the dg-nerve (\cref{subsection:dg-nerve}), pair subdivision (\cref{subsection:the-barycentric-and-pair-subdivision}), and Čech totalisation (\cref{subsection:cech-totalisation}).
  \item \cref{section:three-simplicial-presheaves}: Motivation for our approach via simplicial presheaves, the dg-nerve, and simplicial labelling (\cref{subsection:narrative}), and definitions of the three simplicial presheaves.
  \item \cref{section:complex-analytic-examples}: Details of the holomorphic case, including comparisons to other results.
  \item \cref{section:morphisms-between-the-presheaves}: Study of the three simplicial presheaves in the context of the Kan--Quillen model structure.
  \item \cref{section:summary-and-future-work}: Summary of main results, including questions for future research.
  \item \cref{appendix:example}: Worked example of Čech totalisation for constructing the space of principal $\GL_n(\mathbb{R})$-bundles. 
  \item \cref{appendix:technical-proofs}: Details of lengthier proofs of some of the more technical lemmas.
\end{itemize}

\subsection{Acknowledgments}

We thank Cheyne Glass, Michah Miller, and Thomas Tradler for providing us with an early copy of \cite{GMTZ2022a}; versions of both \cref{lemma:morphism-into-dg-nerve} and \cref{theorem:dg-nerve-and-MC} can be found there.
We also thank the two anonymous reviewers of this article, who helped improve the exposition, structure, and framing, as well as correcting mistakes, errors, and typos.

The first author would also like to thank Evan Cavallo for his patience in helping with the numerous sign issues in an early draft of \cref{appendix:proof-of-1-simplices-in-complex-analytic-twist}, as well as Ivan Di Liberti and Josefien Kuijper for continual interesting conversations.

\section{Preliminaries}
\label{section:preliminaries}

The majority of content in this section is classical, and we simply gather it together here for convenience of the reader, as well as to establish notation.
Some sections (such as \cref{subsection:dg-nerve}, \cref{subsection:the-barycentric-and-pair-subdivision}, and \cref{subsection:cech-totalisation}) contain material that is either slightly more general than what can be found in existing literature, or that is more difficult to find references for.

\emph{Throughout this entire paper, whenever we say ``manifold'' we mean ``paracompact smooth manifold''; whenever we say ``cover'' we mean ``(locally finite) open cover''.}
\emph{For categories $\cal{C}$ and $\cal{D}$, we denote the set (or category) of functors from $\cal{C}\to\cal{D}$ by $[\cal{C},\cal{D}]$.}
\emph{We always use $\subset$ to mean strict subset, and $\subseteq$ to mean non-strict subset, entirely analogous to $<$ and $\leq$.}

\subsection{Spaces via simplicial sets}
\label{subsection:simplicial-sets-and-spaces}

\begin{definition}
  Let $\Delta$ be the abstract simplex category, whose objects are the finite totally ordered sets $[p]=\{0<1<\ldots<p\}$ for $p\in\mathbb{N}$, and whose morphisms $[p]\to[q]$ are order-preserving functions.
  There are injections $f_p^i\colon[p-1]\to[p]$ for $i\in\{0,\ldots,p-1\}$, called \emph{coface} maps; there are surjections $s_i^p\colon[p+1]\to[p]$ for $i\in\{0,\ldots,p\}$, called \emph{codegeneracy} maps.
  The \emph{topological $p$-simplex}, denoted by $\Delta^p$, is the $p$-dimensional polytope given by the convex hull of the affinely independent\footnote{That is, the set $\{e_i-e_0 \mid i=1,\ldots,p\}$ is linearly independent.} set of $p+1$ points $\{e_0,e_1,\ldots,e_p\}$ inside $\mathbb{R}^p$, where $e_i$ is the standard basis unit vector for $i\in\{1,\ldots,p\}$, and $e_0$ is the zero vector.
  Ordered non-empty subsets $\sigma\subseteq[p]$ of cardinality~$k+1$ then correspond bijectively to non-degenerate sub-$k$-simplices $\Delta^k\subseteq\Delta^p$, since $\sigma$ corresponds to a subset of the aforementioned set of $p+1$ points, and we can then take its convex hull (cf. \cref{figure:subsets-give-subsimplices}).
\end{definition}

\begin{figure}[!ht]
  \centering
  \begin{tabular}{c|c}
    subsets
    & sub-simplices
  \\\midrule
    $\{0<2\}\hookrightarrow\{0<1<2\}$
    &
    \begin{tikzpicture}[scale=0.6,baseline=0]
      \begin{scope}[shift={(-2.5,-0.7)}]
        \node [vertex,label={below:{\footnotesize$0$}}] (0) at (0,0) {};
        \node [vertex,label={below:{\footnotesize$2$}}] (2) at (2,0) {};
        \draw [edge] (0) to (2);
      \end{scope}
      \node at (1,0) {$\hookrightarrow$};
      \begin{scope}[shift={(2.5,-0.7)}]
        \node [vertex,label={below left:{\footnotesize$0$}}] (0) at (0,0) {};
        \node [vertex,label={above:{\footnotesize$1$}}] (1) at (1,1.5) {};
        \node [vertex,label={below right:{\footnotesize$2$}}] (2) at (2,0) {};
        \draw [edge,hatched] (0.center) to (1.center) to (2.center) to cycle;
      \end{scope}
    \end{tikzpicture}
  \\[3em]
    $\{0<2<3\}\hookrightarrow\{0<1<2<3\}$
    &
    \begin{tikzpicture}[scale=0.6,baseline=0]
      \begin{scope}[shift={(-2.5,0)}]
        \node [vertex,label={left:{\footnotesize$0$}}] (0) at (0,0) {};
        \node [vertex,label={below:{\footnotesize$2$}}] (2) at (1.5,-0.5) {};
        \node [vertex,label={[label distance=-1mm]above right:{\footnotesize$3$}}] (3) at (2,0.5) {};
        \draw [edge,hatched] (0.center) to (2.center) to (3.center) to cycle;
      \end{scope}
      \node at (1,0) {$\hookrightarrow$};
      \begin{scope}[shift={(2.5,-0.4)}]
        \node [vertex,label={below left:{\footnotesize$0$}}] (0) at (0,0) {};
        \node [vertex,label={above:{\footnotesize$1$}}] (1) at (1,1.4) {};
        \node [vertex,label={below:{\footnotesize$2$}}] (2) at (1.5,-0.5) {};
        \node [vertex,label={right:{\footnotesize$3$}}] (3) at (2,0.5) {};
        \draw [edge,filled] (0.center) to (1.center) to (2.center) to cycle;
        \draw [edge,filled] (1.center) to (2.center) to (3.center) to cycle;
        \draw [edge,dashed] (0) to (3);
      \end{scope}
    \end{tikzpicture}
  \end{tabular}
  \caption{Inclusions of subsets correspond to inclusions of sub-simplices.}
  \label{figure:subsets-give-subsimplices}
\end{figure}

\begin{definition}
  A \emph{simplicial set} is a contravariant functor $X_\bullet\colon[p]\mapsto X_p$ from $\Delta$ to the category of sets, i.e. an object of the category $\sSet\coloneqq[\Delta^\op,\Set]$.
  The coface maps $f_p^i$ of $\Delta$ induce \emph{face} maps $X_\bullet f_p^i\colon X_p\to X_{p-1}$; the codegeneracy maps $s_i^p$ of $\Delta$ induce \emph{degeneracy} maps $X_\bullet s_i^p\colon X_p\to X_{p+1}$.

  Given a category $\cal{C}$, a \emph{simplicial presheaf} on $\cal{C}$ is a presheaf with values in simplicial sets, i.e. an object in $[\cal{C}^\op,\sSet]$.
\end{definition}

We often simply write $X$ instead of $X_\bullet$, and $f^i$ (resp. $s_i$) instead of $X_\bullet f^i_p$ (resp. $X_\bullet s_i^p$).

Since simplicial sets give a model for topological spaces (through geometric realisation), we often refer to the $0$-simplices $x\in X_0$ of a simplicial set $X_\bullet$ as \emph{points} or \emph{vertices}, and the $1$-simplices $x\in X_1$ as \emph{lines} or \emph{edges}.
We reserve the use of the word ``\emph{space}'' to refer to either topological spaces or Kan complexes (defined below), and if we do not make precise to which one we are referring then it is because one can pick either meaning, depending on preference.

\begin{definition}
  The prototypical simplicial sets are those of the form
  \[
    \Delta[p]\coloneqq\Hom_\Delta(-,[p])
  \]
  for $p\in\mathbb{N}$.
  We call $\Delta[p]$ the \emph{standard $p$-simplex}.
\end{definition}

Although we work almost entirely with the ``abstract'' simplices $\Delta[p]$, when drawing diagrams we are really drawing the topological simplices $\Delta^p$, which are related to the abstract simplices by \emph{geometric realisation}.
One needs to be careful about the definition of the category $\Space$ in the definition of geometric realisation, but we do not need to worry about the details here.
For us, what is important is the intuitive understanding of geometric realisation: we take a simplicial set $X_\bullet$, replace every copy of the abstract $p$-simplex with a topological $p$-simplex, and then glue these together exactly in the way that the abstract simplices glue together via the face and degeneracy maps.

\begin{definition}
\label{definition:geometric-realisation}
  We define \emph{geometric realisation} $|\cdot|\colon\sSet\to\Space$ as the functor given on the standard $p$-simplices by $|\Delta[p]|\coloneqq\Delta_p$, and then extend this to arbitrary simplicial sets $X_\bullet$ via
  \[
    |X_\bullet| \coloneqq \colim_{\Delta[p]\to X_\bullet} |\Delta[p]|.
  \]
  More formally, geometric realisation is the left adjoint to the functor $\operatorname{Sing}\colon\Space\to\sSet$ given by $\operatorname{Sing}(Y)_p\coloneqq\Hom_\Space(\Delta^p,Y)$.
\end{definition}

\begin{definition}
  Given $0\leq i\leq p$, the \emph{$i$th horn $\Lambda_i[p]$ of the $p$-simplex} is the simplicial set defined by
  \[
    \Lambda_i[p]([q])
    = \Big\{
      \alpha\in\Hom_\Delta([q],[p]) \mid [p]\not\subseteq\alpha([q])\cup\{i\}
    \Big\}
    \subset \Delta[p]([q]).
  \]
  Topologically, the $i$th horn is what remains after removing the interior of the $p$-simplex and then deleting the $(p-1)$-dimensional face opposite the $i$th vertex (cf. \cref{figure:horns}); more simply, it is the collection of all simplices that contain the $i$th vertex.
  We write $\Lambda_i^p$ to mean the geometric realisation of $\Lambda_i[p]$, so that $\Lambda_i^p\subset\Delta^p$.
\end{definition}

\begin{figure}[!ht]
  \centering
  \begin{tikzpicture}[scale=0.6,baseline=0]
    \begin{scope}[shift={(-2.5,-0.7)}]
      \node [vertex,label={below:{\footnotesize$0$}}] (0) at (0,0) {};
      \node [vertex,label={above:{\footnotesize$1$}}] (1) at (1,1.5) {};
      \node [vertex,label={below:{\footnotesize$2$}}] (2) at (2,0) {};
      \draw [edge] (0) to (1) to (2);
    \end{scope}
    \node at (1,0) {$\hookrightarrow$};
    \begin{scope}[shift={(2.5,-0.7)}]
      \node [vertex,label={below left:{\footnotesize$0$}}] (0) at (0,0) {};
      \node [vertex,label={above:{\footnotesize$1$}}] (1) at (1,1.5) {};
      \node [vertex,label={below right:{\footnotesize$2$}}] (2) at (2,0) {};
      \draw [edge,hatched] (0.center) to (1.center) to (2.center) to cycle;
    \end{scope}
  \end{tikzpicture}
  \\
  \begin{tikzpicture}[scale=0.6,baseline=0]
    \begin{scope}[shift={(-2.5,-0.4)}]
      \node [vertex,label={below left:{\footnotesize$0$}}] (0) at (0,0) {};
      \node [vertex,label={above:{\footnotesize$1$}}] (1) at (1,1.4) {};
      \node [vertex,label={below:{\footnotesize$2$}}] (2) at (1.5,-0.5) {};
      \node [vertex,label={right:{\footnotesize$3$}}] (3) at (2,0.5) {};
      \draw [edge,hatched] (0.center) to (1.center) to (2.center) to cycle;
      \draw [edge,hatched] (0.center) to (1.center) to (3.center) to cycle;
      \draw [edge,hatched] (0.center) to (2.center) to (3.center) to cycle;
      \draw [edge,dashed] (0) to (3);
    \end{scope}
    \node at (1,0) {$\hookrightarrow$};
    \begin{scope}[shift={(2.5,-0.4)}]
      \node [vertex,label={below left:{\footnotesize$0$}}] (0) at (0,0) {};
      \node [vertex,label={above:{\footnotesize$1$}}] (1) at (1,1.4) {};
      \node [vertex,label={below:{\footnotesize$2$}}] (2) at (1.5,-0.5) {};
      \node [vertex,label={right:{\footnotesize$3$}}] (3) at (2,0.5) {};
      \draw [edge,filled] (0.center) to (1.center) to (2.center) to cycle;
      \draw [edge,filled] (1.center) to (2.center) to (3.center) to cycle;
      \draw [edge,dashed] (0) to (3);
    \end{scope}
  \end{tikzpicture}
  \caption{Top: $\Lambda_1^2\subset\Delta^2$; Bottom: $\Lambda_0^3\subset\Delta^3$ (where, on the left, all $2$-dimensional faces are filled in except for $\{1<2<3\}$).}
  \label{figure:horns}
\end{figure}

\begin{definition}
\label{definition:kan-complex}
  A simplicial set $X_\bullet$ is a \emph{Kan complex} if all horns can be filled, i.e. if any map $\Lambda_i[p]\to X_\bullet$ can be extended to a map $\Delta[p]\to X_\bullet$ for all $p\in\mathbb{N}$ and all $0\leq i\leq p$, i.e. if the natural map $\Hom_\sSet(\Delta[p],X_\bullet)\to\Hom(\Lambda_i[p],X_\bullet)$ is surjective.

  If the same condition holds only for $0<i<p$ (for all $p\in\mathbb{N}$), then we say that only \emph{inner} horns can be filled, and we say that the simplicial set is a \emph{quasi-category}.

  This defines two full subcategories of the category of simplicial sets: the category $\Kan$ of Kan complexes, and the category $\QuasiCat$ of quasi-categories.
\end{definition}

\begin{definition}
\label{definition:max-kan}
  The inclusion $\Kan\hookrightarrow\QuasiCat$ has both a left and a right adjoint, where the right adjoint is called the \emph{core}.
  It can be shown (\cite[\href{https://kerodon.net/tag/01D9}{Tag 01D9}]{kerodon}) that the core is given by taking the \emph{maximal Kan complex}, i.e. the largest (by inclusion) Kan complex contained inside the quasi-category.

  We denote this maximal-Kan-complex functor by $\core{-}\colon\QuasiCat\to\Kan$.
\end{definition}

\begin{remark}
  The core of an arbitrary simplicial set is not a priori well defined, but there is a ``model'' of the core which is indeed a functor defined on all of $\sSet$.
  We return to this point in \cref{remark:extending-core-to-sset}.
\end{remark}

\subsection{Cosimplicial simplicial sets}

\begin{definition}
  A \emph{cosimplicial simplicial set} is a covariant functor $X_\bullet^\anotherbullet\colon[p]\mapsto X_\bullet^p$ from $\Delta$ to the category of simplicial sets, i.e. an object of the category $\csSet\coloneqq[\Delta,\sSet]$.
  The coface maps $f_p^i$ of $\Delta$ induce \emph{coface} maps $X_\bullet^\anotherbullet f_p^i\colon X_\bullet^{p-1}\to X_\bullet^p$; the codegeneracy maps $s_i^p$ of $\Delta$ induce \emph{codegeneracy} maps $X_\bullet^\anotherbullet s_i^p\colon X_\bullet^{p+1}\to X_\bullet^p$.
\end{definition}

Note that we can enrich $\csSet$ over $\sSet$ by defining
\[
  \big(\underline{\Hom}_{\csSet}(A_\bullet^\anotherbullet,B_\bullet^\anotherbullet)\big)_p
  = \Hom_{\csSet}(A_\bullet^\anotherbullet\times\Delta[p],B_\bullet^\anotherbullet).
\]

\begin{remark}
  Just to be clear: since we are using simplicial sets as models for spaces, when we talk about the coface maps of a cosimplicial simplicial set, we mean the coface maps coming from the \emph{co}simplicial structure, \emph{not} the face maps coming from the simplicial structure.
\end{remark}

\begin{definition}
  The prototypical cosimplicial simplicial set is
  \[
    \Delta[\anotherbullet]\colon [p] \mapsto \Delta[p] = \Hom_{\Delta}(-,[p]),
  \]
  i.e. ``collecting all the simplicial sets $\Delta[p]$ for $p\in\mathbb{N}$ together''.
\end{definition}

\subsection{Totalisation and homotopy limits}
\label{subsection:totalisation-and-holim}

One functor of particular interest to us regarding simplicial sets and cosimplicial simplicial sets is the \emph{totalisation} functor.

\begin{definition}
  Let $L\colon\sSet\to\csSet$ be the functor given by $Y_\bullet\mapsto Y_\bullet\times\Delta[\anotherbullet]$.
  We define the \emph{totalisation} functor $\Tot\colon\csSet\to\sSet$ as the right adjoint to $L$.
\end{definition}

We often simply write $(\Tot Y)_p$ instead of $(\Tot Y_\bullet^\anotherbullet)_p$.

This functor is of particular interest in the setting of homotopy theory, as explained by the following technical lemma, which we provide without further context.

\begin{lemma}[{\cite[Theorem~18.7.4]{Hir2003a}}]
\label{lemma:reedy-fibrant-implies-tot-is-holim}
  If $Y_\bullet^\anotherbullet\in\csSet$ is Reedy fibrant, then the totalisation $\Tot Y_\bullet^\anotherbullet$ and the homotopy limit $\holim Y_\bullet^\anotherbullet$ are naturally weakly equivalent.
\end{lemma}

Here the homotopy limit is defined as usual (e.g. as the right derived functor of the right adjoint to the constant-diagram functor), but we will not need to appeal to the technical definition in this paper.

There are many ways (\cite[\S5.3]{Dug2008}) to think of totalisation (e.g. as the dual to geometric realisation), but one particularly useful point-of-view for our purposes is the following.
Given a cosimplicial simplicial set $Y_\bullet^\anotherbullet$, we can show that
\[
  \Tot Y_\bullet^\anotherbullet = \underline{\Hom}_{\csSet}(\Delta[\anotherbullet],Y_\bullet^\anotherbullet).
\]
Morally, this is a version of the tensor-hom adjunction.
Using this, it can be proven that $\Tot$ is also given by an equaliser in $\sSet$:
\[
  \Tot Y_\bullet^\anotherbullet =
  \eq\left(
    \prod_p \Hom_{\sSet}(\Delta[p],Y_\bullet^p) \rightrightarrows \prod_{[p]\to[q]} \Hom_{\sSet}(\Delta[p],Y_\bullet^q)
  \right)
\]
(for details, see \cite[Definition~18.6.3]{Hir2003a}).
With this definition of $\Tot$ as an equaliser, we can show the following:
\emph{a point in $\Tot Y_\bullet^\anotherbullet$ consists of $(y^0,y^1,y^2,\ldots)$, with $y^p\in Y_p^p$, such that}
\begin{enumerate}[(a)]
  \item \emph{the images of $y^p$ under the coface maps $f_{p+1}^i\colon Y_\bullet^p\to Y_\bullet^{p+1}$ are exactly the $p$-dimensional faces of $y^{p+1}$; and}
  \item \emph{the images of $y^p$ under the codegeneracy maps $s_i^{p-1}\colon Y_\bullet^p\to Y_\bullet^{p-1}$ are exactly (up to degeneracy) $y^{p-1}$.}
\end{enumerate}
(cf. \cref{figure:point-in-tot}).

\begin{figure}[ht!]
  \centering
  \begin{tikzpicture}[scale=1.2]
    \begin{scope}
      \draw[thick,dashed] (0,0) ellipse (1.4cm and 1.5cm);
      \node (Y0) at (0,-2.1) {$Y_\bullet^0$};
      \node at (0.3,-0.2) {$\bullet$};
      \node[draw,very thick,inner sep=1mm] (y0) at (0.3,0.2) {\footnotesize$y^0$};
    \end{scope}
    \begin{scope}[shift={(-3.5,0)}]
      \draw[thick,dashed] (0,0) ellipse (1.4cm and 1.5cm);
      \node (Y1) at (0,-2.1) {$Y_\bullet^1$};
      \node[label={[label distance=-2mm]above:{\footnotesize$f_1^0(y^0)$}}] (0y0) at (0.3,0.4) {$\bullet$};
      \node[label={[label distance=-2mm]below:{\footnotesize$f_1^1(y^0)$}}] (1y0) at (-0.1,-0.75) {$\bullet$};
      \draw[very thick] (0y0.mid)
        to node(y1)[draw=black,inner sep=1mm,fill=white]{\footnotesize$y^1$} (1y0.mid);
    \end{scope}
    \begin{scope}[shift={(-7,0)}]
      \draw[thick,dashed] (0,0) ellipse (1.4cm and 1.5cm);
      \node (Y2) at (0,-2.1) {$Y_\bullet^2$};
      \node (v0) at (-0.1,0.8) {$\bullet$};
      \node (v1) at (-0.6,-0.6) {$\bullet$};
      \node (v2) at (0.9,-0.2) {$\bullet$};
      \draw[very thick,hatched] (v0.mid)
        to node(0y1)[label={[label distance=-3mm]above left:{\footnotesize$f_2^0(y_1)$}}]{} (v1.mid)
        to node(1y1)[label={[label distance=-1mm]below:{\footnotesize$f_2^1(y_1)$}}]{} (v2.mid)
        to node(2y1)[label={[label distance=-3mm]above right:{\footnotesize$f_2^2(y_1)$}}]{} cycle;
      \node[very thick,draw=black,inner sep=1mm,fill=white] (y2) at (0.05,0.05) {\footnotesize$y^2$};
    \end{scope}
    \draw[thick,-latex] (Y0.165) to (Y1.15);
    \draw[thick,-latex] (Y0.-165) to (Y1.-15);
    \draw[white] (Y0) to node[black,fill=white]{$f_1^i$} (Y1);
    \draw[thick,-latex] (Y1.155) to (Y2.25);
    \draw[thick,-latex] (Y1.-155) to (Y2.-25);
    \draw[thick,-latex] (Y1) to node[fill=white]{$f_2^i$} (Y2);
  \end{tikzpicture}
  \caption{\emph{Visualising a point $y=(y^0,y^1,y^2,\ldots)$ in the totalisation of a cosimplicial simplicial set $Y_\bullet^\anotherbullet$. For aesthetic purposes, we have not drawn the codegeneracy maps, nor anything above degree~$2$.}}
  \label{figure:point-in-tot}
\end{figure}

More generally, we can show that a $k$-simplex in $\Tot Y$ consists of morphisms $\Delta[k]\times\Delta[p]\to Y_\bullet^p$ for $p\in\mathbb{N}$ such that some analogous conditions hold.
This all follows from the definition of the totalisation as $\underline{\Hom}_{\csSet}(\Delta[\anotherbullet],-)$ along with the description as an equaliser.
For a worked example of $1$-simplices in the totalisation, i.e. for morphisms from $\Delta[1]\times\Delta[p]$, see \cref{appendix:proof-of-1-simplices-in-complex-analytic-twist} (or \cref{appendix:example}, though the situation there is rather more trivial); for the more general case, see \cite[Appendix~B]{GMTZ2022a}.

\subsection{The Čech nerve and the categorical nerve}
\label{subsection:decorations}

\begin{definition}
  Given a topological space $X$ with a cover $\cover=\{U_\alpha\}_{\alpha\in I}$, we define the \emph{Čech nerve} of the pair $(X,\cover)$ to be the simplicial space $(\cechnerve\cover)_\bullet\in[\Delta^\op,\Space]$ whose $p$-simplices are given by the disjoint union of all $p$-fold intersections, i.e.
  \[
    (\cechnerve\cover)_p = \coprod_{\alpha_0,\ldots,\alpha_p\in I} U_{\alpha_0\ldots\alpha_p}
  \]
  and where the face (resp. degeneracy) maps are given by dropping (resp. repeating) indices.
\end{definition}

In particular, any open subset $U_\alpha$ is a $0$-simplex in the Čech nerve, and any intersection $U_{\alpha\beta}$ is a $1$-simplex, and so on.
In other words, a $p$-simplex is not simply the combinatorial data of the \emph{choice} of $p$-fold intersection $(\alpha_0,\ldots,\alpha_p)$, but the actual intersection $U_{\alpha_0\ldots\alpha_p}$ itself.

\begin{definition}
  Given a category $\cal{C}$, we define the \emph{ordinary nerve} (or simply the \emph{nerve}) to be the simplicial set $(\nerve\cal{C})_\bullet$ whose $p$-simplices are sequences of length~$p$ of composable morphisms, i.e.
  \[
    (\nerve\cal{C})_p =
    \Big\{
      x_0\xrightarrow{f_1}x_1\xrightarrow{f_2}\ldots\xrightarrow{f_p}x_p \mid f_i\in\Hom_\cal{C}(x_{i-1},x_i)
    \Big\}_{x_0,\ldots,x_p\in\cal{C}}
  \]
  where, for $p=0$, such a sequence is simply a single object of $\cal{C}$, and where the face maps are given by either composing morphisms (for $1\leq i\leq p$) or deleting them (for $i=0,p$), and the degeneracy maps are given by inserting identity morphisms.
\end{definition}

\begin{remark}
\label{remark:blown-up-nerve}
  Given a $p$-simplex $(f_1,\ldots,f_p)$ in $\nerve\cal{C}$, we can ``fill it out'' by taking the $1$-skeleton of the convex hull of an affinely independent embedding of $p+1$ points (as described in \cref{subsection:simplicial-sets-and-spaces}), labelling the $i$th point with the domain of $f_{i+1}$ (or, equivalently, the codomain of $f_i$), labelling the edge connecting the $(i-1)$th point to the $i$th point with $f_i$, and labelling any remaining edges with the composition of the other two morphisms on the same triangular face, so that all triangles commute.
  That is, we label the \emph{spine} of the standard $p$-simplex, and use the fact that the constituent $1$-simplices are exactly the \emph{generating} simplices.
  For example, given a $2$-simplex
  \[
    x_0 \xrightarrow{f_1} x_1 \xrightarrow{f_2} x_2,
  \]
  we obtain the labelling
  
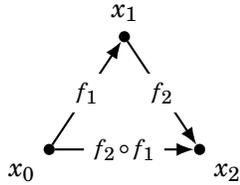
\[
    \begin{tikzpicture}
      \node [vertex,label={below left:{$x_0$}}] (0) at (0,0) {};
      \node [vertex,label={above:{$x_1$}}] (1) at (1,1.5) {};
      \node [vertex,label={below right:{$x_2$}}] (2) at (2,0) {};
      \draw [edge,-Latex] (0) to node[fill=white]{\footnotesize$f_1$} (1);
      \draw [edge,-Latex] (1) to node[fill=white]{\footnotesize$f_2$} (2);
      \draw [edge,-Latex] (0) to node[fill=white]{\footnotesize$f_2\circ f_1$} (2);
    \end{tikzpicture}
  \]
  of (the $1$-skeleton of) $\Delta[2]$.
  By ``filling out'' (or ``blowing up'') the nerve like this, the length-$p$ sequence of morphisms uniquely determines the other $\binom{p}{2}$ by composition, giving $p+\binom{p}{2} = \binom{p}{1}+\binom{p}{2} = \binom{p+1}{2}$ many morphisms in total; see \cref{figure:blowing-up-spine}.
\end{remark}

\begin{figure}[!ht]
  \centering
  \begin{tikzpicture}
    \node at (0,5) {$x_0\xrightarrow{f_1}x_1\xrightarrow{f_2}x_2\xrightarrow{f_3}x_3$};
    \begin{scope}[shift={(0,0)},scale=1.5]
      \node [vertex,label={below left:{\footnotesize$x_0$}}] (0) at (0,0) {};
      \node [vertex,label={above:{\footnotesize$x_1$}}] (1) at (1,1.4) {};
      \node [vertex,label={below:{\footnotesize$x_2$}}] (2) at (1.5,-0.5) {};
      \node [vertex,label={right:{\footnotesize$x_3$}}] (3) at (2,0.5) {};
      \draw [edge,-Latex] (0) to node[fill=white]{\footnotesize$f_1$} (1);
      \draw [edge,-Latex] (1) to node[fill=white]{\footnotesize$f_2$} (2);
      \draw [edge,-Latex] (2) to node[fill=white]{\footnotesize$f_3$} (3);
    \end{scope}
    \begin{scope}[shift={(4,4)},scale=1.5]
      \node [vertex,label={below left:{\footnotesize$x_0$}}] (0) at (0,0) {};
      \node [vertex,label={above:{\footnotesize$x_1$}}] (1) at (1,1.4) {};
      \node [vertex,label={below:{\footnotesize$x_2$}}] (2) at (1.5,-0.5) {};
      \node [vertex,label={right:{\footnotesize$x_3$}}] (3) at (2,0.5) {};
      \draw [edge,dashed,-Latex] (0) to (3);
      \draw [edge,dashed,-Latex] (1) to (3);
      \draw [edge,dashed,-Latex] (0) to (2);
      \draw [edge,-Latex] (0) to node[fill=white]{\footnotesize$f_1$} (1);
      \draw [edge,-Latex] (1) to node [fill=white,near end] {\footnotesize$f_2$} (2);
      \draw [edge,-Latex] (2) to node[fill=white]{\footnotesize$f_3$} (3);
    \end{scope}
    \begin{scope}[shift={(8,0)},scale=2]
      \node [vertex,label={below left:{\footnotesize$x_0$}}] (0) at (0,0) {};
      \node [vertex,label={above:{\footnotesize$x_1$}}] (1) at (1,1.4) {};
      \node [vertex,label={below:{\footnotesize$x_2$}}] (2) at (1.5,-0.5) {};
      \node [vertex,label={right:{\footnotesize$x_3$}}] (3) at (2,0.5) {};
      \draw [edge,-Latex] (0) to node[fill=white]{\footnotesize$f_1$} (1);
      \draw [edge,-Latex] (0) to node[fill=white]{\footnotesize$f_2\circ f_1$} (2);
      \draw [edge,-Latex] (0) to node [fill=white,pos=0.35] {\footnotesize$f_3\circ f_2\circ f_1$} (3);
      \node [white,fill=white] at (1.28,0.34) {$x$};
      \draw [edge,-Latex] (1) to node [fill=white,near end] {\footnotesize$f_2$} (2);
      \draw [edge,-Latex] (1) to node[fill=white]{\footnotesize$f_3\circ f_2$} (3);
      \draw [edge,-Latex] (2) to node[fill=white]{\footnotesize$f_3$} (3);
    \end{scope}
    \draw [line width=0.35mm,->] (-0.5,4.25) to [bend right] node[fill=white]{\footnotesize 1. bend} (0,1.5);
    \draw [line width=0.35mm,->] (3.5,0.25) to [bend right] node[fill=white]{\footnotesize 2. inflate} (5,3);
    \draw [line width=0.35mm,->] (7,5.5) to [bend left] node[fill=white]{\footnotesize 3. fill in} (10,3.5);
  \end{tikzpicture}
  \caption{Given a sequence of three composible morphisms, we can fold them to lie along the spine of the $3$-simplex, and then (uniquely) label the rest of the $1$-skeleton using the compositions.}
  \label{figure:blowing-up-spine}
\end{figure}

\subsection{The dg-nerve and Maurer--Cartan elements}
\label{subsection:dg-nerve}

Throughout this section, and the rest of the paper, whenever we speak of \emph{complexes}, we mean \emph{bounded, negatively graded, cochain complexes}.

\begin{definition}
\label{definition:dg-category}
  A \emph{dg-category} is a category enriched in complexes.
  That is, a category $\cal{D}$ such that the hom-set $\Hom_\cal{D}(x,y)$ is actually a complex for any $x,y\in\cal{D}$, with differential denoted by $\partial$, and such that composition is associative, unital, bilinear, and satisfies the Leibniz rule (for details, see e.g. \cite[\href{https://kerodon.net/tag/00PA}{Tag 00PA}]{kerodon}).

  We often say \emph{dg-category of complexes} to mean a dg-category whose objects are cochain complexes of some objects in an abelian category $\cal{A}$, and whose morphisms are degree-wise morphisms in $\cal{A}$, and whose hom-differential is given by the standard formula.
  More precisely, a morphism $f^\bullet\in\Hom^p(C^\bullet,D^\bullet)$ of degree $p$ consists of morphisms $f^n\colon C^n\to D^{n+p}$ (not necessarily commuting with the differentials $\mathrm{d}_C$ and $\mathrm{d}_D$), and the differential $\partial\colon\Hom^p(C^\bullet,D^\bullet)\to\Hom^{p+1}(C^\bullet,D^\bullet)$ is given by defining $\partial f$ as consisting of the morphisms $(\partial f)^n \coloneqq f^{n+1}\circ\mathrm{d}_C+(-1)^{p+1}\mathrm{d}_D\circ f^n\colon C^n\to D^{n+p+1}$.
\end{definition}

\begin{definition}
\label{definition:dg-nerve}
  Let $\cal{D}$ be a dg-category of chain complexes.
  We define the \emph{dg-nerve} of $\cal{D}$ to be the simplicial set $\dgnerve\cal{D}$ constructed as follows.

  \begin{itemize}
    \item The $0$-simplices of $\dgnerve\cal{D}$ are labellings\footnote{We use this language of \emph{labellings} to be consistent with the constructions later on. It is entirely equivalent, however, to the more standard way of phrasing the definition: ``\emph{$(\dgnerve\cal{D})_1$ consists of triples $(x_0,x_1,f_{\{0<1\}})$, where $x_i\in\cal{D}$ and $f_{\{0<1\}}\colon x_0\to x_1$ is of degree~$0$, such that \ldots}''.} of the standard $0$-simplex by objects of $\cal{D}$, i.e. $(\dgnerve\cal{D})_0$ is in bijection with $\operatorname{Ob}\cal{D}$.
    \item The $1$-simplices of $\dgnerve\cal{D}$ are labellings of the standard $1$-simplex $\{0<1\}$ by morphisms $f_{\{0<1\}}\in\Hom_\cal{D}^0(x_1,x_0)$, where $x_i$ labels the $0$-face $\{i\}\subset\{0<1\}$, such that $\partial f_{\{0<1\}}=0$ (i.e. such that $f_{0<1}$ is a chain map: it commutes with the differentials).
    \item The $2$-simplices of $\dgnerve\cal{D}$ are labellings of the standard $2$-simplex $\{0<1<2\}$ by morphisms $f_{\{0<1<2\}}\in\Hom_\cal{D}^{-1}(x_2,x_0)$ (where $x_i$ labels the $0$-face $\{i\}$, and $f_{i<j}$ labels the $1$-face $\{i<j\}$) such that $\partial f_{\{0<1<2\}}=f_{\{0<2\}}-f_{\{0<1\}}f_{\{1<2\}}$.
    \[
      \begin{tikzcd}[sep=large]
        & {x_1} \\
        {x_0} && {x_2}
        \arrow["{f_{0<1}}"', from=1-2, to=2-1]
        \arrow["{f_{1<2}}"', from=2-3, to=1-2]
        \arrow[""{name=0, anchor=center, inner sep=0}, "{f_{0<2}}", from=2-3, to=2-1]
        \arrow["{f_{0<1<2}}"{description}, shorten <=2pt, Rightarrow, from=1-2, to=0]
      \end{tikzcd}
    \]
    \item Generally, for $p\geq2$, the $p$-simplices of $\dgnerve\cal{D}$ are labellings of every (non-degenerate) face of the standard $p$-simplex; the vertex corresponding to the singleton subset $\{i\}\subset[p]$ is labelled by some object $x_i$ of $\cal{D}$; for $k\geq1$, the $k$-dimensional face corresponding to some $I=\{i_0<i_1<\ldots<i_k\}\subseteq[p]$ is labelled by a morphism $f_I\in\Hom_\cal{D}^{1-k}(x_{i_k},x_{i_0})$; for all non-empty $I\subseteq[p]$ with $|I|-1=k\geq2$, the following relation is satisfied:
      \[
      \label{equation:dg-nerve-definition}
        \partial f_I
        = \sum_{j=1}^{k-1} (-1)^{j-1} f_{I\setminus\{i_j\}}
        + \sum_{j=1}^{k-1} (-1)^{k(j-1)+1} f_{\{i_0<\ldots<i_j\}}\circ f_{\{i_j<\ldots<i_k\}}.
      \tag{\ref*{definition:dg-nerve}.1}
      \]
  \end{itemize}
  The face maps are given by the ``topological'' face maps of simplices: since a $p$-simplex in the dg-nerve is, in particular, a labelling of the $p$-simplex, we obtain face maps by simply looking at the data that labels the faces, which are $(p-1)$-simplices.
  The degeneracy maps are given by inserting identity morphisms for $1$-simplices, and the zero morphism for higher-dimensional simplices.

  For details, see \cite[Definition~2.8, Proposition~2.9, Corollary~2.10]{Fao2017}.
  (N.B. the sign convention differs from that of \cite[Construction~1.3.1.6]{Lur2017} and \cite[\href{https://kerodon.net/tag/00PL}{Tag 00PL}]{kerodon}).
  Note also that the direction of the morphisms is ``backwards'', in that $f_{i_0<\ldots<i_k}$ goes from $x_{i_k}$ to $x_{i_0}$, cf. \cref{remark:morphisms-in-dg-nerve-go-backwards}.
\end{definition}

In the case where $\cal{D}$ is a dg-category of complexes, we could also think of the $0$-simplices as being labelled by morphisms $f_{\{0\}}\in\Hom^1(x_0,x_0)$, which would be exactly the differentials of the complex $x_0$.

\begin{remark}
\label{remark:morphisms-in-dg-nerve-go-backwards}
  In this paper, the morphisms in the dg-nerve go in the ``backwards'' direction, but this is purely a matter of convention: the dg-nerve commutes with $(-)^\op$ up to isomorphism.
  Furthermore, since we are almost exclusively interested in the \emph{core} of the dg-nerve, we can even use the fact that every quasi-groupoid is equivalent to its opposite.
\end{remark}

The following lemma ensures that we can indeed talk of the maximal Kan complex of the dg-nerve of a dg-category.

\begin{lemma}[{\cite[Proposition~1.3.1.10]{Lur2017}}]
  Let $\cal{D}$ be a dg-category.
  Then the simplicial set $\dgnerve\cal{D}$ is a quasi-category.
\end{lemma}

\begin{definition}
\label{definition:1-nerve-of-dg-category}
  Every dg-category $\cal{D}$ has an underlying ``ordinary'' category $K_0\cal{D}$, where
  \[
    \Hom_{K_0\cal{D}}(x,y) =
    \big\{
      f\in\Hom_{\cal{D}}^0(x,y) \mid \partial f=0
    \big\}.
  \]
  For example, when $\cal{D}$ is a dg-category of chain complexes, the morphisms in $K_0\cal{D}$ are exactly chain maps, i.e. those that commute with the differential.

  For notational simplicity, we write $\nerve\cal{D}$ to mean $\nerve(K_0\cal{D})$.
\end{definition}

\begin{lemma}
\label{lemma:nerve-inside-dg-nerve}
  Let $\cal{D}$ be a dg-category of cochain complexes of modules.
  Then, in the notation of \cref{definition:dg-nerve},
  \begin{enumerate}[(i)]
    \item the ordinary nerve $\nerve\cal{D}$ sits inside\footnote{Taking \cref{remark:morphisms-in-dg-nerve-go-backwards} into account, we really mean ``the nerve of the $\cal{D}^\op$''.} the dg-nerve $\dgnerve\cal{D}$ as the simplicial set of labellings with $f_I=0$ for $|I|\geq3$;
    \item the maximal Kan complex $\core{\nerve\cal{D}}$ of the ordinary nerve is given by requiring that the $f_{\{0<1\}}$ be isomorphisms; and
    \item the maximal Kan complex $\core{\dgnerve\cal{D}}$ of the dg-nerve is given by requiring that the $f_{\{0<1\}}$ be quasi-isomorphisms.
  \end{enumerate}
\end{lemma}

\begin{proof}
  \begin{enumerate}[(i)]
    \item (cf. \cite[Remark~1.3.1.9]{Lur2017}).
      This is immediate from \cref{definition:1-nerve-of-dg-category}, since $\nerve\cal{D}\coloneqq\nerve(K_0\cal{D})$ already consists of morphisms $f$ such that $\partial f=0$, so this satisfies the relevant condition in \cref{definition:dg-nerve}.
    \item
      Note that the simplicial set defined by requiring the $f_{\{0<1\}}$ to be isomorphisms is exactly the ordinary nerve of the maximal groupoid $\cal{D}'$ of $\cal{D}$, and thus a Kan complex.
      So let $X_\bullet$ be a Kan complex such that $\nerve\cal{D}'\subseteq X_\bullet\subseteq\nerve\cal{D}$.
      This immediately implies that $X_0=\operatorname{ob}\cal{D}'=\operatorname{ob}\cal{D}$.
      Then, since $X_\bullet$ is Kan, in particular, the outer $2$-horns fill, i.e. for any $f\in X_1$, there exist $g_l,g_r\in X_1$ such that $g_l\circ f=\id=f\circ g_r$, whence $f$ is an isomorphism, since $f^{-1}=g=g_l=g_r$.
      That is, $X_1\subseteq(\nerve\cal{D}')_1$.
      But then, since the nerve of a category is built entirely from $1$-simplices (i.e. is $2$-coskeletal), this implies that $X_p\subseteq(\nerve\cal{D}')_p$ for all $p\in\mathbb{N}$, whence $\nerve\cal{D}'$ is maximal amongst Kan complexes contained inside $\nerve\cal{D}$.
    \item
      This follows from the fact that isomorphisms in $\dgnerve\cal{D}$ correspond to quasi-isomorphisms in $\cal{D}$.
      More precisely, since the dg-nerve is a quasi-category \cite[Proposition~1.3.1.10]{Lur2017}, this follows from \cite[Corollary~1.5]{Joy2002}.
      \qedhere
  \end{enumerate}
\end{proof}

The purpose of the rest of this section is to state \cref{theorem:dg-nerve-and-MC-for-presheaves}, which is a generalisation of some results found in \cite{GMTZ2022a}.
We take the time to restate and reprove the main result \loccit\@ which we are generalising, in a way consistent with the notation used in this current paper, to save the reader the effort of translating from one setting to another.

\begin{lemma}[{\cite[Lemma~2.7]{GMTZ2022a}, \cite[\href{https://kerodon.net/tag/00PV}{Tag 00PV}]{kerodon}}]
\label{lemma:morphism-into-dg-nerve}
  Let $\cal{D}$ be a dg-category, and let $X=X_\bullet$ be a simplicial set.
  Write $\ver_i K$ to mean the $i$th vertex of a $p$-simplex $K\in X_p$.
  Then the following are equivalent:
  \begin{enumerate}[(i)]
    \item a morphism $F\colon X\to\dgnerve(\cal{D})$ of simplicial sets; and
    \item\label{item:data-for-morphism-into-dg-nerve} the data of an object $c_x\in\cal{D}$ for each $0$-simplex $x\in X_0$, along with a morphism $f_K\in \Hom^{1-p}(c_{\ver_pK},c_{\ver_0K})$ for each $p$-simplex $K\in X_p$ for all $p\geq1$, such that the following two conditions hold:
      \begin{enumerate}
        \item for all $K$,
          \[
          \label{equation:top-level-dg-nerve}
            \partial f_K
            = \sum_{j=1}^{p-1} (-1)^{j-1} f_{K\setminus\{\ver_jK\}}
            + \sum_{j=1}^{p-1} (-1)^{p(j-1)+1} f_{\{\ver_0K<\ldots<\ver_jK\}}\circ f_{\{\ver_jK<\ldots<\ver_pK\}}
          \tag{\ref*{lemma:morphism-into-dg-nerve}.1}
          \]
          where the right-hand side is taken to be zero if $p=1$;
        \item if $K$ is degenerate, then
          \[
            f_K =
            \begin{cases}
              \id
              &\mbox{if $p=1$}
            \\0
              &\mbox{if $p\geqslant2$.}
            \end{cases}
          \]
      \end{enumerate}
  \end{enumerate}
\end{lemma}

The key difference between the equation found in the definition of the dg-nerve \cref{equation:dg-nerve-definition} and the above \cref{equation:top-level-dg-nerve} is that the former concerns morphisms labelled by \emph{abstract} simplices, for \emph{all} (non-empty) faces $I\subseteq[p]$, whereas the latter concerns morphisms labelled by \emph{simplices of $X_\bullet$}, and makes \emph{no (direct) reference} to faces/sub-simplices.
The moral of \cref{lemma:morphism-into-dg-nerve}, however, is that these two descriptions give the same result.

\begin{proof}
  Let $F\colon X\to\dgnerve(\cal{D})$ be a morphism of simplicial sets.
  Then, by \cref{definition:dg-nerve}, for any $p\geq1$, any $p$-simplex $K\in X_p$, and any $I=\{i_0<i_1<\ldots<i_k\}\subseteq[p]$ with $1\leq k\leq p$, there exist $f_I\in\Hom^{1-k}(c_{\ver_kI},c_{\ver_0I})$ satisfying \cref{equation:top-level-dg-nerve}.
  In particular then, for $k=p$ (and thus for $\{i_0<i_1<\ldots<i_k\}=[p]$), we have exactly the data given in the statement of the lemma.
  Indeed, the difficulty lies in showing the converse: that having such data \emph{only} for $k=p$ is enough to recover all lower-dimensional data in a functorial way.

  Assume that we have the data of the $c_x$ and the $f_K$, satisfying \cref{equation:top-level-dg-nerve}, as stated in the lemma;
  this gives us a map
  \[
    F_p \colon K \mapsto (\{c_{\ver_i K}\}_{0\leq i\leq p}, \{f_K\})
  \]
  for all $p$-simplices $K\in X_p$, for all $p\geq1$;
  our goal is to \emph{extend} this to a \emph{functorial} map
  \[
    \begin{aligned}
      \overline{F_p}\colon X_p &\to \dgnerve(\cal{D})_p
    \\K &\mapsto (\{c_{\ver_i K}\}_{0\leq i\leq p}, \{f_{|I|}\}_{I\subseteq[p]})
    \end{aligned}
  \]
  (where $|I|$ denotes the face of $K$ defined by $I$), i.e. to extend the singleton set $\{f_K\}$ to a set $\{f_{|I|}\}_{I\subseteq[p]}$ such that all the $f_{|I|}$ satisfy \cref{equation:dg-nerve-definition}, in such a way that $\overline{F_\bullet}$ is functorial.

  So fix $K\in X_p$ and $I=\{i_0<\ldots<i_k\}\subset[p]$ for some $k<p$.
  Let $\sigma\colon[k]\to[p]$ be given by the composition of the coface maps $f^{j_n}$ where $j_n\in[p]\setminus I$, i.e. $\sigma\colon m\mapsto i_m$ for all $0\leq m\leq k$.
  Since $\sigma$ is injective, it induces the morphism
  \[
    \begin{aligned}
      \dgnerve(\sigma)\colon \dgnerve(\cal{D})_p &\to \dgnerve(\cal{D})_k
    \\\left(
        \{x_i\}_{0\leq i\leq p},
        \{f_{|L|}\}_{L\subseteq[p]}
      \right)
      &\mapsto
      \left(
        \{x_{i_m}\}_{0\leq m\leq k},
        \{f_{|\sigma(M)|}\}_{M\subseteq[k]}
      \right)
    \end{aligned}
  \]
  and so we define
  \[
    f_{|I|} \coloneqq \dgnerve(\sigma)(F_p(K))
  \]
  which, by construction, is such that $f_{|I|}=F_k(X_\bullet(\sigma)(K))$.
  That is, we have the commutative diagram
  \[
    \begin{tikzcd}
      X_p \ar[r,"\overline{F_p}"] \ar[d,swap,"X_\bullet(\sigma)"]
      & \dgnerve(\cal{D})_p \ar[d,"\dgnerve(\sigma)"]
    \\X_k \ar[r,swap,"\overline{F_k}"]
      & \dgnerve(\cal{D})_k
    \end{tikzcd}
  \]
  which tells us that the $\overline{F_\bullet}$ thus defined is indeed functorial.
  What remains to be shown is that this definition of $f_{|I|}$ satisfies the equation \cref{equation:dg-nerve-definition} in \cref{definition:dg-nerve}.
  But, by commutativity again,
  \[
    \begin{aligned}
      \dgnerve(\sigma)(F_p(K))
      &= F_k(X_\bullet(\sigma)(K))
    \\&= \left(\{c_{\ver_m|I|}\}_{0\leq m\leq k}, \{f_{|I|}\}\right)
    \end{aligned}
  \]
  which satisfies \cref{equation:top-level-dg-nerve}, and thus (under the identification $|I|\leftrightarrow I$, since $K$ is fixed) \cref{equation:dg-nerve-definition}.

  An analogous argument shows that, if $K\in X_p$ is a \emph{degenerate} simplex, then the morphism $f_K$ given by $F_k(K)$ is the identity if $p=1$, and is the zero morphism if $p\geq2$.
\end{proof}

\begin{corollary}
\label{corollary:morphism-into-dg-nerve-maximal-kan}
  The image of the morphism $F\colon X\to\dgnerve(\cal{D})$ defined by the data of \ref{item:data-for-morphism-into-dg-nerve} in \cref{lemma:morphism-into-dg-nerve} lies in the maximal Kan complex $\core{\dgnerve(\cal{D})}$ if and only if, for all $1$-simplices $K\in X_1$, the morphisms $f_K$ are quasi-isomorphisms.
\end{corollary}

\begin{proof}
  This follows from \cref{lemma:nerve-inside-dg-nerve}.
\end{proof}

\begin{definition}
\label{definition:labelling}
  Let $\cal{D}$ be a dg-category, and let $X_\bullet$ be a simplicial set.
  Then a \emph{labelling of the $0$-simplices of $X_\bullet$ by objects of $\cal{D}$} is a map of sets $\cal{L}\colon X_0\to\cal{D}$.
  Alternatively, we can think of such a labelling as a set $\cal{L}=\{c_x\in\cal{D}\}_{x\in X_0}$
\end{definition}

\begin{definition}
\label{definition:bigraded-dg-algebra}
  Let $\cal{D}$ be a dg-category of chain complexes, and let $X=X_\bullet$ be a simplicial set.
  Fix some labelling $\cal{L}=\{c_x\in\cal{D}\}_{x\in X_0}$ of the $0$-simplices of $X_\bullet$ by objects of $\cal{D}$.

  We define a bigraded dg-algebra $C^{\bullet,\anotherbullet}(X,\cal{D};\cal{L})$ by setting
  \[
    C^{p,q}(X,\cal{D};\cal{L})
    =
    \bigg\{
      \big(
        f_K\in\Hom_\cal{D}^{q}(c_{\ver_{p} K},c_{\ver_{0}K})
      \big)_{K\in X_p}
    \bigg\}
  \]
  for\footnote{The missing $p=0$ term corresponds to the already existing (internal) differential $\partial$ of $\cal{D}$, and the fact that we prescribe the degree-$0$ part separately as the labelling $\cal{L}$.} $p\geq1$ and $q\in\mathbb{Z}$, and define a \emph{(deleted) Čech differential} by
  \[
    \begin{aligned}
      \hat{\delta}\colon C^{p,q}(X,\cal{D};\cal{L})
      &\longrightarrow C^{p+1,q}(X,\cal{D};\cal{L})
    \\(f_K)_{K\in X_p}
      &\longmapsto
      \left(
        \sum_{i=1}^{p} (-1)^i f_{L\setminus\{\ver_i L\}}
      \right)_{L\in X_{p+1}}
    \end{aligned}
  \]
  and an \emph{internal differential} by
  \[
    \begin{aligned}
      \partial\colon C^{p,q}(X,\cal{D};\cal{L})
      &\longrightarrow C^{p,q+1}(X,\cal{D};\cal{L})
    \\(f_K)_{K\in X_p}
      &\longmapsto ((-1)^{q+1}\partial f_K)_{K\in X_p}
    \end{aligned}
  \]
  (where $\partial f_K$ is given by the dg-structure of $\cal{D}$, with sign conventions as in \cref{definition:dg-category}), with graded multiplication
  \[
    \begin{aligned}
      C^{p,q}(X,\cal{D};\cal{L})\times C^{r,s}(X,\cal{D};\cal{L})
      &\longrightarrow C^{p+r,q+s}(X,\cal{D};\cal{L})
    \\\left(
        (f_K)_{K\in X_p},(g_L)_{L\in X_r}
      \right)
      &\longmapsto
      \left(
        (-1)^{qr} f_{\{\ver_0 M<\ldots<\ver_p M\}}\circ g_{\{\ver_p M<\ldots<\ver_{p+r}M\}}
      \right)_{M\in X_{p+r}}.
    \end{aligned}
  \]

  We then define the \emph{total complex} $\Tot^\bullet(C^{p,q}(X,\cal{D};\cal{L}))$ by
  \[
    \Tot^n(C^{p,q}(X,\cal{D};\cal{L})) = \bigoplus_{p+q=n}C^{p,q}(X,\cal{D};\cal{L})
  \]
  with \emph{total differential}
  \[
    \operatorname{D}\colon
    \Tot^n(C^{p,q}(X,\cal{D};\cal{L}))
    \longrightarrow \Tot^{n+1}(C^{p,q}(X,\cal{D};\cal{L}))
  \]
  defined by $\operatorname{D}=\hat{\delta}+(-1)^p\partial$.
\end{definition}

In words,
\begin{itemize}
  \item the degree-$(p,q)$ elements of $C^{\bullet,\anotherbullet}(X,\cal{D};\cal{L})$ are labellings of the $p$-simplices of $X_\bullet$ by morphisms $f\colon c_x\to c_y[-q]$ of $\cal{D}$, where $x$ is the first vertex of the $p$-simplex and $y$ is the last;
  \item the differential $\hat{\delta}$ of degree~$(1,0)$ is given by taking the alternating sum of the morphisms labelling the $p$-simplices given by removing the $i$th vertex for $i=1,2,\ldots,k$ (but \emph{not} for $i=0$ or $i=k+1$);
  \item the differential $\partial$ of degree~$(0,1)$ is given by simply applying the differential of $\cal{D}$ coming from its dg-structure, with a global sign depending on the degree of $f$;
  \item the graded multiplication takes an element $f=(f_K)$ of degree~$(p,q)$ and an element $g=(g_L)$ of degree~$(r,s)$ and gives an element $f\cdot g$ of degree~$(p+r,q+s)$ by labelling the $(p+r)$-simplices as follows: split each $(p+r)$-simplex into a \emph{front half} (given by taking the first $p+1$ vertices) and a \emph{back half} (given by taking the last $r+1$ vertices), and then label the front half with $f$ and the back half with $g$ (with some global sign depending on the degrees of $f$ and $g$).
\end{itemize}

\begin{definition}
  An element $f\in\Tot^\bullet(C^{p,q}(X,\cal{D};\cal{L}))$ is said to be \emph{Maurer--Cartan} if it satisfies the Maurer--Cartan equation:
  \[
    \operatorname{D}f+f\cdot f = 0.
  \]
  Note that, for degree reasons, all Maurer--Cartan elements are of the form
  \[
    f = \big(f_p\in C^{p,1-p}(X,\cal{D};\cal{L})\big)_{p\geq1}
  \]
  and so $\deg f=1$.
  That is, the Maurer--Cartan elements of $\Tot^\bullet(C^{p,q}(X,\cal{D};\cal{L}))$ are exactly the Maurer--Cartan elements of $\Tot^1(C^{p,q}(X,\cal{D};\cal{L}))$.
\end{definition}

\begin{theorem}[{\cite[Corollary~3.5]{GMTZ2022a}}]
\label{theorem:dg-nerve-and-MC}
  Let $\cal{D}$ be a dg-category of cochain complexes of modules, let $X=X_\bullet$ be a simplicial set, and let $\cal{L}=\{c_x\in\cal{D}\}_{x\in X_0}$ be a labelling of the $0$-simplices of $X_\bullet$ by $\cal{D}$.
  Then there is a bijection
  \[
    \Big\{
      f\in \Tot^1(C^{p,q}(X,\cal{D};\cal{L})) \mid \operatorname{D}f+f\cdot f=0
    \Big\}
    \longleftrightarrow
    \Big\{
      F\colon X\to\dgnerve(\cal{D}) \mid \mbox{$F(x) = c_x$ for all $x\in X_0$}
    \Big\}
  \]
  between Maurer--Cartan elements of $C^{\bullet,\anotherbullet}(X,\cal{D};\cal{L})$ and morphisms of simplicial sets from $X_\bullet$ to dg-nerve of $\cal{D}$ that agree with the labelling $\cal{L}$.
\end{theorem}

\begin{proof}
  By \cref{lemma:morphism-into-dg-nerve}, we know that an element of the set on the right-hand side (i.e. a morphism $F\colon X\to\dgnerve(\cal{D})$ such that $F(x)=c_x$ for all $x\in X_0$) is equivalent to the data of a morphism $f_K\in\Hom^{1-p}(c_{\ver_pK},c_{\ver_0k})$ for each $p$-simplex $K\in X_p$, for all $p\geq1$, such that \cref{equation:top-level-dg-nerve} holds (unless $K$ is degenerate, in which case $f_K$ is the identity when $p=0$ and the zero morphism when $p\geq1$).
  The collection of all these $f_K$ is then an element of the bigraded dg-algebra:
  \[
    f_p = \big( f_K\in\Hom^{1-p}(c_{\ver_pK},c_{\ver_0k}) \big)_{K\in X_p}
    \in C^{p,1-p}(X,\cal{D};\cal{L}).
  \]
  Now, by \cref{definition:bigraded-dg-algebra}, for $p,r\geq1$
  \[
    \begin{aligned}
      f_p\cdot f_r
      &=
      \big(
        (-1)^{(1-p)r} f_{\{\ver_0M<\ldots<\ver_pM\}}\circ f_{\{\ver_pM<\ldots<\ver_{p+r+1}M\}}
      \big)_{M\in X_{p+r}}
    \end{aligned}
  \]
  and
  \[
    \begin{aligned}
      \operatorname{D}(f_p)
      &= \hat{\delta}f_p+(-1)^p\partial f_p
    \\&=
      \left(
        \sum_{i=1}^p(-1)^i f_{L\setminus\{\ver_iL\}}
      \right)_{L\in X_{p+1}}
      +
      \big(
        \underbrace{(-1)^p(-1)^{(1-p)+1}}_{=1}\partial f_K
      \big)_{K\in X_p}
    \end{aligned}
  \]
  whence, for $\lambda\geq2$,
  \[
    \begin{aligned}
      \big(\operatorname{D}f+f\cdot f\big)_\lambda
      =
      \Bigg(
        &\sum_{j=1}^{\lambda-1}(-1)^j f_{M\setminus\{\ver_jM\}}
      \\+ &\partial f_M
      \\+ &\sum_{j=1}^{\lambda-1}(-1)^{(1-j)(\lambda-j)} f_{\{\ver_0M<\ldots<\ver_jM\}}\circ f_{\{\ver_jM<\ldots<\ver_{\lambda+1}M\}}
      \Bigg)_{M\in X_\lambda}
    \end{aligned}
  \]
  but \cref{equation:top-level-dg-nerve} says that
  \[
    \partial f_M
    = -\left(
      \sum_{j=1}^{\lambda-1}(-1)^{j}f_{M\setminus\{\ver_jM\}}
      + \sum_{j=1}^{\lambda-1}(-1)^{\lambda(j-1)} f_{\{\ver_0M<\ldots<\ver_jM\}}\circ f_{\{\ver_jM<\ldots<\ver_{\lambda+1}M\}}
    \right)
  \]
  and so it suffices to show that
  \[
    (-1)^{\lambda(j-1)} = (-1)^{(1-j)(\lambda-j)}
  \]
  but $(1-j)\equiv(j-1)\mod2$, and $j(j-1)\equiv0\mod2$.
  This means that $(\operatorname{D}f+f\cdot f)_\lambda=0$ for all $\lambda\geq2$, and so $\operatorname{D}f+f\cdot f=0$.

  Conversely, if we start with some Maurer--Cartan element $f=(f_p)$, then, by the exact same argument, the fact that $\operatorname{D}f+f\cdot f=0$ is satisfied implies that the collection of the $f_K$ that define the $f_p$ also satisfy \cref{equation:top-level-dg-nerve}.
\end{proof}

A fundamental example of \cref{theorem:dg-nerve-and-MC} is given by taking $X_\bullet$ to be the prototypical simplicial set $\Delta[n]=\Hom_{\sSet}(-,[n])$ for some fixed $n\in\mathbb{N}$.
By the Yoneda lemma,
\[
  \Hom_{\sSet}\big(\Delta[n],\dgnerve(\cal{D})\big)
  \cong \dgnerve(\cal{D})_n
\]
whence the following corollary.

\begin{corollary}
\label{corollary:dg-nerve-as-MC-things}
  With the notation and hypotheses of \cref{theorem:dg-nerve-and-MC}, we have a bijection
  \[
    \Big\{
      \mbox{Maurer--Cartan elements of $\Tot^1\big(C^{p,q}(\Delta[n],\cal{D};K)\big)$}
    \Big\}
    \longleftrightarrow
    \Big\{
      \mbox{$n$-simplices $K\in\dgnerve(\cal{D})_n$}
    \Big\}
  \]
  where the $n$-simplex $K$ defines the labelling $\{i\}\mapsto\ver_iK$.
\end{corollary}

For our purposes, we need a version of \cref{definition:bigraded-dg-algebra} that is both generalised and specialised: generalising a single dg-category to a \emph{presheaf} of dg-categories, but specialising to the specific example where the simplicial set is the Čech nerve.

\begin{definition}
\label{definition:bigraded-dg-algebra-for-presheaves}
  Let $\scr{D}$ be a presheaf of dg-categories on the category of spaces, and $\cechnerve\cover_\bullet$ the Čech nerve of the cover $\cover$ of some space $X$.
  Fix some labelling\footnote{Note that this usage of ``labelling'' is slightly more general than \cref{definition:labelling}, since each label $c_\alpha$ lives in a different dg-category $\scr{D}(U_\alpha)$.} \mbox{$\cal{L}=\{c_\alpha\in\scr{D}(U_\alpha)\}_{U_\alpha\in\cover}$} of $\cechnerve\cover_0$ by $\scr{D}$.

  We define a bigraded dg-algebra $\delcech^{\bullet,\anotherbullet}(\cover,\scr{D};\cal{L})$, which we call the \emph{(deleted) Čech (bi)algebra}, by setting
  \[
    \delcech^{p,q}(\cover,\scr{D};\cal{L})
    = \left\{
      \big(
        f_{\alpha_0\ldots\alpha_p}
        \in \Hom_{\scr{D}(U_{\alpha_0\ldots\alpha_p})}^q(c_{\alpha_p}|U_{\alpha_0\ldots\alpha_p},c_{\alpha_0}|U_{\alpha_0\ldots\alpha_p})
      \big)_{U_{\alpha_0\ldots\alpha_p}\in\cechnerve\cover_p}
    \right\}
  \]
  for $p\geq1$ and $q\in\mathbb{Z}$.
  We then define the deleted Čech differential, internal differential, graded multiplication, and total differential entirely analogously to \cref{definition:bigraded-dg-algebra}, so that e.g. the graded multiplication is given by
  \[
    \begin{aligned}
      \delcech^{p,q}(\cover,\scr{D};\cal{L})\times\delcech^{r,s}(\cover,\scr{D};\cal{L})
      &\longrightarrow \delcech^{p+r,q+s}(\cover,\scr{D};\cal{L})
    \\\Big((f_{\alpha_0\ldots\alpha_p})_{U_{\alpha_0\ldots\alpha_p}},(g_{\beta_0\ldots\beta_r})_{U_{\beta_0\ldots\beta_r}}\Big)
      &\longmapsto \Big((-1)^{qr}f_{\gamma_0\ldots\gamma_p}\circ g_{\gamma_{p+1}\ldots\gamma_{p+r}}\Big)_{U_{\gamma_0\ldots\gamma_{p+r}}}
    \end{aligned}
  \]
  and the deleted Čech differential is given by
  \[
    \begin{aligned}
      \hat{\delta}\colon \delcech^{p,q}(\cover,\scr{D};\cal{L})
      &\to \delcech^{p+1,q}(\cover,\scr{D};\cal{L})
    \\(f_{\alpha_0\ldots\alpha_p})_{U_{\alpha_0\ldots\alpha_p}}
      &\mapsto
      \left(
        \sum_{i=1}^{p} (-1)^i f_{\alpha_0\ldots\widehat{\alpha_i}\ldots\alpha_{p+1}}
      \right)_{U_{\alpha_0\ldots\alpha_{p+1}}}
    \end{aligned}
  \]
  where, as per usual, the hat denotes omission.
  Note that this is well defined since the deleted Čech differential preserves the first and last vertices of the simplex, so that $f_{\alpha_0\ldots\widehat\alpha_i\ldots\alpha_p}$ is still a morphism from $c_{\alpha_p}$ to $c_{\alpha_0}$ for all $0<i<p$.
\end{definition}

\begin{remark}
  The very definition of the set $\delcech^{p,q}(\cover,\scr{D};\cal{L})$ in \cref{definition:bigraded-dg-algebra-for-presheaves} relies upon the fact that the Čech nerve gives us restriction maps $\mathscr{D}(U_{\alpha_i})\to\mathscr{D}(U_{\alpha_0\ldots\alpha_p})$ induced by $U_{\alpha_0\ldots\alpha_p}\hookrightarrow U_{\alpha_i}$.
  It would be possible to give the definition for arbitrary simplicial sets possessing this property, but we are only ever interested in the Čech nerve in this paper.
\end{remark}

Something that will turn up in \cref{subsection:abstract-machinery-for-example} is the idea of pulling back a simplicial presheaf along the opposite of the Čech nerve.
There we will better motivate and justify the importance of this construction, but for now we content ourselves with its definition.

\begin{definition}
\label{definition:evaluating-on-the-cech-nerve}
  The opposite of the Čech nerve $\cechnerve^\op\colon\Space_\cover^\op\to[\Delta,\Space^\op]$ is a functor from the category of spaces with a chosen cover to that of cosimplicial spaces, and so we can pre-compose any simplicial presheaf $\scr{F}\colon\Space^\op\to\sSet$ with this to obtain a cosimplicial simplicial set $\scr{F}(\cechnerve\cover_\anotherbullet)$ whenever we evaluate on any given space $X$ with cover $\cover$.
  We call this process \emph{evaluating $\scr{F}$ on the Čech nerve of $\cover$}.
\end{definition}

Note also that $\dgnerve\scr{D}(\cechnerve\cover_0)$ is a simplicial set concentrated in dimension~$0$, with $\dgnerve\scr{D}(\cechnerve\cover_0)_0\cong\scr{D}(\cechnerve\cover_0)$ (as sets), by definition of the dg-nerve;
the latter is exactly $\scr{D}(\coprod_\alpha U_\alpha)$ by definition of the Čech nerve.
For our purposes, we will be interested in presheaves of dg-categories such that $\scr{D}(\coprod_\alpha U_\alpha)\cong\prod_\alpha\scr{D}(U_\alpha)$.
Using this, we can prove a theorem that is to \cref{definition:bigraded-dg-algebra-for-presheaves} what \cref{corollary:dg-nerve-as-MC-things} is to \cref{definition:bigraded-dg-algebra}.

\begin{theorem}
\label{theorem:dg-nerve-and-MC-for-presheaves}
  Let $\scr{D}$ be a presheaf of dg-categories of cochain complexes of modules on the category of spaces, let $\cechnerve\cover_\bullet$ be the Čech nerve of the cover $\cover$ of some space $X$, and let $\cal{L}=\{c_\alpha\in\scr{D}(U_\alpha)\}_{U_\alpha\in\cover}$ be a labelling of $\cechnerve\cover_0$ by $\scr{D}$.
  Assume that $\scr{D}$ turns disjoint unions into products, i.e. that $\scr{D}(\sqcup_\alpha U_\alpha)\cong\prod_\alpha\scr{D}(U_\alpha)$ is a bijective-on-objects equivalence.\footnote{The one specific example of $\scr{D}$ that we are interested in for the purposes of this current paper is that which sends a ringed space $(X,\OO_X)$ to the dg-category of bounded complexes of free $\OO_X$-modules on $X$. For this choice of $\scr{D}$, and in the case where the ringed spaces have representable structure sheaves, it is indeed the case that $\scr{D}(\coprod_\alpha U_\alpha)\cong\prod_\alpha\scr{D}(U_\alpha)$ is an \emph{isomorphism} of categories. However, ``bijective-on-objects equivalence'' seems to be preferable terminology.}
  Then there is a bijection
  \[
    \Big\{
      f \in \Tot^1(\delcech^{p,q}(\cover,\scr{D};\cal{L}))
      \mid
      \operatorname{D}f+f\cdot f=0
    \Big\}
    \longleftrightarrow
    \Big\{
      F\colon\Delta[\anotherbullet]\to\dgnerve\scr{D}(\cechnerve\cover_\anotherbullet)
      \mid
      F(\Delta[0]) = \cal{L}
    \Big\}
  \]
  between Maurer--Cartan elements of the Čech algebra $\delcech^{\bullet,\anotherbullet}(\cover,\scr{D};\cal{L})$ and morphisms of cosimplicial simplicial sets from $\Delta[\anotherbullet]$ to $\dgnerve\scr{D}(\cechnerve\cover_\anotherbullet)$ that send the unique non-degenerate \mbox{$0$-simplex} in $\Delta[0]$ to the element $(c_\alpha)_{U_\alpha\in\cover}\in\prod_\alpha\scr{D}(U_\alpha)\simeq\dgnerve\scr{D}(\cechnerve\cover_0)$.
\end{theorem}

The proof of this theorem is given below, but the reader might find it helpful to also read the proof of \cref{theorem:tot0-twist-is-twcs} (which is indeed the intended application for this result), where the specific case of degree~$2$ is spelled out in more detail.

The main idea of the proof is deceptively simple.
If we fix some cosimplicial degree $q\in\mathbb{N}$ then a morphism $F$ of cosimplicial simplicial sets simply becomes a morphism of simplicial sets, and we can apply \cref{corollary:dg-nerve-as-MC-things} to obtain a bijection with Maurer--Cartan elements in the bialgebra from \cref{definition:bigraded-dg-algebra}.
Then, using the fact that the Čech nerve gives us restriction maps, and that a morphism of cosimplicial simplicial sets is, in particular, functorial with respect to the cosimplicial structure, we can ``glue together'' all these Maurer--Cartan elements for each $q\in\mathbb{N}$ to obtain the desired result.
We now explain this in detail.

\begin{proof}
  Let $F\colon\Delta[\anotherbullet]\to\dgnerve\scr{D}(\cechnerve\cover_\anotherbullet)$ be a morphism of cosimplicial simplicial sets that sends the unique non-degenerate $0$-simplex in $\Delta[0]$ to the element $(c_\alpha)_{U_\alpha\in\cover}$.
  For each $q\in\mathbb{N}$, we thus have a morphism
  \[
    F^q\colon\Delta[q]\to\dgnerve\scr{D}(\cechnerve\cover_q)
  \]
  of simplicial sets, but, by the Yoneda lemma, this is exactly the data of a $q$-simplex of $\dgnerve\scr{D}(\cechnerve\cover_q)$, which we will also denote by $F^q$.
  But since $\scr{D}$ turns disjoint unions into products, and since the dg-nerve is a right adjoint and thus preserves products, such a $q$-simplex is exactly the data of a $q$-simplex of $\dgnerve\scr{D}(U_{\alpha_0\ldots\alpha_q})$ for all $U_{\alpha_0\ldots\alpha_q}$; we denote this by
  \[
    F^q=(F_{\alpha_0\ldots\alpha_q}^q)_{U_{\alpha_0\ldots\alpha_q}\in\cechnerve\cover_q}.
  \]

  By \cref{corollary:dg-nerve-as-MC-things}, to each $F_{\alpha_0\ldots\alpha_q}^q$ there corresponds a Maurer--Cartan element in $\Tot^1 C^{m,n}(\Delta[q],\scr{D}(U_{\alpha_0\ldots\alpha_q});F_{\alpha_0\ldots\alpha_q}^q)$.
  If we denote the $0$-simplices of some $F_{\alpha_0\ldots\alpha_q}^q$ by $x_0,\ldots,x_q$ then the corresponding Maurer--Cartan element $\varphi=(\varphi^p)_{p\geq1}$ has components
  \[
    \big(\varphi_L^p\in\Hom_{\scr{D}(U_{\alpha_0\ldots\alpha_q})}^{1-p}(x_{\ver_{p}L},x_{\ver_0 L})\big)_{L\in\Delta[q]_p}
  \]
  for each $p\geq1$.
  However, since the morphisms corresponding to degenerate simplices are zero (cf. the proof of \cref{lemma:morphism-into-dg-nerve}) we can rewrite this as
  \[
    \varphi
    = \big(\varphi_K\in\Hom_{\scr{D}(U_{\alpha_0\ldots\alpha_q})}^{1-k}(x_{i_k},x_{i_0})\big)_{K\subseteq[q]}
  \]
  where $k=|K|-1\geq1$.

  The hypothesis on the morphism $F$ of cosimplicial simplicial sets tells us that the image of $\{0\}\in\Delta[0]$ is exactly $(c_\alpha)_{U_\alpha\in\cover}$; the fact that $F$ is a morphism tells us, in particular, that it is functorial with respect to the cosimplicial structure, and so the $0$-simplices of all the $F_{\alpha_0\ldots\alpha_q}^q$ are determined entirely by this data, and are given by $x_i=c_{\alpha_i}|U_{\alpha_0\ldots\alpha_q}$.
  More generally, for any $p<q$, the $p$-simplices in $F^q$ are exactly the restrictions of the corresponding $p$-simplices in $F^p$.
  This tells us that the functorial collection of elements in $\Tot^1 C^{m,n}(\Delta[q],\scr{D}(U_{\alpha_0\ldots\alpha_q});F^q)$ for all $q\in\mathbb{N}$ is exactly the same as an element in $\Tot^1\delcech^{m,n}(\cover,\scr{D};\cal{L})$; furthermore, under this correspondence, the two definitions of $\hat{\delta}$ agree (as do the two definitions of $\partial$, though this is more immediate), which means that if the functorial collection of elements all satisfy the Maurer--Cartan condition, then so too does the resulting element in the Čech bialgebra.
\end{proof}

We will use \cref{theorem:dg-nerve-and-MC-for-presheaves} to prove that the points in the totalisation of a certain cosimplicial simplicial presheaf are exactly Maurer--Cartan elements in a well-known Čech algebra (\cref{theorem:tot0-twist-is-twcs}), but we are also interested in the \emph{paths} of this space.
To study these, we need to understand the result analogous to \cref{theorem:dg-nerve-and-MC-for-presheaves} but where we replace morphisms $F\colon\Delta[\anotherbullet]\to\ldots$ by morphisms $F\colon\Delta[\anotherbullet]\times\Delta[1]\to\ldots$, and it turns out that such morphisms correspond to \emph{closed} elements in some relevant bialgebra.
Although one could give a general statement, analogous to \cref{theorem:dg-nerve-and-MC-for-presheaves}, we consider only the specific application that is of interest to us: this is the content of \cref{theorem:1-simplices-in-complex-analytic-twist}, and the explanation of how this relates to closed elements in a bialgebra is given in \cref{appendix:proof-of-1-simplices-in-complex-analytic-twist}.

\subsection{The pair subdivision}
\label{subsection:the-barycentric-and-pair-subdivision}

We will now briefly discuss the \emph{barycentric subdivision} of a simplex, only for the purpose of contrasting it with the \emph{pair subdivision}.

Similar to how non-empty ordered subsets of $[p]$ are in bijective correspondence with sub-simplices of $\Delta[p]$, we can describe the $k$-simplices of the barycentric subdivision of $\Delta[p]$ in a combinatorial way.
We write $\bary{\Delta[p]}$ to mean the \emph{barycentric subdivision} of $\Delta[p]$, which we now define as a simplicial set (see also \cref{figure:points-in-bcs}).

\begin{itemize}
  \item The $0$-simplices of $\bary{\Delta[p]}$ correspond exactly to the $k$-simplices of $\Delta[p]$ for $k\leq p$, i.e. the (non-degenerate) faces of $\Delta[p]$.
    But, as we have already said, these are in bijection with non-empty ordered subsets of $[p]=\{0,1,\ldots,p\}$, and there are $2^{p+1}-1$ of these.
  \item The $1$-simplices of $\bary{\Delta[p]}$ correspond exactly to a choice of a $k$-simplex $\sigma$ of $\Delta[p]$ along with an $\ell$-simplex $\tau$ such that $\tau\subset\sigma$, and these are in bijective correspondence with pairs $(S,T)$ of non-empty subsets of $[p]$ such that $T\subset S$, of which there are $\sum_{k=1}^p\sum_{\ell=0}^{k-1}\binom{p+1}{k+1}\binom{k+1}{\ell+1}$.
  \item More generally, the $q$-simplices of $\bary{\Delta[p]}$ correspond exactly to a choice of $k_j$-simplex $\sigma_j$ of $\Delta_p$, for $j=0,\ldots,q$, such that $\sigma_q\subset\sigma_{q-1}\subset\ldots\subset\sigma_0$;
    these are in bijective correspondence with tuples $(S_q,\ldots,S_0)$ of non-empty subsets of $[p]$ such that $S_0\subset S_1\subset\ldots\subset S_q$.
\end{itemize}

\begin{figure}[!ht]
  \centering
  \begin{tabular}{c|c|c}
    $0$-simplex of $\bary{\Delta[2]}$
    & subset of $[2]=\{0<1<2\}$
    & sub-simplex of $\Delta[2]$
  \\\midrule&&
  \\\begin{tikzpicture}[scale=0.7,baseline={(0,0.5)}]
      \node [vertex] (0) at (0,0) {};
      \node [vertex=white] (1) at (1,1.5) {};
      \node [vertex=white] (2) at (2,0) {};
      \draw [edge,dashed] (0) to node (01) [vertex=white] {} (1)
        to node (12) [vertex=white] {} (2)
        to node (02) [vertex=white] {} (0);
      \draw [edge,dashed] (0) to (12);
      \draw [edge,dashed] (1) to (02);
      \draw [edge,dashed] (2) to (01);
      \node [vertex=white] (c) at (1,0.5) {};
    \end{tikzpicture}
    &$\{0\}\subseteq[2]$
    & \begin{tikzpicture}[scale=0.7,baseline={(0,0.5)}]
      \node [vertex] (0) at (0,0) {};
      \node [vertex=white] (1) at (1,1.5) {};
      \node [vertex=white] (2) at (2,0) {};
      \draw [edge,dashed] (0) to (1) to (2) to (0);
    \end{tikzpicture}
  \\&&
  \\\begin{tikzpicture}[scale=0.7,baseline={(0,0.5)}]
      \node [vertex=white] (0) at (0,0) {};
      \node [vertex=white] (1) at (1,1.5) {};
      \node [vertex=white] (2) at (2,0) {};
      \draw [edge,dashed] (0) to node (01) [vertex=white] {} (1)
        to node (12) [vertex,solid] {} (2)
        to node (02) [vertex=white] {} (0);
      \draw [edge,dashed] (0) to (12);
      \draw [edge,dashed] (1) to (02);
      \draw [edge,dashed] (2) to (01);
      \node [vertex=white] (c) at (1,0.5) {};
    \end{tikzpicture}
    &$\{1<2\}\subseteq[2]$
    & \begin{tikzpicture}[scale=0.7,baseline={(0,0.5)}]
      \node [vertex=white] (0) at (0,0) {};
      \node [vertex] (1) at (1,1.5) {};
      \node [vertex] (2) at (2,0) {};
      \draw [edge,dashed] (0) to (1) to (2) to (0);
      \draw [edge] (1) to (2);
    \end{tikzpicture}
  \\&&
  \\\begin{tikzpicture}[scale=0.7,baseline={(0,0.5)}]
      \node [vertex=white] (0) at (0,0) {};
      \node [vertex=white] (1) at (1,1.5) {};
      \node [vertex=white] (2) at (2,0) {};
      \draw [edge,dashed] (0) to node (01) [vertex=white] {} (1)
        to node (12) [vertex=white] {} (2)
        to node (02) [vertex=white] {} (0);
      \draw [edge,dashed] (0) to (12);
      \draw [edge,dashed] (1) to (02);
      \draw [edge,dashed] (2) to (01);
      \node [vertex] (c) at (1,0.5) {};
    \end{tikzpicture}
    &$\{0<1<2\}\subseteq[2]$
    & \begin{tikzpicture}[scale=0.7,baseline={(0,0.5)}]
      \node [vertex] (0) at (0,0) {};
      \node [vertex] (1) at (1,1.5) {};
      \node [vertex] (2) at (2,0) {};
      \draw [edge,hatched] (0.center) to (1.center) to (2.center) to cycle;
    \end{tikzpicture}
    \\&&
  \end{tabular}
  \caption{Points in the barycentric subdivision correspond to sub-simplices of the abstract simplex.}
  \label{figure:points-in-bcs}
\end{figure}

For our purposes, one defect of the barycentric subdivision is the fact that it doesn't take the codimension of inclusions into account.
That is, both the inclusion $\{0\}\subset[2]$ of a \emph{point} into the $2$-simplex and the inclusion $\{0<1\}\subset[2]$ of a \emph{line} into the $2$-simplex correspond to $0$-simplices of the barycentric subdivision, even though the former is of codimension~$2$ and the latter of codimension~$1$.
What we would like is a method of subdivision where codimension-$k$ inclusions correspond to $k$-simplices, and where we forget about the length of the flags and look only at pairs $\tau\subset\sigma$.
It turns out that such a subdivision exists.

\begin{definition}
\label{definition:pair-subdivision}
  Given the standard $p$-simplex $\Delta[p]$, we define the \emph{pair subdivision} $\pair{\Delta[p]}$ as follows (see \cite[\S2]{Zee1963} and \cite[\S3.2]{Rou2010} for more details; see also \cref{figure:cells-in-pair-subdivision}).
  The vertices are the original vertices of $\Delta[p]$ along with the barycentres of each face (just as in the barycentric subdivision), and these are labelled by pairs $(\sigma,\sigma)$, where $\sigma\subset\Delta[p]$ is exactly the face in question.
  In general, the $k$-cells are given by pairs $(\tau,\sigma)$, where $\tau\subseteq\sigma\subseteq\Delta[p]$ and $k=\codim_\sigma\tau\coloneqq\dim\sigma-\dim\tau$;
  the vertices of such a $k$-cell are the barycentres of the simplices $\eta$ such that $\tau\subseteq\eta\subseteq\sigma$;
  the set of codimension-$\ell$ faces of $(\tau,\sigma)$ is the union of the set of cells of the form $(\underline{\tau},\sigma)$, where $\underline{\tau}\subset\tau$ is a codimension-$\ell$ face, and the set of cells of the form $(\tau,\overline{\sigma})$, where $\sigma\subset\overline{\sigma}$ is a codimension-$\ell$ face.

  We define the boundary of a pair $(\tau,\sigma)$ by
  \[
    \partial(\tau,\sigma) \coloneqq
    (\tau,\partial\sigma) + (-1)^{\dim\tau}(d\tau,\sigma)
  \]
  where $d$ is the coboundary operator, sending an $\ell$-simplex to the (signed) sum of all $(\ell+1)$-simplices that contain it as a face.
\end{definition}

For example, given the pair $(\bullet,\blacktriangle)$, where $\Delta[0]=\bullet=\{0\}\hookrightarrow\{0<1<2\}=\blacktriangle=\Delta[2]$, we can calculate the boundary:
\[
  \partial(\bullet,\blacktriangle)
  \,\,\,\,=\,\,\,\,
  \underbrace{-(\bullet,\,\raisebox{-0.1em}{\rotatebox{90}{$\vert$}}\,\,)+(\bullet,\,\raisebox{0.15em}{\rotatebox{-30}{$\vert$}}\,)}_{(\bullet,\partial\blacktriangle)}
  \,\,\,\,+\,\,\,\,
  \underbrace{-(\,\,\raisebox{-0.1em}{\rotatebox{90}{$\vert$}}\,,\blacktriangle)+(\,\raisebox{0.15em}{\rotatebox{-30}{$\vert$}}\,,\blacktriangle)}_{(d\bullet,\blacktriangle)}
  \,\,\,\,=\,\,\,\,
  \begin{tikzpicture}[baseline={(0,0.6)}]
    \node [vertex] (0) at (0,0) {};
    \node [vertex=white] (1) at (1,1.5) {};
    \node [vertex=white] (2) at (2,0) {};
    \draw [edge,dashed] (0) to node (01) [vertex] {} (1)
      to node (12) [vertex=white] {} (2)
      to node (02) [vertex] {} (0);
    \node [vertex] (c) at (1,0.5) {};
    \draw [edge,dashed] (01) to (c);
    \draw [edge,dashed] (12) to (c);
    \draw [edge,dashed] (02) to (c);
    \draw [thick,-stealth] (0) to (01);
    \draw [thick,-stealth] (01) to (c);
    \draw [thick,-stealth] (0) to (02);
    \draw [thick,-stealth] (02) to (c);
  \end{tikzpicture}
\]
So we see that the boundary of the $2$-cell corresponding to $(\bullet,\blacktriangle)$ consists of four $1$-cells, and is thus a \emph{$2$-cube}.

\begin{figure}[!ht]
  \centering
  \begin{tabular}{c|c|c}
    cells in $\pair{\Delta[2]}$
    & pairs of subsets of $[2]=\{0<1<2\}$
    & sub-simplex of $\Delta[p]$
  \\\midrule&&
  \\\begin{tikzpicture}[scale=0.7,baseline={(0,0.5)}]
      \node [vertex] (0) at (0,0) {};
      \node [vertex=white] (1) at (1,1.5) {};
      \node [vertex=white] (2) at (2,0) {};
      \draw [edge,dashed] (0) to node (01) [vertex] {} (1)
        to node (12) [vertex=white] {} (2)
        to node (02) [vertex] {} (0);
      \node [vertex] (c) at (1,0.5) {};
      \draw [edge,hatched] (0.center) to (01.center) to (c.center) to (02.center) to cycle;
    \end{tikzpicture}
    &$\{0\}\subseteq[2]$
    & \begin{tikzpicture}[scale=0.7,baseline={(0,0.5)}]
      \node [vertex] (0) at (0,0) {};
      \node [vertex=white] (1) at (1,1.5) {};
      \node [vertex=white] (2) at (2,0) {};
      \draw [edge,dashed] (0) to (1) to (2) to (0);
    \end{tikzpicture}
  \\&&
  \\\begin{tikzpicture}[scale=0.7,baseline={(0,0.5)}]
      \node [vertex=white] (0) at (0,0) {};
      \node [vertex=white] (1) at (1,1.5) {};
      \node [vertex=white] (2) at (2,0) {};
      \draw [edge,dashed] (0) to node (01) [vertex=white] {} (1)
        to node (12) [vertex,solid] {} (2)
        to node (02) [vertex=white] {} (0);
      \node [vertex] (c) at (1,0.5) {};
      \draw [edge,dashed] (01) to (c);
      \draw [edge,ultra thick] (12) to (c);
      \draw [edge,dashed] (02) to (c);
    \end{tikzpicture}
    &$\{1<2\}\subseteq[2]$
    & \begin{tikzpicture}[scale=0.7,baseline={(0,0.5)}]
      \node [vertex=white] (0) at (0,0) {};
      \node [vertex] (1) at (1,1.5) {};
      \node [vertex] (2) at (2,0) {};
      \draw [edge,dashed] (0) to (1) to (2) to (0);
      \draw [edge] (1) to (2);
    \end{tikzpicture}
  \\&&
  \\\begin{tikzpicture}[scale=0.7,baseline={(0,0.5)}]
      \node [vertex=white] (0) at (0,0) {};
      \node [vertex] (1) at (1,1.5) {};
      \node [vertex=white] (2) at (2,0) {};
      \draw [edge,dashed] (0) to node (01) [vertex,solid] {} (1)
        to node (12) [vertex=white] {} (2)
        to node (02) [vertex=white] {} (0);
      \draw [edge,ultra thick] (01) to (1);
      \node [vertex=white] (c) at (1,0.5) {};
      \draw [edge,dashed] (01) to (c);
      \draw [edge,dashed] (12) to (c);
      \draw [edge,dashed] (02) to (c);
    \end{tikzpicture}
    &$\{1\}\subseteq\{0<1\}$
    & \begin{tikzpicture}[scale=0.7,baseline={(0,0.5)}]
      \node [vertex=white] (0) at (0,0) {};
      \node [vertex] (1) at (1,1.5) {};
      \node [vertex=white,draw=white] (2) at (2,0) {};
      \draw [edge,dashed] (0) to (1);
    \end{tikzpicture}
  \\&&
  \\\begin{tikzpicture}[scale=0.7,baseline={(0,0.5)}]
      \node [vertex=white] (0) at (0,0) {};
      \node [vertex=white] (1) at (1,1.5) {};
      \node [vertex=white] (2) at (2,0) {};
      \draw [edge,dashed] (0) to node (01) [vertex=white] {} (1)
        to node (12) [vertex=white] {} (2)
        to node (02) [vertex=white] {} (0);
      \node [vertex] (c) at (1,0.5) {};
      \draw [edge,dashed] (01) to (c);
      \draw [edge,dashed] (12) to (c);
      \draw [edge,dashed] (02) to (c);
    \end{tikzpicture}
    &$[2]\subseteq[2]$
    & \begin{tikzpicture}[scale=0.7,baseline={(0,0.5)}]
      \node [vertex] (0) at (0,0) {};
      \node [vertex] (1) at (1,1.5) {};
      \node [vertex] (2) at (2,0) {};
      \draw [edge,hatched] (0.center) to (1.center) to (2.center) to cycle;
    \end{tikzpicture}
    \\&&
  \end{tabular}
  \caption{The dimension of a cell $(\tau,\sigma)$ in $\pair{\Delta[2]}$ is given by the codimension of $\tau$ as a face inside $\sigma$.}
  \label{figure:cells-in-pair-subdivision}
\end{figure}

Indeed, the barycentric subdivision gives us a triangulation, whereas the pair subdivision gives us a \emph{cubification}.
However, we do not concern ourselves with the cubical structure of the pair subdivision in this present work, and so we will refer to the cubes simply as \emph{cells}.

\subsection{Homotopy theory of simplicial presheaves}
\label{subsection:simplicial-homotopy-theory}

There is an extensive theory of homotopy groups of simplicial presheaves, but we will not need the vast majority of it in for our purposes: we will content ourselves with a combinatorial definition of $\pi_0$ and some black-boxed statements about the weak equivalences in certain model structures.

Using geometric realisation, one can give a simple definition of homotopy groups of simplicial sets: we define the \emph{$n$th topological homotopy group} to be $\pi_n|X_\bullet|$.
However, passing to the geometric realisation can be computationally fiddly, and one might desire a more combinatorial approach to homotopy groups.
\emph{For Kan complexes}, there is a definition of \emph{simplicial homotopy groups} which uses the notion of \emph{simplicial homotopy}, and for which there exist rather neat results expressing when two homotopy classes are equal in terms of their representatives bounding a common simplex of one dimension higher.
This is all covered in e.g. \cite[Chapter~I]{GJ2009}.
One of the most important results concerning these simplicial homotopy groups is \cite[Chapter~I, Proposition~11.1]{GJ2009}, which implies that they are naturally isomorphic to the topological homotopy groups (i.e. those of the geometric realisation).
We do not prove that $\Green$ or $\sTwist$ are presheaves of Kan complexes, so we will not be able to make use of this statement, but there is a partial result that we can apply: if a simplicial set is such that all $2$-horns fill, then the simplicial $\pi_0$ is well defined and agrees with the topological $\pi_0$.

Recall that, if $Y$ is a space, then $\pi_0(Y)$ is the set of path-connected components of $Y$.
This means that, if $X_\bullet$ is a Kan complex, the set $\pi_0|X_\bullet|$ consists of equivalence classes of $0$-simplices of $X_\bullet$, where two $0$-simplices are equivalent if they can be connected by a zig-zag of $1$-simplices in $X_\bullet$, i.e. $v\sim w$ if and only if there exist $1$-simplices $e_0,\ldots,e_n\in X_1$ such that $v$ is an endpoint of $e_0$, $w$ is an endpoint of $e_n$, and each pair $(e_i,e_{i+1})$ share a common endpoint but $(e_i,e_{i+2})$ do not.

The definition of the simplicial $\pi_0$ is simpler, reducing the zig-zag to length~$1$.

\begin{definition}
  Let $X_\bullet$ be a Kan complex.
  Then we define the \emph{$0$th simplicial homotopy group} $\pi_0X_\bullet$ to be equivalence classes of $0$-simplices of $X_\bullet$, where two $0$-simplices are equivalent if and only if can be connected by a single $1$-simplex in $X_\bullet$, i.e. $v\sim w$ if and only if there exists a $1$-simplex $e\in X_1$ such that $v=f^1(e)$ and $w=f^0(e)$.
\end{definition}

A priori, these two definitions of $\pi_0$ need not agree: the equivalence relation on the simplicial $\pi_0$ is finer than that on the topological $\pi_0$.
As mentioned above, it turns out that all simplicial $\pi_n$ are indeed naturally isomorphic to the topological $\pi_n$, but the $n=0$ case still holds under much weaker conditions.

\begin{lemma}
\label{lemma:2-horns-fill-implies-simplicial-pi0-equals-topological-pi0}
  Let $X_\bullet$ be a simplicial set such that all $2$-horns fill.
  Then the simplicial homotopy group $\pi_0X_\bullet$ is naturally isomorphic\footnote{Recall that $\pi_0$ is merely a set, so here ``isomorphic to'' means ``in bijective correspondence with''.} to the topological homotopy group $\pi_0|X_\bullet|$.
\end{lemma}

\begin{proof}
  Since the equivalence relation on $\pi_0X_\bullet$ is finer than that on $\pi_0|X_\bullet|$, it suffices to show that if two vertices are in the same equivalence class in the latter then they are in the same equivalence class in the former.

  So suppose that $v,w\in X_0$ are such that $[v]=[w]$ in $\pi_0|X_\bullet|$.
  This means that there exists a zig-zag of $1$-simplices $e_0,\ldots,e_n\in X_1$ connecting $v$ and $w$.
  But consider the pair of $1$-simplices $(e_{n-1},e_n)$, which, by hypothesis, share a common endpoint.
  This means that they form a $2$-horn in $X_\bullet$, and, again by hypothesis, we can fill this $2$-horn to obtain, in particular, a $1$-simplex $\widetilde{e}_{n-1}$ such that $v$ and $w$ are joined by the zig-zag $e_0,\ldots,\widetilde{e}_{n-1}$.
  Repeating this $n-2$ more times we obtain a single $1$-simplex $\widetilde{e}_0$ whose endpoints are exactly $v$ and $w$, whence $[v]=[w]$ in $\pi_0X_\bullet$.
\end{proof}

The category of simplicial sets can be endowed with different model structures, but the model structure of interest to us here is the one that models the $(\infty,1)$-category of topological spaces, namely the \emph{Kan--Quillen} (or \emph{classical}) model structure (\cite[Definition~7.10.8]{Hir2003a}), which induces the \emph{global projective} model structure on the category of simplicial presheaves.
We will not provide here many details about these structures except for those of which we will later have need.
Since we only consider these model structures, we fix some terminology now.

\begin{definition}
  We say that a simplicial set is \emph{fibrant} if it is a Kan complex, and that a simplicial presheaf is \emph{globally fibrant} if it is a presheaf of Kan complexes.

  We say that a morphism $f\colon X_\bullet\to Y_\bullet$ of simplicial sets is a \emph{weak equivalence} if it is a weak equivalence in the Kan--Quillen model structure, i.e. if it induces an isomorphism on topological homotopy groups $f\colon\pi_n|X_\bullet|\cong\pi_n|Y_\bullet|$ for all $n\in\mathbb{N}$.

  We say that a morphism $f\colon\scr{F}\to\scr{G}$ of simplicial presheaves is a (\emph{global}) \emph{weak equivalence} if it is a weak equivalence in the induced global projective model structure, i.e. if it induces an isomorphism on topological homotopy groups $f\colon\pi_n|\scr{F}(X)|\cong\pi_n|\scr{G}(X)|$ for all objects $X$ and all $n\in\mathbb{N}$.
\end{definition}

\subsection{Čech totalisation}
\label{subsection:cech-totalisation}

Combining the Čech nerve and the totalisation of a cosimplicial simplicial set gives us a construction which we will use repeatedly, and it deserves a name.

\begin{definition}
  Let $X$ be a space with cover $\cover$, and let $\scr{F}\colon\Space^\op\to\sSet$ be a simplicial presheaf on the category of spaces.
  We define the \emph{Čech totalisation} of $\scr{F}$ (\emph{at the cover $\cover$}) to be the simplicial set $\Tot\scr{F}(\cechnerve\cover_\bullet)$ given by the totalisation of the cosimplicial simplicial set $\scr{F}(\cechnerve\cover_\bullet)$ given by evaluating $\scr{F}$ on the Čech nerve (\cref{definition:evaluating-on-the-cech-nerve}).
\end{definition}

This procedure of Čech totalisation looks very similar to some sort of sheafification: indeed, if applied to a presheaf of sets (considered as a constant simplicial presheaf) then it recovers the usual construction of sheafification via taking sections of the espace étalé.
Not only that, but \cref{corollary:cech-tot-of-kan-is-holim} tells us that, in the case of presheaves of Kan complexes, this totalisation is really the homotopy limit.
More generally, there are certain conditions under which Čech totalisation does compute the sheafification of a simplicial presheaf (for example, one should expect to have to take a colimit over refinements of covers, so the cover $\cover$ should be particularly nice somehow).
However, we do not concern ourselves with such questions in this paper, referring the interested reader instead to \cite[\S5.1]{GMTZ2022a}.
For us it is sufficient that this construction returns ``interesting'' results in many examples, and that it satisfies the following useful properties.

\begin{lemma}[{\cite[Proposition~C.1]{GMTZ2022a}}]
\label{lemma:tot-of-cech-of-kan-is-kan}
  Let $\scr{F}$ be a presheaf of Kan complexes on $\Space$.
  Then $\Tot\scr{F}(\cechnerve\cover_\bullet)$ is a Kan complex.
\end{lemma}

\begin{proof}
  The idea of the proof is relatively straightforward: we use the fact that $\Tot$ is the right-adjoint part of a Quillen equivalence between the Reedy model structure on cosimplicial simplicial sets and the Kan--Quillen model structure on simplicial sets, and thus preserves fibrant objects; we then apply \cref{lemma:kan-on-cech-is-reedy-fibrant}.
\end{proof}

\begin{lemma}
\label{lemma:tot-of-we-of-kan-is-we}
  Let $\scr{F}$ and $\scr{G}$ be presheaves of Kan complexes on $\Space$.
  If $\scr{F}$ and $\scr{G}$ are weakly equivalent, then so too are their Čech totalisations $\Tot\scr{F}(\cechnerve\cover_\bullet)$ and $\Tot\scr{G}(\cechnerve\cover_\bullet)$.

  In other words, Čech totalisation sends weak equivalences of presheaves of Kan complexes to weak equivalences of Kan complexes.
\end{lemma}

\begin{proof}
  Since we are using the global projective model structure on simplicial presheaves (\cref{subsection:simplicial-homotopy-theory}), we know that $\scr{F}(U)$ and $\scr{G}(U)$ are weakly equivalent for all spaces $U$.
  Since $\scr{F}$ and $\scr{G}$ are presheaves of Kan complexes, we can apply \cref{lemma:kan-on-cech-is-reedy-fibrant}.
  Then we again use the fact that $\Tot$ is a Quillen right adjoint, and thus preserves weak equivalences between fibrant objects.
\end{proof}

Our initial justification in \cref{subsection:totalisation-and-holim} for studying the totalisation was that it computes the homotopy limit in the case of Reedy fibrant objects, and we do indeed find ourselves in this case whenever we have a presheaf of Kan complexes.

\begin{lemma}[{\cite[Lemma~C.5]{GMTZ2022a}}]
  \label{lemma:kan-on-cech-is-reedy-fibrant}
  Let $\scr{F}$ be a presheaf of Kan complexes on $\Space$, and let $X$ be a space with cover $\cover$.
  Then the Čech totalisation $\scr{F}(\cechnerve\cover_\bullet)$ is a Reedy fibrant cosimplicial simplicial set.
\end{lemma}

\begin{corollary}
\label{corollary:cech-tot-of-kan-is-holim}
  Let $\scr{F}$ be a presheaf of Kan complexes on $\Space$, and let $X$ be a space with cover $\cover$.
  Then the Čech totalisation $\scr{F}(\cechnerve\cover_\bullet)$ of $\scr{F}$ computes the homotopy limit $\holim\scr{F}(\cechnerve\cover_\bullet)$.
\end{corollary}

\begin{proof}
  This follows immediately from \cref{lemma:reedy-fibrant-implies-tot-is-holim}.
\end{proof}

\begin{lemma}
\label{lemma:comparsion-of-tot-dg-and-tot-sset}
  Let $\scr{F}\colon\Space^\op\to\dgCat$ be a presheaf of dg-categories that sends finite products to coproducts.
  Then there is a weak equivalence of Kan complexes
  \[
    \Tot\core{\dgnerve\scr{F}(\cechnerve\cover)}
    \simeq \core{\dgnerve\big(\Tot\scr{F}(\cechnerve\cover)\big)}
  \]
  where on the left-hand side we take the totalisation of cosimplicial simplicial sets, and on the right-hand side we take the totalisation of cosimplicial dg-categories.
\end{lemma}

We haven't formally defined Čech totalisation for dg-categories, and we will not do so, nor will we explain the Dwyer--Kan model structure on $\dgCat$ (\cite{Tab2005}) since we expect this Lemma to mainly be of interest to those already somewhat familiar with these: it is somewhat of a comparison result for \cite{BHW2017}, as we explain in \cref{remark:turning-resolution-by-twisting-cochains-into-pi-0-statement}.

\begin{proof}
  Recalling that evaluation on the Čech nerve is given exactly by pre-composition with the functor $\cechnerve\cover_\bullet^\op\colon\Delta\to\Space^\op$ (\cref{definition:evaluating-on-the-cech-nerve}), the statement of the lemma is equivalent to the commutativity (up to weak equivalence) of the diagram
  \[
    \begin{tikzcd}[sep=huge]
      {[}\Space^\op,\dgCat{]}_\mathrm{fpp}
        \ar[r,"(-)\circ\cechnerve\cover_\bullet^\op"]
        \ar[d,swap,"\dgnerve(-)"]
      & {[}\Delta,\dgCat{]}
        \ar[r,"\Tot"]
        \ar[d,swap,"\dgnerve(-)"]
      & \dgCat
        \ar[dd,"\core{\dgnerve(-)}"]
    \\{[}\Space^\op,\sSet{]}
        \ar[r,swap,"(-)\circ\cechnerve\cover_\bullet^\op"]
        \ar[d,swap,"\core{-}"]
      & {[}\Delta,\sSet{]}
        \ar[d,swap,"\core{-}"]
    \\{[}\Space^\op,\sSet{]}
        \ar[r,swap,"(-)\circ\cechnerve\cover_\bullet^\op"]
      & {[}\Delta,\sSet{]}
        \ar[r,swap,"\Tot"]
      & \sSet
    \end{tikzcd}
  \]
  where $[\Space^\op,\dgCat]_\mathrm{fpp}$ denotes the subcategory of $[\Space^\op,\dgCat]$ consisting of those presheaves that preserve finite products.

  The two smaller squares on the left commute on the nose, since composition of functors is strictly associative, and the horizontal arrows are given by pre-composition with the opposite of the Čech nerve and the vertical arrows are given by post-composition with the dg-nerve or maximal-Kan-complex functor.

  Next, we can apply \cite[Proposition~4.3]{BHW2017}, which says that the pre-composition of a finite-product-preserving presheaf of dg-categories with a split simplicial object is a Reedy fibrant cosimplicial dg-category, since the Čech nerve of an open cover is always a split simplicial object, and the presheaves preserve finite products by assumption.
  This means that the top-left horizontal arrows lands inside the subcategory of fibrant objects of $[\Delta,\dgCat]$.
  Similarly, \cref{lemma:kan-on-cech-is-reedy-fibrant} tells us that a presheaf of Kan complexes evaluated on the Čech nerve is Reedy fibrant cosimplicial simplicial set, and so the composite of the two leftmost vertical arrows followed by the bottom left horizontal arrow also lands inside the subcategory of fibrant objects of $[\Delta,\sSet]$.
  This, combined with the strict commutativity of the two smaller squares in the above diagram, allows us to reduce to studying the diagram
  \[
    \begin{tikzcd}[sep=huge]
      {[}\Space^\op,\dgCat{]}_\mathrm{fpp}
        \ar[r,"\cechnerve\cover_\bullet^\op"]
      & {[}\Delta,\dgCat{]}^\mathrm{fib}
        \ar[r,"\Tot"]
        \ar[d,swap,"\core{\dgnerve(-)}"]
      & \dgCat^\mathrm{fib}
        \ar[d,"\core{\dgnerve(-)}"]
    \\&{[}\Delta,\sSet{]}^\mathrm{fib}
        \ar[r,swap,"\Tot"]
      & \sSet^\mathrm{fib}
    \end{tikzcd}
  \]
  since every object in $\dgCat$ is fibrant in the Dwyer--Kan model structure, and \cref{lemma:tot-of-cech-of-kan-is-kan} tells us that the Čech totalisation of any presheaf of Kan complexes is a Kan complex and thus fibrant in $\sSet$.
  Note that there is a small abuse of notation which makes this square look like it should trivially commute, but there is indeed something to prove: the two totalisations take place in different categories, and the right-hand vertical arrow is the pointwise version of the left-hand one.

  Since we are only considering Reedy fibrant simplicial objects, the totalisation computes the homotopy limit.
  More precisely, we have a natural weak equivalence $\Tot Y_\bullet^\anotherbullet\simeq\holim Y_\bullet^\anotherbullet$ for all $Y_\bullet^\anotherbullet$ in either $\dgCat^\mathrm{fib}$ or $\sSet^\mathrm{fib}$.
  This means that, \emph{under the assumption that the vertical arrows send weak equivalences to weak equivalences}, it suffices to show that the diagram
  \[
    \begin{tikzcd}[sep=huge]
      {[}\Space^\op,\dgCat{]}_\mathrm{fpp}
        \ar[r,"\cechnerve\cover_\bullet^\op"]
      &{[}\Delta,\dgCat{]}^\mathrm{fib}
        \ar[r,"\holim"]
        \ar[d,swap,"\core{\dgnerve(-)}"]
      & \dgCat^\mathrm{fib}
        \ar[d,"\core{\dgnerve(-)}"]
    \\&{[}\Delta,\sSet{]}^\mathrm{fib}
        \ar[r,swap,"\holim"]
      & \sSet^\mathrm{fib}
    \end{tikzcd}
  \]
  commutes.
  We shall first prove this, and then show that the two vertical arrows do indeed satisfy this hypothesis.\footnote{We could give a much more succinct, but more abstract, proof from here on, simply appealing to the fact that $k^!\circ\dgnerve$ is a Quillen right adjoint and that $k^!\simeq\core{-}$ (see \cref{remark:extending-core-to-sset}), but we opt to continue ``by hand''.}

  So let $\cal{D}^\bullet\in[\Delta,\dgCat]^\mathrm{fib}$ be a Reedy fibrant cosimplicial dg-category given by evaluating some finite-product-preserving presheaf of dg-categories on the Čech nerve.
  Since it is Reedy fibrant, we know that
  \[
    \holim\cal{D}^\bullet
    \simeq \lim\cal{D}^\bullet
  \]
  and, by the assumption that the vertical arrows send weak equivalences to weak equivalences, we thus have that
  \[
    \core{\dgnerve\holim\cal{D}^\bullet}
    \simeq \core{\dgnerve\lim\cal{D}^\bullet}.
  \]
  Now note that $\core{\dgnerve(-)}$ is the composition of three right adjoints
  \[
    \dgCat
    \xrightarrow{\dgnerve} \QuasiCat
    \xrightarrow{\core{-}} \Kan
    \hookrightarrow \sSet
  \]
  (since the inclusion $\Kan\hookrightarrow\sSet$ also admits a left adjoint, as mentioned in \cref{definition:max-kan}), which means that it itself is a right adjoint and thus commutes with limits, whence
  \[
    \core{\dgnerve\lim\cal{D}^\bullet}
    \cong \lim\core{\dgnerve\cal{D}^\bullet}.
  \]
  But we have already argued that $\core{\dgnerve\cal{D}^\bullet}$ is Reedy fibrant (by strict commutativity of the two leftmost squares in the original diagram), and so its limit actually computes the homotopy limit:
  \[
    \lim\core{\dgnerve\cal{D}^\bullet}
    \simeq \holim\core{\dgnerve\cal{D}^\bullet}.
  \]
  Chaining these equivalences together, we see that
  \[
    \core{\dgnerve\holim\cal{D}^\bullet}
    \simeq \holim\core{\dgnerve\cal{D}^\bullet}
  \]
  and so the diagram commutes up to weak equivalence.

  It remains only to show that $\core{\dgnerve(-)}$ sends weak equivalences to weak equivalences both individually and pointwise, i.e. both as a functor $\dgCat^\mathrm{fib}\to\sSet^\mathrm{fib}$ and as a functor $[\Delta,\dgCat]^\mathrm{fib}\to[\Delta,\sSet]^\mathrm{fib}$.
  The functor $\dgnerve\colon\dgCat\to\sSet_\mathrm{Joyal}$ is a Quillen right adjoint \emph{when we endow $\sSet$ with the Joyal model structure}, and thus preserves weak equivalences between fibrant objects.
  Since all dg-categories are fibrant in the Dwyer--Kan model structure, this means that a weak equivalence $\cal{C}\simto\cal{D}$ of dg-categories gets sent to a weak equivalence $\dgnerve\cal{C}\simto\dgnerve\cal{D}$ \emph{in the Joyal model structure}.
  But a categorical equivalence of quasi-categories (i.e. a weak equivalence in the Joyal model structure) induces a weak equivalence (in the Kan--Quillen model structure) of their maximal Kan complexes (\cite[Lemma~34]{Jar2019}).
  Thus we get a weak equivalence $\core{\dgnerve\cal{C}}\simto\core{\dgnerve\cal{D}}$, as required.
  As for the induced functor $[\Delta,\dgCat]^\mathrm{fib}\to[\Delta,\sSet]^\mathrm{fib}$, since the weak equivalences in the Reedy model structure on any $[\mathcal{R},\mathcal{M}]$ are simply those that are object-wise weak equivalences in $\mathcal{M}$, we are done.
\end{proof}

\begin{remark}
\label{remark:extending-core-to-sset}
  In the proof of \cref{lemma:comparsion-of-tot-dg-and-tot-sset}, one might wonder why we don't simply show that the maximal-Kan functor is a \emph{Quillen} right adjoint, since it is already a right adjoint by definition, and then commutativity with the homotopy limit would be immediate.
  But note that the domain of the maximal-Kan functor $\core{-}$ is $\QuasiCat$, not all of $\sSet$, and so we cannot simply compose it with $\dgnerve\colon\dgCat\to\sSet$.
  It \emph{is} true that the image of the dg-nerve actually lies entirely inside $\QuasiCat\hookrightarrow\sSet$, but the dg-nerve only gives a Quillen right adjoint when considered with codomain equal to all of $\sSet$.

  It is possible to ``model'' the maximal Kan functor by a functor $k^!\colon\sSet\to\sSet$ which then does realise a Quillen adjunction (indeed, even a homotopy colocalisation) between the Kan--Quillen and the Joyal model structures on $\sSet$ (\cite[Proposition~1.16 through to Proposition~1.20]{JT2007}), so that $k^!\dgnerve\colon\dgCat\to\sSet_\mathrm{Kan–Quillen}$ is indeed a Quillen right adjoint, but for our purposes it is convenient to work with the direct definition of the maximal Kan complex instead.
\end{remark}

\begin{remark}
\label{remark:induced-quillen-adjunction-for-reedy}
  In the proof of \cref{lemma:comparsion-of-tot-dg-and-tot-sset}, we use the fact that weak equivalences in the Reedy model structure are defined object-wise.
  If we had opted to use the Quillen right adjoint $k^!$ from \cref{remark:extending-core-to-sset} instead, then we could also appeal to a more general fact about Reedy model structures: if we have a Quillen adjunction $\mathcal{M}\rightleftarrows\mathcal{N}$ then this induces a Quillen adjunction $[\mathcal{R},\mathcal{M}]\rightleftarrows[\mathcal{R},\mathcal{N}]$ between Reedy model structures for any Reedy category $\mathcal{R}$.
  To prove this, note that e.g. a right adjoint preserves limits and thus sends matching objects in $[\mathcal{R},\mathcal{N}]$ to matching objects in $[\mathcal{R},\mathcal{M}]$, and a Quillen right adjoint also preserves fibrations; these two facts combined tell us that the Quillen right adjoint will send Reedy fibrations to Reedy fibrations.
\end{remark}

\subsubsection*{Example: the space of principal $G$-bundles}
\label{subsection:abstract-machinery-for-example}

\begin{remark}
  This example can be seen as a $1$-categorical version of \cite[\S3.2.1]{FSS2012}; we will see a full $\infty$-categorical example when we define the simplicial presheaf $\Twist$ in \cref{subsection:twist-construction}.
\end{remark}

By considering an example of a simplicial presheaf built from the categorical nerve, we can start to see how Čech totalisation can be thought of as ``introducing geometry''.
Here we sketch a general construction that provides inspiration for our main object of study, introduced at the end of \cref{subsection:perfect-complexes}.
We provide details of the specific case of principal $\GL_n(\mathbb{R})$-bundles in \cref{appendix:example}.

Let $G$ be a Lie group, so that $G$ is, in particular, also an element of the category of smooth manifolds $\Smth$, and consider the presheaf on $\Smth$ given by Yoneda:
\[
  \yon(G) = \Smth(-,G).
\]
Using the Lie group structure of $G$, we can endow $\yon(G)$ with the structure of a Lie group pointwise, and thus consider $\yon(G)$ as a presheaf of Lie groups.
This means that we can deloop $\yon(G)$ to obtain a presheaf of one-element groupoids:
\[
  \mathbb{B}\yon(G)(-).
\]
That is, for $X\in\Smth$, the groupoid $\mathbb{B}\yon(G)(X)$ has one object, which we denote by $*$, and with $\Hom(*,*)\cong\Smth(X,G)$, where we again use the group structure of $G$ to endow $\Smth(X,G)$ with a group structure.
We can then take the categorical nerve of this to obtain a presheaf of simplicial sets:
\[
  \nerve\mathbb{B}\yon(G)(-).
\]
Abstractly, then, we have a functor
\[
  \nerve\mathbb{B}\yon\colon \LieGroup\to[\Smth^\op,\sSet].
\]

Next, write $\SmthC$ to mean the category whose objects are pairs $(X,\cover)$, where $X\in\Smth$, and $\cover$ is a good\footnote{That is, all non-empty finite intersections $U_{\alpha_0\ldots\alpha_p}$ (including the case where $p=0$) of open sets in the cover are contractible.} cover of $X$, and whose morphisms $(X,\cover)\to(Y,\anothercover)$ are the morphisms $f\colon X\to Y$ in $\Smth$ such that $\cover$ is a refinement of $f^{-1}(\anothercover)$.
Then we have the Čech nerve
\[
  \cechnerve\colon \SmthC\to[\Delta^\op,\Smth]
\]
which, using the fact that $[\cal{C},\cal{D}]^\op\cong[\cal{C}^\op,\cal{D}^\op]$ for any categories $\cal{C}$ and $\cal{D}$, induces a functor
\[
  \cechnerve^\op\colon \SmthC^\op\to[\Delta^\op,\Smth]^\op\cong[\Delta,\Smth^\op].
\]
So pre-composing $\nerve\mathbb{B}\yon$ with the opposite of the Čech nerve, we obtain a functor
\[
  (\cechnerve^\op)^*\nerve\mathbb{B}\yon\colon \LieGroup\to[\SmthC^\op,[\Delta,\sSet]]=[\SmthC^\op,\csSet].
\]
This means that, given any $G\in\LieGroup$, we obtain a presheaf of cosimplicial simplicial sets on $\SmthC$.
To simplify notation, we write
\[
  \bigfunctor \coloneqq (\cechnerve^\op)^*\nerve\mathbb{B}\yon.
\]

Finally then, we can apply totalisation to obtain a functor with values in presheaves of simplicial sets:
\[
  \Tot(\bigfunctor)\colon\LieGroup\to[\SmthC^\op,\sSet].
\]

\begin{remark}
  Before taking the totalisation, $\bigfunctor$ took values in cosimplicial simplicial sets.
  The \emph{cosimplicial} structure came from pulling back along the opposite\footnote{The Čech nerve itself is a simplicial object, so the opposite turns it into a cosimplicial one.} of the Čech nerve, and the \emph{simplicial} structure came from the ordinary nerve; we totalise over the \emph{former}.
\end{remark}

So what is the purpose of this functor?
If we apply it to a specific Lie group $G$, then, since $\mathbb{B}\yon(G)(X)$ is a groupoid for any manifold $X$, the resulting simplicial set $\Tot(\bigfunctor)(G)(X)$ will be a Kan complex, i.e. a space.
It turns out that the points of this space are exactly principal $G$-bundles, and the paths are exactly isomorphisms of principal $G$-bundles: this space deserves the name ``the space of principal $G$-bundles''.
We provide the details of this argument for the case where $G=\GL_n(\mathbb{R})$ in \cref{appendix:example}.

\begin{remark}
\label{remark:pulling-back-along-cech-op}
  What is very important in the construction described above is that we pull back along the \emph{opposite} of the Čech nerve.
  Of course, we are required to do this in order to compose the functors, but it also has an important geometric significance: when working with (pre)sheaves of functions on open sets, it ensures that we will have trivial codegeneracy maps and \emph{restriction} coface maps, and not trivial face maps and \emph{extension} degeneracy maps.

  To understand what we mean by this, consider the Čech nerve, which has face maps $f_p^i\colon U_{\alpha_0\ldots\alpha_p}\to U_{\alpha_0\ldots\widehat{\alpha_i}\ldots\alpha_p}$ and degeneracy maps $s_i^p\colon U_{\alpha_0\ldots\alpha_p}\to U_{\alpha_0\ldots\alpha_i\alpha_i\ldots\alpha_p}$.
  The degeneracy maps are trivial: $U_{\alpha_0\ldots\alpha_p} = U_{\alpha_0\ldots\alpha_i\alpha_i\ldots\alpha_p}$; the face maps are (in general) non-trivial: $U_{\alpha_0\ldots\alpha_p}\subseteq U_{\alpha_0\ldots\widehat{\alpha_i}\ldots\alpha_p}$.
  If we are considering, say, a (pre)sheaf $\scr{F}$ such that $\scr{F}(U)$ consists of some sort of functions on $U$, then defining a map $\scr{F}(U_\alpha)\to\scr{F}(U_{\alpha\alpha})$ is trivial, since we can simply take the identity; defining a map $\scr{F}(U_{\alpha\beta})\to\scr{F}(U_\alpha)$ is \emph{hard}, since we might not be able to extend functions.
  Working with the \emph{opposite} of the Čech nerve, however, means that we will not have this problem: we will have to construct maps of the form $\scr{F}(U_{\alpha})\to\scr{F}(U_{\alpha\beta})$, and this can be done by simply restricting the functions on the former.

  This is explained in the context of a worked example in \cref{appendix:example}.
\end{remark}

\subsection{Perfectness of complexes}
\label{subsection:perfect-complexes}

The classical references for the various notions relating to perfectness are \cite[Exposés~I and II]{BGI1971}; see also \cite[\S2.1]{Wei2016} and \cite[\href{https://stacks.math.columbia.edu/tag/08C3}{Tag 08C3}]{stacks-project}.

One important fact of complex-analytic geometry is that not every coherent analytic sheaf can be resolved by a complex of locally free sheaves, but it can be \emph{locally} resolved.
Indeed, throughout this paper, \emph{the motivating example is always when $(X,\OO_X)$ is a complex-analytic manifold with the sheaf of holomorphic functions}.
Perfectness conditions allow us to study this phenomenon more generally.

\begin{definition}
\label{definition:perfect-complexes}
  Let $(X,\OO_X)$ be a locally ringed space, and $M^\bullet$ a cochain complex of $\OO_X$-modules.
  \begin{itemize}
    \item We say that $M^\bullet$ is \emph{finitely generated free} if it is bounded and such that each $M^i$ is a finite\footnote{By which we mean ``of finite rank''. Note that, if $X$ is connected, then the rank is also constant.} free $\OO_X$-module.
    \item We say that $M^\bullet$ is \emph{strictly perfect} if it is bounded and such that each $M^i$ is a finite locally free $\OO_X$-module.
    \item We say that $M^\bullet$ is \emph{perfect} if it is locally quasi-isomorphic to a strictly perfect complex.
      That is, if, for all $x\in X$, there exists some open neighbourhood $U$ of $x$, and some bounded complex $L_U^\bullet$ of finite locally free $\OO_X$-modules on $U$, such that $M^\bullet|U \simeq L_U^\bullet$.
  \end{itemize}
  We write $\Free(X)$ to denote the dg-category of finitely generated free complexes on $(X,\OO_X)$.
\end{definition}

Here we are mainly interested in finitely generated free complexes, and we mention strictly perfect and perfect complexes simply for context.
The ``full'' story about these finiteness conditions involves the fact that twisting cochains constitute a dg-enhancement of the category of perfect complexes (a result of \cite{Wei2016}), something to which we later allude in \cref{theorem:1-simplices-in-complex-analytic-twist} and \cref{remark:turning-resolution-by-twisting-cochains-into-pi-0-statement}.

Now we can restate the fact about local resolutions of coherent analytic sheaves using this more abstract terminology.
Indeed, \cite[Exposé~I, Exemple~5.11]{BGI1971} tells us that the derived category of bounded complexes of coherent analytic sheaves on a complex-analytic manifold $X$ is equivalent to the derived category of perfect complexes on $X$.
But the fact that there exist coherent analytic sheaves that do \emph{not} admit global resolutions by locally free sheaves is an example of the fact that, although strictly perfect clearly implies perfect, the converse is not necessarily true (see also \cite[Remark~2.4]{Wei2016} for examples of how this converse is also false in the algebraic case).

\begin{remark}
\label{remark:strictly-perfect-locally-ringed}
  The hypothesis that $(X,\OO_X)$ be \emph{locally} ringed is necessary for our definition of \emph{strictly} perfect in \cref{definition:perfect-complexes}: for an arbitrary ringed space, it is not necessarily true that a direct summand of a finite free $\OO_X$-module is finite and locally free, and we have used this statement to simplify the definition of strictly perfect complexes.
  However, in the rest of this paper we do not deal with the notion of strictly perfect complexes, and so we will instead work in the more general setting of arbitrary ringed spaces.
\end{remark}

\begin{remark}
  The constructions that we are going to give build things out of free $\OO_X$-modules, so if we want any hope of recovering \emph{coherent} sheaves of $\OO_X$-modules at the end somehow, then it needs to be the case that \emph{free modules are themselves coherent}.
  In the complex-analytic setting, this is ensured by the Oka coherence theorem, which tells us that $\OO_X$ is coherent; in the complex-algebraic setting, this is ensured if we work with a locally Noetherian scheme (\cite[\href{https://stacks.math.columbia.edu/tag/01XZ}{Tag 01XZ}]{stacks-project}).
  Although the constructions still ``make sense'' in settings where $\OO_X$ is \emph{not} coherent, we do not know how exactly the objects that we construct will relate to coherent sheaves.
  To deal with such questions, one would need to appeal to the more general definition of \emph{pseudo-coherence} (\cite[Exposé~I, \S0.~Introduction]{BGI1971}).

  More generally, the relation between perfectness and coherence is an interesting subject of study.
  One particularly useful result is that, if the local rings $\OO_{X,x}$ are all regular (which is the case if, for example, $(X,\OO_X)$ is a complex-analytic manifold, and thus smooth), then every coherent sheaf is perfect, and, more generally, every pseudo-coherent complex with locally bounded cohomology is perfect \cite[Exposé~I, Corollaire~5.8.1]{BGI1971}.
\end{remark}

\section{Three simplicial presheaves}
\label{section:three-simplicial-presheaves}

Using the Čech totalisation from \cref{subsection:cech-totalisation} we can ``apply geometry'' to presheaves of simplicial sets, and it turns out that many familiar geometric objects arise in this way, such as complexes of locally free sheaves.
We start by considering exactly this example, building it from the category of finitely generated free complexes (\cref{definition:perfect-complexes}), and then describe how to make this functorial, turning it into a presheaf on connected ringed spaces.
From this presheaf we will construct three generalisations (\cref{subsection:twist-construction}, \cref{subsection:green-construction}, and \cref{subsection:stwist-construction}), which form the main objects of study of this paper.

As a gentle reminder, we draw attention to \cref{remark:abuse-of-terminology-no-rectification}: \textbf{these presheaves are in fact only \emph{pseudo}-presheaves in general, and if applied to anything other than the Čech nerve must possibly first be rectified}.

\subsection{Narrative}
\label{subsection:narrative}

Since we will eventually be interested in \emph{local} properties of simplicial presheaves on ringed spaces, from now on we freely switch between $(U,\OO_U)$ and $(X,\OO_X)$ as notation for an arbitrary ringed space.

We are interested in the most restrictive of the notions of perfectness from \cref{definition:perfect-complexes}, namely that of finitely generated free complexes.
Given a ringed space $(U,\OO_U)$, the objects of $\Free(U)$ are \emph{bounded} complexes of \emph{finite-rank} \emph{free} $\OO_U$-modules\footnote{Formally, we really work with the skeleton of this category: a free sheaf is uniquely determined by its rank.}
\[
  C = \big(
  0 \to \OO_U^{\oplus r_k} \xrightarrow{d_{k-1}} \OO_U^{\oplus r_{k-1}} \xrightarrow{{d_{k-2}}} \quad\ldots\quad \xrightarrow{d_2} \OO_U^{\oplus r_2} \xrightarrow{d_1} \OO_U^{\oplus r_1} \to 0
  \big)
\]
and the morphisms are given by
\[
  \Hom_{\Free(U)}^n(C,D) =
  \prod_{m\in\mathbb{Z}} \Hom_{\OO_U}(C^m,D^{m+n}).
\]
This gives a dg-category by defining the differential $\partial$ on the hom-sets as in \cref{definition:dg-category}.

Taking the ordinary nerve of this category\footnote{Recall \cref{definition:1-nerve-of-dg-category}: this really means the ordinary nerve of $K_0$ of this category.} gives us a simplicial set $\nerve\Free(U)$, whose $p$-simplices are composible sequences of $p$-many morphisms:
\[
  \nerve\Free(U)_p =
  \big\{
    C_0\xrightarrow{\varphi_1}C_1\xrightarrow{\varphi_2}\ldots\xrightarrow{\varphi_p}C_p
    \mid \varphi_i\in\Hom_{K_0\Free(U)}(C_{i-1},C_i)
  \big\}
\]
(though it will prove useful to think of the nerve as in \cref{remark:blown-up-nerve}, so that we really have $\binom{p+1}{2}$ many morphisms).
Finally, we take the maximal Kan complex
\[
  \core{\nerve\Free(U)}\subseteq\nerve\Free(U)
\]
which, by \cref{lemma:nerve-inside-dg-nerve}~(ii), is equivalent to asking that all the $\varphi_i$ be \emph{isomorphisms}.
Now we have a simplicial presheaf on ringed spaces given by
\[
  (U,\OO_U) \mapsto \core{\nerve\Free(U)}
\]
and so we can try to understand the Čech totalisation (at a given space $X$ and cover $\cover$) of this simplicial presheaf through its homotopy groups:
\[
\label{equation:what-we-want-to-generalise}
  \pi_n\Tot\core{\nerve\Free(\cechnerve\cover_\bullet)}.
\tag{$\divideontimes$}
\]
Note that here we are extending Čech totalisation from simplicial presheaves on spaces to simplicial presheaves on \emph{ringed} spaces, which relies on us extending the Čech nerve from spaces to ringed spaces: with which ringed space structure do we endow $\cechnerve\cover_p$?
The sheaf condition requires that
\[
  \colim_{i\in I}(U_i,\OO_{U_i})
  \cong \big(\colim_{i\in I} U_i, \lim_{i\in I}((\iota_i)_*\OO_{U_i})\big)
\]
where $\iota_i\colon U_i\to\colim U_i$ are the morphisms induced by the universal property of the colimit.
In particular, this tells us that we should set
\[
  \coprod_{i\in I}(U_i,\OO_{U_i})
  = \left(
    \coprod_{i\in I} U_i,
    \prod_{i\in I}\big((\iota_i)_*\OO_{U_i}\big)
  \right).
\]

\bigskip

From one point of view, finding good generalisations of the construction in ($\divideontimes$)  is one key motivation for this entire paper.
To explain this, let us take a step back and first explain why we care about \cref{equation:what-we-want-to-generalise}.

The fact that \cref{equation:what-we-want-to-generalise} describes any sort of interesting mathematical object is at least partially justified by an example.\footnote{On a higher level, the general idea of (co)descent objects being useful for computing higher stackification is one with far-reaching consequences throughout geometry and other fields. We do not give a full justification for this here.}
In \cref{appendix:example} we show how this machinery recovers a space whose points are principal bundles and whose paths are gauge transformations (and thus whose loops are gauge groups).
For our applications, we use locally free sheaves instead of principal bundles, since these admit morphisms that are not simply automorphisms, which is necessary for the following.

So we start with the notion of locally free sheaves (on some fixed ringed space), and think about useful ways in which we can generalise this.
If we let the rank of the sheaf change across open subsets, and allow things to be \emph{surjected on} by something free instead of being free themselves, all in a ``controlled'' way, then we arrive at the notion of \emph{coherent} sheaf.
The category of coherent sheaves is also very well behaved: it is an abelian category, whereas the category of vector bundles is not.
Because of this (amongst many other reasons), coherent sheaves turn out to be very useful objects to study.
For some particularly nice ringed spaces, the category of coherent sheaves is equivalent to the category of \emph{cochain complexes} of locally free sheaves, which suggests that we might eventually wish to consider cochain complexes.

To relate this back to the story we are trying to tell, let's consider the subcategory of $\Free(U)$ spanned by complexes concentrated in degree zero (i.e. each complex is just a single locally free sheaf).
Then a point in the $\pi_0$ case of \cref{equation:what-we-want-to-generalise} describes the data of a free sheaf over each open subset, with isomorphisms between them on overlaps wherever possible.
In particular, the rank of the free sheaf on each open subset is the same.
That is, we are describing exactly locally free sheaves \emph{of constant rank}.
If we consider all of $\Free(U)$ then we end up with something similar: we have a complex of free sheaves on each open subset, with isomorphisms between them on overlaps, so we are describing cochain complexes of locally free sheaves of constant rank\footnote{In the case where $(U,\OO_U)$ is \emph{locally} ringed, these are exactly strictly perfect complexes, following \cref{remark:strictly-perfect-locally-ringed}.} (the formal version of this statement is \cref{theorem:tot0-free-is-loc-free}).
But we wanted to be able to talk about objects where the rank can jump across open subsets, so we see that \cref{equation:what-we-want-to-generalise} is too strict or discrete in some sense, since sheaves of different ranks cannot interact with one another.
There are (at least) two natural ways to solve this problem:
\begin{enumerate}
  \item we could use \cref{equation:what-we-want-to-generalise} as local input for some simplicial construction, obtaining an infinite tower of homotopical data; or
  \item we could replace the nerve in \cref{equation:what-we-want-to-generalise} by the dg-nerve, since \cref{lemma:nerve-inside-dg-nerve} tells us that then we won't be restricted to isomorphisms, but will instead be allowing \emph{quasi-isomorphisms}.\footnote{Since we are working with \emph{free} modules, the generalised Whitehead theorem tells us that quasi-isomorphisms are exactly chain homotopy equivalences.}
\end{enumerate}
We will end up formalising both of these approaches: the first in \cref{subsection:green-construction}, and the second in \cref{subsection:twist-construction}.
The obvious question then presents itself: ``how do these two constructions relate to one another?''.
One way of answering this is to consider what happens when we apply both simultaneously, and to ask if this lets us mediate between them --- this is what we do in \cref{subsection:stwist-construction} and further in \cref{section:morphisms-between-the-presheaves}.

\bigskip

The first step is to generalise the construction described above to obtain a simplicial presheaf $\core{\nerve\Free(-)}$ on the category of connected ringed spaces.
From this, we will construct three simplicial presheaves which are the main subjects of this current paper.
All of these presheaves are constructed precisely so that \cref{theorem:tot0-all-three} holds, i.e. so that we recover \emph{Green complexes}, \emph{twisting cochains}, and \emph{simplicial twisting cochains} after Čech totalisation.
On this note, we also point out that these three presheaves will be named for what they become \emph{after} applying Čech totalisation in the complex-analytic case, not for what they are beforehand.

The category on which these presheaves are defined is the category $\ConnRingSpace$ of \emph{connected} ringed spaces, where we require connectedness in order for the rank of a free sheaf to be constant.

\begin{lemma}
\label{lemma:free-gives-pseduopresheaf}
  The assignment $(U,\OO_U)\mapsto\Free(U)$ that sends a ringed space to the dg-category of bounded complexes of free modules on that space defines a pseudofunctor $\Free\colon(\RingSpace)^\op\to\dgCat$ by sending a morphism $(f,f^\sharp)\colon(U,\OO_U)\to(V,\OO_V)$ to the pullback $f^*=f^{-1}(-)\otimes_{f^{-1}\OO_V}\OO_U\colon\Mod{V}\to\Mod{U}$.
\end{lemma}

\begin{proof}
  First, note that $f^*$ sends free $\OO_V$-modules to free $\OO_U$-modules:
  \[
    f^*(\OO_V^r)
    \cong f^{-1}\OO_V^r\otimes_{f^{-1}\OO_V}\OO_U
    \cong \OO_U^r.
  \]
  Pseudofunctoriality follows from the fact that $(gf)^*\cong f^*g^*$ is a natural isomorphism but not necessarily an equality.
\end{proof}

The fact that $\Free$ only defines a \emph{pseudo}functor means that a diagram in $\RingSpace$ will \emph{not} give us a diagram in $\dgCat$ when we compose with $\Free$, but merely a ``pseudo-diagram'', and we cannot a priori take (homotopy) limits of such things.
One solution to this problem is via \emph{rectification}, which is a strictification procedure: \cite[Proposition~4.2]{BHW2017} shows that one can replace any pseudo-presheaf of dg-categories with an dg-equivalent strict presheaf.
However, for our purposes in this paper, we can make use of a much more elementary fact (which also appears as \cite[Remark~4.5]{BHW2017}), which is that evaluating $\Free$ specifically on the Čech nerve does result in a strict cosimplicial diagram since the coface maps are then given by restriction to open subsets.
Because of this, we will not worry about rectification here, but this disclaimer is important enough to merit a remark.

\begin{remark}
\label{remark:abuse-of-terminology-no-rectification}
  Since $\Free(\cechnerve\cover_\bullet)$ is a \emph{strict} cosimplicial dg-category, we do not need to first rectify $\Free$ and obtain a strict presheaf of dg-categories.
  Throughout this paper, since we only work with the Čech nerve, we will continue to refer to $\Free$ (the other ``presheaves'' that we define in \cref{section:three-simplicial-presheaves}) as a presheaf instead of a pseudo-presheaf, \textbf{but this really is an abuse of terminology}.
\end{remark}

Note, however, that this pseudo/strict distinction only really matters when considering $\Free$ applied to some diagram of ringed spaces --- whenever we talk about $\Free(U)$ for some fixed $(U,\OO_U)$, the issue of pseudofunctoriality does not arise.
It is also important to understand that $(-)^*$ is a map $\Hom_{\RingSpace}(U,V)\to[\Free(V),\Free(U)]$ that forms part of a \emph{pseudo}functor, whereas, for any specific $f\colon U\to V$, the map $f^*\colon\Mod{V}\to\Mod{U}$ is a \emph{strict} functor.

\begin{lemma}
\label{lemma:maxkan-nerve-free-is-presheaf}
  The assignment $(U,\OO_U)\mapsto\core{\nerve\Free(U)}$ defines a simplicial presheaf on $\RingSpace$.
\end{lemma}

\begin{proof}
  Let
  \[
    C_0 \xrightarrow{\varphi_1}
    C_1 \xrightarrow{\varphi_2}
    \ldots \xrightarrow{\varphi_p}
    C_p
  \]
  be a $p$-simplex in $\nerve\Free(V)$, so that each $C_i$ is a bounded chain complex of finite free $\OO_V$-modules, and each $\varphi_i$ is a chain map.
  Since $f^*$ is a functor from $\OO_V$-modules to $\OO_U$-modules, it gives objects $f^*C_0,\ldots,f^*C_p$ in $\Free(U)$, as well as degree-wise maps $f^*\varphi_1,\ldots\,f^*\varphi_p$, but we need to justify why these are indeed still chain maps in order to obtain a $p$-simplex in $\nerve\Free(U)$.
  However, this follows immediately from the functoriality of $f^*$, since functors preserve commutative squares, and so the $f^*\varphi_i$ are indeed chain maps.

  So we have a map on objects $f^*\colon\nerve\Free(V)\to\nerve\Free(U)$, but for this to be a morphism of simplicial sets we need to show that it commutes with the simplicial structure of the nerve.
  But since $f^*$ is a functor, it sends identities to identities and compositions to compositions, which means that it respects the face and degeneracy maps of the nerve, and thus indeed gives a morphism of simplicial sets $f^*\colon\nerve\Free(V)\to\nerve\Free(U)$.

  Finally, for $f^*$ to induce a morphism $\core{\nerve\Free(V)}\to\core{\nerve\Free(U)}$, we need to know that it sends isomorphisms to isomorphisms (since \cref{lemma:nerve-inside-dg-nerve}~(ii) tells us that this defines the maximal Kan complex of the ordinary nerve).
  But this is again immediate: any functor preserves isomorphisms, by functoriality.
\end{proof}

\subsection{Twisting cochains}
\label{subsection:twist-construction}

We are now in a position to define the first of our three simplicial presheaves, namely that of \emph{twisting cochains}.
This definition will recover (\cref{theorem:tot0-all-three}) the classical definition of twisting cochains as found throughout the work of Toledo--Tong \cite{TT1976,TT1978,TT1986} and O'Brian--Toledo--Tong \cite{OTT1981c,OTT1981b,OTT1981a,OTT1985}, as well as in more recent work \cite{Wei2016,BHW2017,Wei2021,GMTZ2022a,GMTZ2022b}.

\begin{definition}
\label{definition:twist}
  Define
  \[
    \Twist(U) = \core{\dgnerve\Free(U)}
  \]
  for any ringed space $(U,\OO_U)$.
  Note that this is, by definition, a simplicial set, and indeed even a Kan complex.
\end{definition}

\begin{lemma}
\label{lemma:twist-is-a-simplicial-presheaf}
  The assignment $(U,\OO_U)\mapsto\Twist(U)$ defines a simplicial presheaf on $\ConnRingSpace$.
\end{lemma}

\begin{proof}
  The proof of this statement is almost identical to that of \cref{lemma:maxkan-nerve-free-is-presheaf}, but we just need to modify the argument to account for the fact that we are now taking the dg-nerve instead of the ordinary nerve.

  First of all, note that $f^*$ does indeed induce a dg-functor $\Free(V)\to\Free(U)$, since (\cite[\href{https://stacks.math.columbia.edu/tag/09LB}{Tag 09LB}]{stacks-project}) it is an additive functor from $\OO_V$-modules to $\OO_U$-modules.
  Secondly, we need to know that $f^*$ sends quasi-isomorphisms to quasi-isomorphisms, which is equivalent to $f^*$ being exact\footnote{This is a ``standard'' fact, but we sketch a proof here for completeness. If $f^*$ preserves quasi-isomorphisms then it will in particular preserve the quasi-isomorphism between a short exact sequence and the zero complex, and so the image will also be quasi-isomorphic to zero, i.e. short exact; if $f^*$ is exact, then it preserves acyclic complexes, but a quasi-isomorphism is exactly a morphism whose mapping cone is acyclic, and since $f$ is additive we know that $f^*$ preserves mapping cones, and thus sends quasi-isomorphisms to quasi-isomorphisms.}, but this follows from the fact that we are working only with chain complexes of \emph{free} modules, which are, in particular, flat.\footnote{Since $\operatorname{Tor}$ is symmetric with respect to its two arguments, the functor $(-\otimes N)$ is exact if $N$ is flat but also if it is restricted to a full subcategory consisting of flat modules: given a short exact sequence $0\to A\to B\to C\to 0$, the associated long exact sequence is $\ldots\to\operatorname{Tor}_1(C,N)\to A\otimes N\to B\otimes N\to C\otimes N\to0$, and all the $\operatorname{Tor}$ terms vanish if all of $A$, $B$, and $C$ are flat.}
\end{proof}

\begin{remark}
  If we were not working with the dg-categories of complexes of \emph{free} modules, but instead arbitrary modules, then we would need to restrict the presheaf to the wide subcategory of (connected) ringed spaces with e.g. flat morphisms.
  The key point is that dg-functors do not a priori preserve quasi-isomorphisms, as one might hope.
\end{remark}

\subsection{Green complexes}
\label{subsection:green-construction}

\begin{definition}
\label{definition:elementary-complex}
  Let $R$ be a commutative ring.
  Given $R$-modules $M_1,\ldots,M_r$, we say that a complex of $R$-modules is \emph{$\{M_1,\ldots,M_n\}$-elementary} if it is a direct sum \mbox{$\epsilon_{p_1}^{i_1}\oplus\ldots\oplus\epsilon_{p_m}^{i_m}$} of complexes of the form
  \[
    \epsilon_{p_j}^{i_j} \coloneqq (0\to M_{i_j}\xrightarrow{\id}M_{i_j}\to0)[p_j]
  \]
  for some $p_1,p_2,\ldots,p_m\in\mathbb{Z}$.
  More generally, we simply say that a complex is \emph{elementary} if there exists some finite set of modules $\{M_1,\ldots,M_n\}$ for which it is $\{M_1,\ldots,M_n\}$-elementary.
\end{definition}

Note that the definition of elementary complexes extends immediately to complexes of sheaves of $\OO_X$-modules.

\begin{definition}
\label{definition:elementary-morphism}
  Given two elementary complexes, we define \emph{the elementary morphism} between them to be the (unique) morphism given by the ``maximal'' direct sum of elementary identity morphisms.
  That is, if $E$ is $\{M_1,\ldots,M_n\}$-elementary and $E'$ is $\{M'_1,\ldots,M'_{n'}\}$-elementary $E'$, then we can write
  \[
    E =\bigoplus_{j=1}^m \epsilon_{p_j}^{i_j}
    \quad\mbox{and}\quad
    E' =\bigoplus_{j=1}^{m'} {\epsilon'}_{p'_j}^{i'_j}
  \]
  and the elementary morphism from $E$ to $E'$, which we denote by $E\dashrightarrow E'$, is then defined to be
  \[
    \bigoplus_j\id_{\epsilon_{p_j}^{i_j}} \colon E \dashrightarrow E'
  \]
  where the sum is taken over all $j$ such that $\epsilon_{p_j}^{i_j}={\epsilon'}_{p'_j}^{i'_j}$.

  In the dg-category of complexes of modules, we define \emph{the elementary morphism of degree~$k$} to be exactly the elementary morphism defined above when $k=0$, and exactly the zero map when $k\neq0$.
  Concretely then, the elementary morphism is either zero or the inclusion into a direct sum (and thus, in particularly nice cases, the identity map).
\end{definition}

By construction, elementary complexes are acyclic.
Morally, we can think of them as the algebraic analogue of \emph{contractible} spaces.

\begin{lemma}
\label{lemma:adding-elementary-is-quasi-iso}
  Taking the direct sum with an elementary complex induces a quasi-isomorphism.
  More explicitly, if $C^\bullet$ is a complex of modules, and $E=(M\xrightarrow{\id}M)[0]$ is an elementary complex, then the inclusion
  \[
    i\colon C^\bullet \hookrightarrow C^\bullet\oplus E^\bullet
  \]
  is a chain homotopy equivalence (and thus, in particular, a quasi-isomorphism).
  The quasi-inverse is given by the projection
  \[
    p\colon C^\bullet\oplus E^\bullet \twoheadrightarrow C^\bullet.
  \]
\end{lemma}

\begin{proof}
  One composition $p\circ i\colon C^\bullet\to C^\bullet$ is the identity on the nose.
  The other composition $i\circ p\colon C^\bullet\oplus E^\bullet\to C^\bullet\oplus E^\bullet$ is homotopic to the identity, as witnessed by the homotopy that is zero in all degrees except for in degree~$0$, where it is $\left(\begin{smallmatrix}0&0\\0&-\id_M\end{smallmatrix}\right)$, i.e.
  \[
    \begin{tikzcd}[ampersand replacement=\&,column sep=huge,row sep=4em]
      \ldots \ar[r,"d_C"]
      \& C^2 \ar[r,"{\begin{pmatrix}d_C\\0\end{pmatrix}}"]
        \ar[dl,"0" description]
      \& C^1\oplus M \ar[r,"{\begin{pmatrix}d_C&0\\0&\id_M\end{pmatrix}}"]
        \ar[dl,"0" description]
      \& C^0\oplus M \ar[r,"{\begin{pmatrix}d_C&0\end{pmatrix}}"]
        \ar[dl,"{\begin{pmatrix}0&0\\0&-\id_M\end{pmatrix}}" description]
      \& C^{-1} \ar[r,"d_C"]
        \ar[dl,"0" description]
      \& \ldots
        \ar[dl,"0" description]
    \\\ldots \ar[r,swap,"d_C"]
      \& C^2 \ar[r,swap,"{\begin{pmatrix}d_C\\0\end{pmatrix}}"]
      \& C^1\oplus M \ar[r,swap,"{\begin{pmatrix}d_C&0\\0&\id_M\end{pmatrix}}"]
      \& C^0\oplus M \ar[r,swap,"{\begin{pmatrix}d_C&0\end{pmatrix}}"]
      \& C^{-1} \ar[r,swap,"d_C"]
      \& \ldots
    \end{tikzcd}
  \]
\end{proof}

We now give a fundamental definition, which will be used in constructing two of the three simplicial presheaves in which we are interested.
For want of a more descriptive name, we use the acronym ``GTT'' for ``Green--Toledo--Tong'' in reference to \cite{TT1986}.

\begin{definition}
\label{definition:GTT-labelling}
  A \emph{GTT-labelling of $\Delta[p]$ by $\core{\dgnerve\Free(U)}$} (or simply a \emph{GTT-labelling of $\Delta[p]$}) consists of a labelling of $\pair{\Delta[p]}$ (\cref{definition:pair-subdivision}) subject to some conditions.
  More precisely, it consists of the following data:
  \begin{enumerate}[(a)]
    \item To each $0$-cell
      \[
        (\sigma,\sigma)
        \quad\longleftrightarrow\quad
        \{i_0<\ldots<i_k\}\subseteq\{i_0<\ldots<i_k\}\subseteq[p]
      \]
      we assign a $k$-simplex of $\core{\dgnerve\Free(U)}$, i.e. bounded complexes of finite-rank free $\OO_U$-modules
      \[
        C_{i_0}(\sigma),\, C_{i_1}(\sigma),\, \ldots,\, C_{i_k}(\sigma)
        \in\Free(U)
      \]
      along with, for all non-empty subsets $J=\{j_0<\ldots<j_\ell\}\subseteq\{i_0<\ldots<i_k\}$, morphisms
      \[
        \varphi_J(\sigma)
        \in\Hom_{\Free(U)}^{1-\ell}\left(
          C_{j_\ell}(\sigma),C_{j_0}(\sigma)
        \right)
      \]
      such that
      \begin{itemize}
        \item if $|J|=1$, then $\varphi_J(\sigma)$ is the differential on $C_J(\sigma)$;
        \item if $|J|=2$, then $\varphi_J(\sigma)\colon C_{j_1}(\sigma)\to C_{j_0}(\sigma)$ is a chain map;
        \item if $|J|\geq3$, then
        \[
          \partial \varphi_J(\sigma)
          =
          \sum_{m=1}^{\ell-1} (-1)^{m-1} \varphi_{J\setminus\{j_m\}}(\sigma)
          +
          \sum_{m=1}^{\ell-1} (-1)^{\ell(m-1)+1} \varphi_{\{j_m<\ldots<j_k\}}(\sigma) \circ \varphi_{\{j_0<\ldots<j_m\}}(\sigma).
        \]
      \end{itemize}
    \item To each $(k-\ell)$-cell
      \[
        (\tau,\sigma)
        \quad\longleftrightarrow\quad
        \{j_0<\ldots<j_\ell\}\subset\{i_0<\ldots<i_k\}\subseteq[p]
      \]
      (where $0\leq l<k$) we assign an $(\ell+1)$-tuple of objects
      \[
        \left(
          C_{j_m}^{\perp\sigma}(\tau) \in \Free(U)
        \right)_{0\leq m\leq \ell}
      \]
      where each $C_{j_m}^{\perp\sigma}(\tau)$ is elementary.\footnote{See \cref{remark:elementary-complements-in-GTT}.}
      We refer to the $C_{j_m}^{\perp\sigma}(\tau)$ as the \emph{elementary (orthogonal) complements}.
  \end{enumerate}
  This data is subject to the conditions that, for any $(k-\ell)$-cell
  \[
    (\tau,\sigma)
    \quad\longleftrightarrow\quad
    \{j_0<\ldots<j_\ell\}\subset\{i_0<\ldots<i_k\}\subseteq[p]
  \]
  (where $0\leq l<k$) the following are satisfied:
  \begin{enumerate}[(i)]
    \item There is a direct-sum decomposition
      \[
        \theta_{j_m}^{\perp\sigma}(\tau)\colon
        C_{j_m}(\tau)\oplus C_{j_m}^{\perp\sigma}(\tau)
        \xrightarrow{\cong} C_{j_m}(\sigma)
      \]
      for all $0\leq m\leq \ell$.
      We refer to the isomorphism $\theta_{j_m}^{\perp\sigma}(\tau)$ as the \emph{$\tau$-trivialisation of $C_{j_m}(\sigma)$}.
    \item For any non-empty subset $K\subseteq\{j_0<\ldots<j_\ell\}$ with $|K|\geq2$, the diagram
      \[
        \begin{tikzcd}[sep=huge]
          C_{\ver_{|K|}K}(\tau)
            \ar[r,hook]
            \ar[d,swap,"\varphi_K(\tau)"]
          & C_{\ver_{|K|}K}(\tau)\oplus C_{\ver_{|K|}K}^{\perp\sigma}(\tau)
            \ar[r,two heads]
            \ar[d,"\varphi_K(\sigma)_\tau"]
          & C_{\ver_{|K|}K}^{\perp\sigma}(\tau)
            \ar[d,dashed]
        \\C_{\ver_0K}(\tau)
            \ar[r,hook]
          & C_{\ver_0K}(\tau)\oplus C_{\ver_0K}^{\perp\sigma}(\tau)
            \ar[r,two heads]
          & C_{\ver_0K}^{\perp\sigma}(\tau)
        \end{tikzcd}
      \]
      commutes, where the $\hookrightarrow$ are the inclusions and the $\twoheadrightarrow$ are the projections of the direct sum, the dashed arrow on the far right is the elementary morphism\footnote{Recall \cref{definition:elementary-morphism}: for $|K|\neq2$, this is zero; for $|K|=2$ (i.e. for $K=\{j_a<j_b\}$) this is a sum of identity maps. In the latter case, although $C_{j_a}^{\perp\sigma}(\tau)$ is elementary in $C_{j_a}(\sigma)$, and $C_{j_b}^{\perp\sigma}(\tau)$ is elementary in $C_{j_b}(\sigma)$, both $C_{j_a}(\sigma)$ and $C_{j_a}(\sigma)$ consist of free modules over the same ring, namely $\OO(U)$, and so the elementary morphism will ``often'' (i.e. in practice, when the GTT-labelling arises from the twisting cochain constructed from a coherent sheaf) be non-zero.} of degree~$(2-|K|)$, and we write $\varphi_K(\sigma)_\tau$ to mean the composition
      \[
        C_{\ver_{|K|}K}(\tau)\oplus C_{\ver_{|K|}K}^{\perp\sigma}(\tau)
        \xrightarrow{\theta_{\ver_{|K|}K}^{\perp\sigma}(\tau)} C_{\ver_{|K|}K}(\sigma)
        \xrightarrow{\varphi_K(\sigma)} C_{\ver_0K}(\sigma)
        \xrightarrow{\theta_{\ver_0K}^{\perp\sigma}(\tau)^{-1}} C_{\ver_0K}(\tau)\oplus C_{\ver_0K}^{\perp\sigma}(\tau)
      \]
      which we refer to as the \emph{$\tau$-trivialisation of $\varphi_K(\sigma)$}.
  \end{enumerate}
  \vspace{-3em}
\end{definition}

\begin{remark}
\label{remark:comments-on-GTT-definition}
  There are some important comments to make concerning \cref{definition:GTT-labelling}, which hopefully elucidate the rather opaque specificities.

  Firstly, the direct-sum decomposition in condition~(ii) is ``strict'', i.e. the morphism $C_{i_m}(\tau)\hookrightarrow C_{i_m}(\sigma)$ is exactly the inclusion into the direct sum, and not just some arbitrary monomorphism (and similarly for the projection $\twoheadrightarrow$ out of the direct sum).
  Another way of expressing this condition would be to ask for the $\varphi_K$ to induce a morphism of short exact sequences
  \[
    \begin{tikzcd}[sep=huge]
      0
        \ar[r]
      & C_{\ver_{|K|}K}(\tau)
        \ar[r,hook]
        \ar[d,swap,"\varphi_K(\tau)"]
      & C_{\ver_{|K|}K}(\sigma)
        \ar[r,two heads]
        \ar[d,"\varphi_K(\sigma)"]
      & C_{\ver_{|K|}K}^{\perp\sigma}(\tau)
        \ar[r]
        \ar[d,dashed]
      & 0
    \\0
        \ar[r]
      & C_{\ver_0K}(\tau)
        \ar[r,hook]
      & C_{\ver_0K}(\sigma)
        \ar[r,two heads]
      & C_{\ver_0K}^{\perp\sigma}(\tau)
        \ar[r]
      & 0
    \end{tikzcd}
  \]
  where now the middle vertical arrow is $\varphi_K(\sigma)$ instead of $\varphi_K(\sigma)_\tau$, and the horizontal arrows contain the composition with the $\theta_{j_m}^{\perp\sigma}(\tau)$.
  The moral reason for this condition is that we want for the homotopies on higher-dimensional faces to restrict down to agree exactly with those already present on the lower-dimensional faces, and to also restrict down to be exactly the elementary morphism on the elementary orthogonal complements.
  Alternatively, writing morphisms between direct sums in block matrix form, this condition says that, in the $\tau$-trivialisation,
  \[
    \varphi_K(\sigma)
    = \begin{pmatrix}
      \varphi_K(\tau) & *
    \\0 & e
    \end{pmatrix}
  \]
  where $e$ is the elementary morphism of degree~$(2-|K|)$, and $*$ is some arbitrary morphism $C_{\ver_{|K|}K}^{\perp\sigma}(\tau)\to C_{\ver_0K}(\tau)$.

  Secondly, it is tempting to try to include the $|K|=1$ case in condition~(ii), since this would seem to express the fact that the isomorphisms $\theta_{j_m}^{\perp\sigma}(\tau)$ from condition~(i) do indeed commute with the differentials, allowing us to weaken condition~(i) to simply ask for \emph{degree-wise} isomorphisms of the complexes.
  However, this isn't quite so simple, since we do not want the differential $\varphi_K(\sigma)$ (in the case $|K|=1$) to simply be upper triangular (in the block-matrix point of view described above) with respect to the differential $\varphi_K(\tau)$, but instead diagonal: the differential on $C_{j_m}(\sigma)$ should be exactly the direct sum of the differential of $C_{j_m}(\tau)$ with the differential of $C_{j_m}^{\perp\sigma}(\tau)$.

  Finally, note that we could remove the need for the data of the elementary complements entirely (thus giving only the data of $k$-simplices labelling the \emph{vertices} of $\pair{\Delta[p]}$) and rephrase condition~(i) entirely to require two things: that $C_{j_m}(\tau)\hookrightarrow C_{j_m}(\sigma)$ for all $0\leq m\leq\ell$; and that the cokernel of this morphism be elementary (and thus free, implying that the short exact sequence splits).
  This definition sounds much more concise, and could maybe even be expressed without reference to the pair subdivision at all, but instead as some sort of totalisation.
  But for the purposes of explicit calculation, we need a specific choice of cokernel and isomorphism with the direct-sum decomposition for each pair of faces $\tau\subset\sigma$, and these all need to be coherent with one another, whence the definition we give.

  In summary, there is some matter of taste in how one chooses to phrase this definition, and we have opted for the one that seems closest to that found in \cite{TT1986}, but we do not think that this is necessarily the most succinct one possible.
  Indeed, finding a better definition would potentially allow us to prove the full generalisation of \cref{corollary:filling-2-horns-in-green}.
\end{remark}

\begin{remark}
\label{remark:elementary-complements-in-GTT}
  In \cite[\S1.4]{Gre1980}, Green specifies the modules with respect to which the orthogonal complements in \cref{definition:GTT-labelling} are elementary; the restatement of this definition in \cite[Green's~Theorem~1]{TT1986}, as well as the very definition of simplicial twisting cochains \loccit, makes no reference to these modules.
  In practice, this makes no difference: the important fact is that the complexes are elementary with respect to \emph{something}, since being elementary implies being homotopically zero, irrespective of the choice of modules.
  Furthermore, being elementary ensures that one obtains a \emph{compatible sequence of connections} (\cite[\S4.5]{Hos2020a}).
  We \emph{could} say something more precise about how these complexes are elementary: one can ask for those corresponding to inclusions $\{i\}\subset\{i<j\}$ to be $\{C_i(i<j),C_j(i<j)\}$-elementary, and for all complements labelling higher-dimensional cells to be elementary with respect to both these and the modules constituting the target complexes (as is the case in the proof of \cref{lemma:filling-2-horns-in-stwist}, for example).
\end{remark}

There is a particular case of \cref{definition:GTT-labelling} which merits its own name.

\begin{definition}
\label{definition:GTT-1-labelling}
  If a GTT-labelling of $\Delta[p]$ by $\core{\dgnerve\Free(U)}$ is such that the $\varphi_J(\sigma)$ are zero for all $|J|\geq3$ then we call it a \emph{GTT-$1$-labelling}.
\end{definition}

\begin{remark}
\label{remark:GTT-1-labelling-equivalent-definition}
  Although the condition on $|J|\geq3$ in \cref{definition:GTT-1-labelling} might suggest that this should instead be called a ``3-labelling'', this is another occurrence of the surprisingly common ``+2'' degree shift that one sees in the theory of higher sheaves.\footnote{Another example: the condition defining an $n$-sheaf concerns the $(n+2)$-truncation of the Čech nerve.}

  By \cref{lemma:nerve-inside-dg-nerve}, asking for $\varphi_J(\sigma)$ to be zero for all $|J|\geq3$ in \cref{definition:GTT-1-labelling} is equivalent to asking that they lie inside the Kan sub-complex $\core{\nerve\Free(U)}\hookrightarrow\core{\dgnerve\Free(U)}$.
  In particular, the $\varphi_J(\sigma)$ for $|J|=2$ are then isomorphisms, i.e. the $k$-simplex assigned to any $0$-cell $\sigma$ in $\pair{\Delta[p]}$ is of the form
  \[
    C_{i_0}(\sigma) \cong C_{i_1}(\sigma) \cong \ldots \cong C_{i_k}(\sigma)
  \]
  with $\varphi_{i_j<i_{j+1}}(\sigma)$ being the isomorphism $C_{i_{j+1}}(\sigma)\to C_{i_j}(\sigma)$.
  Because of this, we may also refer to a GTT-$1$-labelling as a \emph{GTT-labelling by $\core{\nerve\Free(U)}$}, which is purely 1-categorical (whence the terminology of GTT-1-labelling).
\end{remark}

\begin{definition}
\label{definition:green}
  Define
  \[
    \Green(U)_p \coloneqq
    \Big\{
      \mbox{GTT-$1$-labellings of $\Delta[p]$}
    \Big\}
  \]
  for any ringed space $(U,\OO_U)$ and any integer $p\geq0$.
\end{definition}

\begin{remark}
  It is a sign of rather poorly chosen notation that the left-hand side of \cref{definition:green} has an explicit dependence on $U$ whereas the right-hand side does not.
  However, we do not intend for the terminology of GTT-labellings to be adopted beyond this paper.
  Indeed, this method of labelling the pair subdivision provides an incredibly concrete hands-on way of constructing these objects, but one would hope that there is a cleaner, more conceptual definition.
\end{remark}

\begin{lemma}
\label{lemma:green-is-a-simplicial-set}
  The set $\{\Green(U)_p\}_{p\in\mathbb{N}}$ can be given the structure of a simplicial set, such that the face maps are given by restricting labellings and the degeneracy maps are given by adding trivial data.
\end{lemma}

Before giving the detailed proof, let us give a sketch.
We need to construct face and degeneracy maps for $\Green(U)_\bullet$ and show that they satisfy the simplicial identities.
The maps that we construct will be induced by the coface and codegeneracies of the standard simplices $\Delta[p]$ (which, recalling \cref{subsection:simplicial-sets-and-spaces}, form a \emph{cosimplicial} space).
The underlying idea of constructing these maps is relatively simple: morphisms $[p-1]\to[p]$ correspond to inclusions $\Delta[p-1]\hookrightarrow\Delta[p]$ of standard simplices, and this induces $\Green(U)_p\to\Green(U)_{p-1}$ by \emph{restricting} (or \emph{forgetting}) labellings, as in \cref{figure:simplicial-structure-of-green-face}; morphisms $[p+1]\to[p]$ correspond to ``collapses'' $\Delta[p+1]\twoheadrightarrow\Delta[p]$ of standard simplices, and this induces $\Green(U)_p\to\Green(U)_{p+1}$ by \emph{``adding trivial data''} to the degenerate faces of $\Delta[p+1]$, as in \cref{figure:simplicial-structure-of-green-degeneracy}.
We now give the full technical details.

\begin{proof}
  First, say we have some GTT-$1$-labelling of $\Delta[p]$ by $\core{\nerve\Free(U)}$, and consider a coface map $f_p^i\colon[p-1]\to[p]$ in $\Delta$.
  The corresponding coface map $\Delta[p-1]\to\Delta[p]$, which is simply the inclusion of a codimension-$1$ face, defines a labelling of $\Delta[p-1]$ by restriction (or ``forgetting'').
  For example, given a labelling $\cal{L}$ of $\Delta[3]$, we define $f_3^1(\cal{L})$ to be the labelling of $\Delta[p-1]$ given by the labelling of the face $f_3^1(\Delta[p-1])\subset\Delta[p]$, which is a (strict) subset of the data described by $\cal{L}$.

  Although restriction defines an ``obvious'' face map on labellings, we need to check that the conditions of \cref{definition:GTT-1-labelling} (and thus also \cref{definition:GTT-labelling}) are still satisfied in order for this face map to actually land in $\Green(U)_{p-1}$.
  But note that conditions~(i) and (ii) refer to the data that labels the $0$-cells, and so when we remove a single $0$-cell and all of the cells that contain it (which is exactly what restriction does), none of the remaining higher-dimensional cells will ``see'' this missing data.

  Now, a codegeneracy map $s_i^p\colon[p+1]\to[p]$ in $\Delta$ gives a geometric degeneracy map $\Delta[p]\to\Delta[p+1]$ that ``repeats'' the $i$th vertex, so we obtain a \emph{partial} labelling of $\Delta[p+1]$ by simply using our existing labelling of $\Delta[p]$.
  The vertices that we have left to label are exactly those corresponding to simplices of $\Delta[p+1]$ that contain the new copy of the $i$th vertex;
  the higher-dimensional cells that we have left to label are exactly those corresponding to pairs $(\tau,\sigma)$ where either $\sigma$ or $\tau$ contains the new copy of the $i$th vertex.

  We number the vertices of $\Delta[p+1]$ with $\{0<1<\ldots<i-1<i<i'<i+1<\ldots<p\}$, where $i'$ is the new copy of the $i$th vertex.
  Given any simplex in $\Delta[p+1]$ containing $i'$ but \emph{not containing} $i$, we label the corresponding vertex in $\pair{\Delta[p+1]}$ with whatever labels the already-labelled vertex corresponding to the simplex in $\Delta[p+1]$ that has the same vertices but with $i$ replacing $i'$.
  In particular, the vertex corresponding to $\{i'\}$ is labelled with the same data as what already labels the vertex corresponding to $\{i\}$.
  We label the vertices corresponding to simplices containing both $i$ and $i'$ by simply inserting an identity morphism.
  For example, since the vertices corresponding to $\{i\}$ and to $\{i'\}$ are labelled identically, say with the complex $C$, we label the vertex corresponding to the $1$-simplex $\{i<i'\}$ with the identity morphism $C\xleftarrow{\id} C$.
  Then, for any $n$-simplex $\sigma$ in $\Delta[p+1]$ of the form $\{j_0<\ldots<i<i'<\ldots<j_{n-2}\}$ with $n\geq2$, the vertex in $\pair{\Delta[p+1]}$ corresponding to the $(n-1)$-simplex $\{j_0<\ldots<i<\ldots<j_{n-2}\}$ is already labelled with some
  \[
    C_{j_0}\from\ldots\from C_i\from\ldots\from C_{j_{n-2}}
  \]
  and so we label the vertex in $\pair{\Delta[p+1]}$ corresponding to $\sigma$ with
  \[
    C_{j_0}\from\ldots\from C_i\xleftarrow{\id}C_i\from\ldots\from C_{j_{n-2}}.
  \]

  Now for the $n$-cells of $\pair{\Delta[p+1]}$ for $n\geq1$, which correspond to codimension-$n$ inclusions $\tau\subset\sigma$ of simplices in $\Delta[p+1]$.
  If $\sigma$ contains $i'$ but \emph{does not contain} $i$, then we label the corresponding cell in $\pair{\Delta[p+1]}$ with whatever labels the already-labelled cell corresponding to the pair $(\widetilde{\tau},\widetilde{\sigma})$ where we replace $i'$ by $i$.
  Otherwise, if $i'\not\in\tau$ but $i'\in\sigma$, then we know that the vertex corresponding to $\sigma$ is labelled in such a way that the face corresponding to $\tau$ of the element of the nerve is exactly what labels the vertex corresponding to $\tau$, and so we can simply label $(\tau,\sigma)$ with all zero objects.
  Finally, if $i'\in\tau$, then $i'\in\sigma$, and we label the cell $(\tau,\sigma)$ in $\pair{\Delta[p+1]}$ with whatever already labels the cell $(\widehat{\tau},\widehat{\sigma})$, where $\widehat{\tau}\coloneqq\tau\setminus\{i'\}$ and $\widehat{\sigma}\coloneqq\sigma\setminus\{i'\}$, with one complement corresponding to $i$ duplicated for $i'$, i.e.
  \[
    C_j^{\perp\sigma}(\tau)\coloneqq
    \begin{cases}
      C_j^{\perp\widehat{\sigma}}(\widehat{\tau})
      &\mbox{if $j\neq i'$};
    \\C_i^{\perp\widehat{\sigma}}(\widehat{\tau})
      &\mbox{if $j=i'$}.
    \end{cases}
  \]

  It remains only to check that the conditions of \cref{definition:GTT-labelling} are satisfied by this new labelling.
  By construction, we only need to check conditions~(i) and (ii) for the $n$-cells that contain both $i$ and $i'$.
  If $i'\not\in\tau$ but $i'\in\sigma$, then the elementary complements are all zero, and so both conditions hold immediately;
  if $i'\in\tau$ then the elementary complements are exactly those that come from the original GTT-labelling of $\Delta[p]$, and so the conditions hold by hypothesis.
\end{proof}

\begin{figure}[!ht]
  \centering
    \begin{tikzpicture}[scale=3.5]
      \node (0) at (0,0) {$C_0(0)$};
      \node (1) at (1,1.5) {$C_1(1)$};
      \node [vertex=white] (2) at (2,0) {};
      \draw [white] (0) to node [black] (01) {\footnotesize$C_0(0<1)\xleftarrow{\varphi_{0<1}(0<1)} C_1(0<1)$} (1);
      \draw [edge,dashed] (0) to (2);
      \node [vertex=white] (02) at (1,0) {};
      \draw [edge,dashed] (1) to node (12) [vertex=white] {} (2);
      \node [vertex=white] (012) at (1,0.5) {};
      \draw [edge,dashed] (01) to (012);
      \draw [edge,dashed] (12) to (012);
      \draw [edge,dashed] (02) to (012);
      \draw [thick,-latex] (0) to node[fill=white]{\footnotesize$C_0^{\perp\{0<1\}}(0)$} (01);
      \draw [thick,-latex] (1) to node[fill=white]{\footnotesize$C_1^{\perp\{0<1\}}(1)$} (01);
    \end{tikzpicture}
  \caption{The image of a $2$-simplex in $\Green(U)$ under the face map $f_2^2$, given simply by forgetting the labelling outside of the face $\{0<1\}$.}
  \label{figure:simplicial-structure-of-green-face}
\end{figure}

\begin{figure}[!ht]
  \centering
  \begin{tikzpicture}[xscale=6,yscale=5.5]
    \node (0) at (210:1) {$C_0(0)$};
    \node (1) at (90:1) {$C_1(1)$};
    \node (2) at (330:1) {$C_1(1)$};
    \node (01) at (150:1) {\footnotesize$C_0(0<1)\xleftarrow{\varphi_{0<1}(0<1)} C_1(0<1)$};
    \node (12) at (30:1) {\footnotesize$C_1(1)\xleftarrow{\id}C_1(1)$};
    \node (02) at (270:1) {\footnotesize$C_0(0<1)\xleftarrow{\varphi_{0<1}(0<1)}C_1(0<1)$};
    \node (012) at (0,0) {\footnotesize$\begin{tikzcd}[row sep=6em,column sep=0em]&C_1(0<1)\ar[dl,swap,"\varphi_{0<1}(0<1)"]\\C_0(0<1)&&C_1(0<1)\ar[ul,swap,"\id"]\ar[ll,"\varphi_{0<1}(0<1)"]\end{tikzcd}$};
    \draw [thick,-latex] (0) to node[fill=white]{\footnotesize$C_0^{\perp\{0<1\}}(0)$} (01);
    \draw [thick,-latex] (0) to node[fill=white]{\footnotesize$C_0^{\perp\{0<1\}}(0)$} (02);
    \draw [thick,-latex] (1) to node[fill=white]{\footnotesize$C_1^{\perp\{0<1\}}(1)$} (01);
    \draw [thick,-latex] (1) to node[fill=white]{\footnotesize$0$} (12);
    \draw [thick,-latex] (2) to node[fill=white]{\footnotesize$0$} (12);
    \draw [thick,-latex] (2) to node[fill=white]{\footnotesize$C_1^{\perp\{0<1\}}(1)$} (02);
    \draw [thick,-latex,shorten >=-3em] (01) to node[fill=white]{\footnotesize$(0,0)$} (012);
    \draw [thick,-latex,shorten >=-3em] (12) to node[fill=white]{\footnotesize$\big(C_1^{\perp\{0<1\}}(1),C_1^{\perp\{0<1\}}(1)\big)$} (012);
    \draw [thick,-latex,shorten >=1em] (02) to node[fill=white]{\footnotesize$(0,0)$} (012);
    \draw [double,double equal sign distance,-implies,shorten <=1em,shorten >=0.5em] (0) to node[fill=white]{\footnotesize$C_0^{\perp\{0<1\}}(0)$} (012);
    \draw [double,double equal sign distance,-implies,shorten <=2em,shorten >=1em] (1) to node[fill=white]{\footnotesize$C_1^{\perp\{0<1\}}(0)$} (012);
    \draw [double,double equal sign distance,-implies,shorten <=1em,shorten >=0.5em] (2) to node[fill=white]{\footnotesize$C_1^{\perp\{0<1\}}(0)$} (012);
  \end{tikzpicture}
  \caption{The image of a $1$-simplex in $\Green(U)$ under the degeneracy map $s_1^1$, given by placing the $1$-simplex along the $\{0<1\}$ edge and the $\{0<2\}$ edge, and then filling in the rest with homotopically trivial data. To make the diagram easier to read, we ``inflate'' the triangle into a hexagon. }
  \label{figure:simplicial-structure-of-green-degeneracy}
\end{figure}

\begin{lemma}
\label{lemma:green-is-a-simplicial-presheaf}
  The assignment $(U,\OO_U)\mapsto\Green(U)_\bullet$ defines a simplicial presheaf on $\ConnRingSpace$.
\end{lemma}

\begin{proof}
  Again, the majority of this proof is identical to the proof of \cref{lemma:maxkan-nerve-free-is-presheaf}, and the only thing that we need to prove here is that the GTT-$1$-labelling conditions are preserved by $f^*$.
  But this is immediate, since the only conditions of \cref{definition:GTT-1-labelling} are that certain direct sum decompositions exist and that two squares commute, and $f^*$ commutes with direct sums (since it is an adjoint) and sends commutative squares to commutative squares (since it is a functor).
\end{proof}

\subsection{Simplicial twisting cochains}
\label{subsection:stwist-construction}

\begin{definition}
\label{definition:stwist}
  Define
  \[
    \sTwist(U)_p \coloneqq
    \Big\{
      \mbox{GTT-labellings of $\Delta[p]$}
    \Big\}
  \]
  for any ringed space $(U,\OO_U)$ and any integer $p\geq0$.
\end{definition}

\begin{remark}
  Comparing \cref{definition:stwist} to \cref{definition:green} and \cref{definition:twist}, it might be surprising that we write $\sTwist$ instead of $\sGreen$, since the definition seems much closer to that of $\Green$ than of $\Twist$.
  We justify this in \cref{remark:green-as-a-specific-case-of-stwist}.
\end{remark}

\begin{lemma}
  The assignment $(U,\OO_U)\mapsto\sTwist(U)_\bullet$ defines a simplicial presheaf on $\ConnRingSpace$.
\end{lemma}

\begin{proof}
  The fact that $\sTwist(U)_\bullet$ has the structure of a simplicial set is almost the same as the proof of \cref{lemma:green-is-a-simplicial-set}.
  The only real difference is in constructing the degeneracy maps, since we need to label vertices of $\pair{\Delta[p+1]}$ with elements of the dg-nerve instead of the regular nerve.
  However, since all the faces of any element of the nerve commute on the nose, we can enrich to obtain an element of the dg-nerve by simply adding identity homotopies everywhere necessary.
  In other words, the degeneracy maps actually land in the subset $\Green(U)_{p+1}\subset\sTwist(U)_{p+1}$.

  Showing functoriality is then exactly the same as in \cref{lemma:green-is-a-simplicial-presheaf}, since the higher conditions necessary to be a full GTT-labelling (rather than just a GTT-$1$-labelling) just posit that yet more squares commute, and $f^*$ sends commutative squares to commutative squares.
\end{proof}

\begin{remark}
\label{remark:green-as-a-specific-case-of-stwist}
  Just as $\Green$ was defined exactly to recover the definition of complexes of ``simplicial vector bundles'' (satisfying the conditions arising from Green's resolution) given in \cite[\S1]{TT1986}, we have defined $\sTwist$ exactly to recover the definition of simplicial twisting cochains given in \cite[\S3]{TT1986}, as we will prove in \cref{theorem:tot0-all-three}.

  We know, by \cref{lemma:nerve-inside-dg-nerve}, how the nerve sits inside the dg-nerve (and similarly for their maximal Kan complexes), and this corresponds exactly to the explanation that ``\emph{complexes of simplicial vector bundles are specific examples of simplicial twisting cochains}'' given in \cite[p.~269]{TT1986}:
  \begin{enumerate}[1)]
    \item the fact that the $f_{\{0<1\}}$ are isomorphisms (\cref{lemma:nerve-inside-dg-nerve}~(ii)) means that
      \begin{quote}
        \emph{$E_{\sigma,\alpha}^\bullet$ is independent of $\alpha$}
      \end{quote}
      since all the complexes are isomorphic; and
    \item the fact that the $f_I$ are zero for $I\geq3$ (\cref{lemma:nerve-inside-dg-nerve}~(i)) means that
      \begin{quote}
        \emph{${}^\sigma \mathrm{a}^{k,1-k}=0$ for $k>1$}.
      \end{quote}
  \end{enumerate}
  This suggests the alternative nomenclature ``\emph{dg-Green complex}'' to mean ``simplicial twisting cochain''; conversely, we might suggest ``\emph{homotopy-truncated simplicial twisting cochain}'' to mean ``Green complex''.
\end{remark}

\section{Complex-analytic examples}
\label{section:complex-analytic-examples}

To justify our interest in the three simplicial presheaves $\Twist$, $\Green$, and $\sTwist$, we now show that, in the setting of complex-analytic manifolds, we recover well-known objects (and, in the case of $\Twist$, well-known paths between them).
This means that we now turn our study towards the Čech totalisation of these presheaves in the case where we have a good Stein cover of a connected complex-analytic manifold $X$, with $\OO_X$ the sheaf of holomorphic functions.
After this section, we will return to the more general study of the three simplicial presheaves before Čech totalisation.

In some sense, the main theorem of this paper (namely \cref{theorem:tot0-all-three}) is a purely technical lemma of combinatorics and higher category theory, requiring no real input from complex geometry.
Although it might be seen as requiring the reader to be familiar with the definitions that it is drawing together (from \cite{TT1986,Gre1980,Hos2020a}), it is also intended to be useful in the other direction, as more of a black box that one can treat as ``the'' primitive definition of twisting cochains, simplicial twisting cochains, and Green complexes, without asking for any sort of complex-geometric intuition.

\subsection{Points in all three}
\label{subsection:vertices-in-complex-analytic-examples}

\begin{theorem}
\label{theorem:tot0-all-three}
  Let $\underline{X}=(X,\cover)$, with $X$ a connected complex-analytic manifold with the structure sheaf $\OO_X$ of holomorphic functions, and $\cover$ a Stein cover.\footnote{For the statement of the theorem we do not necessarily need the cover to be Stein, but any open cover can always be refined to a Stein cover (and, more precisely, one whose intersections are also all Stein) so there is no loss in generality in assuming this, and this will be necessary if one wants to talk about coherent analytic sheaves.}
  Then
  \begin{theoremenum}
    \item\label{theorem:tot0-free-is-loc-free}
      $\Tot^0\core{\nerve\Free(\cechnerve\cover_\bullet)}$ is the set of bounded complexes of locally free sheaves\footnote{Since here $(X,\OO_X)$ is \emph{locally} ringed, these are exactly the strictly perfect complexes, as mentioned in \cref{remark:strictly-perfect-locally-ringed}.} on $(X,\cover)$;
    \item\label{theorem:tot0-twist-is-twcs}
      $\Tot^0\Twist(\cechnerve\cover_\bullet)$ is the set of twisting cochains \cite[\S3]{TT1986} on $(X,\cover)$;
    \item\label{theorem:tot0-green-is-green-complexes}
      $\Tot^0\Green(\cechnerve\cover_\bullet)$ is the set of complexes of ``simplicial vector bundles'' \cite[\S1]{TT1986} satisfying the conditions necessary for them to be a Green complex \cite{Hos2020a} on $(X,\cover)$;
    \item\label{theorem:tot0-stwist-is-simplicial-twcs}
      $\Tot^0\sTwist(\cechnerve\cover_\bullet)$ is the set of simplicial twisting cochains \cite[\S3]{TT1986} on $(X,\cover)$.
  \end{theoremenum}
\end{theorem}

\begin{proof}
  The proof of these theorems is generally nothing more than ``by construction'' --- the simplicial presheaves $\Green$, $\Twist$, and $\sTwist$ were defined exactly so that these results would hold.
  Here we will assume some familiarity with $\Tot$ calculations; we refer the interested reader to \cref{appendix:example} or \cref{appendix:proof-of-tot0-stwist-is-simplicial-twcs} for examples of worked calculations with more detail, or to \cite[Appendix~B]{GMTZ2022a} for a more general formal discussion.

  \begin{enumerate}[(a)]
    \item
      Since $\OO(\coprod_\alpha U_\alpha)\cong\prod_\alpha\OO(U_\alpha)$, a free $\OO(\coprod_\alpha U_\alpha)$-module is exactly the data of a free $\OO(U_\alpha)$-module for all $\alpha$.
      Since the dg-nerve and the core functor are both right adjoints, their composition preserves all limits.
      This means that
      \[
        \core{\nerve\Free(\coprod_\alpha U_\alpha)}_0\cong\prod_\alpha\core{\nerve\Free(U_\alpha)}
      \]
      Thus the data of a $0$-simplex in $\Tot\core{\nerve\Free(\cechnerve\cover_\bullet)}$ is exactly a cohesive choice of
      \begin{itemize}
        \item $E_\alpha^\bullet\in\core{\nerve\Free(U_\alpha)}_0=\Free(U_\alpha)$ for all $U_\alpha\in\cover$; and
        \item $\varphi_{\alpha_0\ldots\alpha_p}\in\core{\nerve\Free(U_{\alpha_0\ldots\alpha_p})}_p$ for all $U_{\alpha_0\ldots\alpha_p}$
      \end{itemize}
      where ``cohesive'' means that, in particular, the endpoints of $\varphi_{\alpha\beta}$ are exactly $E_\alpha^\bullet$ and $E_\beta^\bullet$, and the boundary of $\varphi_{\alpha\beta\gamma}$ consists exactly of $\varphi_{\alpha\beta}$, $\varphi_{\beta\gamma}$, and $\varphi_{\alpha\gamma}$; there are analogous conditions coming from the degeneracy maps.

      The $1$-dimensional face conditions first tell us that $\varphi_{\alpha\beta}\colon E_\beta^\bullet|U_{\alpha\beta}\to E_\alpha^\bullet|U_{\alpha\beta}$; the fact that $\varphi_{\alpha\beta}$ is an element of the maximal Kan complex of the nerve tells us that it is an isomorphism.
      The $2$-dimensional face conditions tell us that $\varphi_{\beta\gamma}\circ\varphi_{\alpha\beta}=\varphi_{\alpha\gamma}$.
      Since the ordinary nerve is generated by its $1$-simplices, and since composition of chain maps is strictly associative, the higher-dimensional face conditions impose no further restriction nor give any further information;
      in fact, this also tells us that the $\varphi_{\alpha_0\ldots\alpha_p}$ are simply sequences of $(p+1)$-many composible morphisms, and the face conditions tell us that they are given exactly by the $\varphi_{\alpha_i\alpha_{i+1}}$.
      The degeneracy conditions tell us that $\varphi_{\alpha\alpha}=\id_{E_\alpha^\bullet}$ since the only degenerate simplices in the nerve are those given by identity morphisms.

      In summary then, we have bounded complexes $E_\alpha^\bullet$ of free sheaves for each $U_\alpha$, along with isomorphisms $\varphi_{\alpha\beta}\colon E_\beta^\bullet|U_{\alpha\beta}\to E_\alpha^\bullet|U_{\alpha\beta}$ on each overlap $U_{\alpha\beta}$, satisfying the cocycle condition $\varphi_{\beta\gamma}\circ\varphi_{\alpha\beta}=\varphi_{\alpha\gamma}$ and the degeneracy condition $\varphi_{\alpha\alpha}=\id_{E_\alpha^\bullet}$.
      This is exactly the data of a bounded complex of locally free sheaves.

    \item Morally, the argument here starts identically to that of \cref{theorem:tot0-free-is-loc-free}, in that the data of a $0$-simplex in $\Tot\core{\dgnerve\Free(\cechnerve\cover_\bullet)}$ is exactly a cohesive choice of
      \begin{itemize}
        \item $E_\alpha^\bullet\in\core{\dgnerve\Free(U_\alpha)}_0$ for all $U_\alpha\in\cover$; and
        \item $\varphi_{\alpha_0\ldots\alpha_p}\in\core{\dgnerve\Free(U_{\alpha_0\ldots\alpha_p})}_p$ for all $U_{\alpha_0\ldots\alpha_p}$
      \end{itemize}
      but since we are working with the dg-nerve instead of the ordinary nerve, all $\varphi_{\alpha_0\ldots\alpha_p}$ will be relevant, not just the $\varphi_{\alpha\beta}$.

      As before, the $1$-dimensional face conditions tell us that $\varphi_{\alpha\beta}\colon E_\beta^\bullet|U_{\alpha\beta}\to E_\alpha^\bullet|U_{\alpha\beta}$; the fact that $\varphi_{\alpha\beta}$ is an element of the maximal Kan complex of the dg-nerve tells us that it is a \emph{quasi}-isomorphism.
      But now it is \emph{not} the case that these satisfy the cocycle condition: the $2$-dimensional face conditions tell us that $\varphi_{\alpha\beta\gamma}\colon\varphi_{\alpha\gamma}\Rightarrow\varphi_{\beta\gamma}\circ\varphi_{\alpha\beta}$ is a (possibly non-trivial) chain homotopy.
      As a specific example of this, $\varphi_{\alpha\beta\alpha}$ and $\varphi_{\beta\alpha\beta}$ are the chain homotopies witnessing that $\varphi_{\alpha\beta}$ is a quasi-isomorphism with quasi-inverse $\varphi_{\beta\alpha}$.
      More generally, the $p$-dimensional face conditions describe homotopies controlling the $(p-1)$-dimensional elements.
      As described in \cite[(8.2.7)]{Hos2020}, this seems to correspond exactly to the data of a twisting cochain --- an idea that this theorem now makes precise.

      To prove this formally, note that $\scr{D}=\Free$ satisfies the hypothesis of \cref{theorem:dg-nerve-and-MC-for-presheaves}, in that it turns disjoint unions of spaces into products\footnote{An isomorphism of ringed spaces $f\colon(X,\OO_X)\cong(Y,\OO_Y)$ induces an equivalence of categories $f^*\colon\Free(Y)\simeq\Free(X)$, since $f^*\circ(f^{-1})^*\cong(f^{-1}\circ f)^*\cong\id_{\Free(X)}$, and this is bijective on objects since we are working only with free modules, which are uniquely determined by their rank, and this is preserved by pullback: $f^*\OO_Y^r\coloneqq f^{-1}\OO_Y^r\otimes_{f^{-1}\OO_Y}\OO_X\cong\OO_X^r$. Then we just need to show that $\OO_{(-)}$ turns disjoint unions into products, but in the case of the sheaf of holomorphic functions this follows from the fact that it is representable, with $\OO_{(-)}\cong\Hom(-,\mathbb{C})$.}, and so elements of $\Tot^0\dgnerve\Free(\cechnerve\cover_\bullet)$ are exactly Maurer--Cartan elements in the corresponding Čech algebra.
      Then we simply appeal to \cref{corollary:morphism-into-dg-nerve-maximal-kan}, which tells us that elements of $\Tot^0\core{\dgnerve\Free(\cechnerve\cover_\bullet)}$ are exactly those Maurer--Cartan elements whose $1$-simplices are quasi-isomorphisms, and these are exactly twisting cochains.

      To better explain this, it may be helpful to spell out the details of what happens in degree~$2$.
      By definition, a $0$-simplex of $\Tot\core{\dgnerve\Free(\cechnerve\cover_\bullet)}$ is exactly a morphism of cosimplicial simplicial sets
      \[
        \Delta[\anotherbullet] \to \core{\dgnerve\Free(\cechnerve\cover_\anotherbullet)}
      \]
      which is exactly a functorial collection of morphisms of simplicial sets
      \[
        \begin{aligned}
          \Delta[0]
          &\to \Free(\cechnerve\cover_0) \cong \prod_\alpha\Free(U_\alpha)
        \\\Delta[1]
          &\to \core{\dgnerve\Free(\cechnerve\cover_1)}_1 \cong \prod_{\alpha\beta}\core{\dgnerve\Free(U_{\alpha\beta})}
        \\\Delta[2]
          &\to \core{\dgnerve\Free(\cechnerve\cover_2)}_1 \cong \prod_{\alpha\beta\gamma}\core{\dgnerve\Free(U_{\alpha\beta\gamma})}
        \end{aligned}
      \]
      and so on, where we again appeal to the fact that both the dg-nerve and the core functor are right adjoints (as explained in \cref{appendix:proof-of-tot0-stwist-is-simplicial-twcs}).
      This data forms a Čech bialgebra, generalising \cref{definition:bigraded-dg-algebra} to a presheaf of dg-categories instead of a single fixed dg-category, though this relies on the fact that the simplicial set in question is exactly the Čech nerve (see \cref{definition:bigraded-dg-algebra-for-presheaves} for the precise definition).
      In particular, we can consider the bidegree-$(2,0)$ parts of a $0$-simplex of this totalisation.
      The image of $\Delta[2]$ will give us one: some homotopy $f_{xyz}$ fitting into a diagram of the form
      \[
        \begin{tikzcd}[sep=large]
          & y \\
          x && z
          \arrow["{f_{xy}}"', from=1-2, to=2-1]
          \arrow["{f_{yz}}"', from=2-3, to=1-2]
          \arrow[""{name=0, anchor=center, inner sep=0}, "{f_{xz}}", from=2-3, to=2-1]
          \arrow["{f_{xyz}}"', shorten <=2pt, Rightarrow, from=1-2, to=0]
        \end{tikzcd}
      \]
      i.e. a map $f_{xyz}:z\to x$ of degree~$-1$ such that $\partial f_{xyz}=f_{xz}-f_{yz}\circ f_{xy}$, defined over some $U_{\alpha\beta\gamma}$, with $x\in\Free(U_\alpha)$, $y\in\Free(U_\beta)$, and $z\in\Free(U_\gamma)$.
      There are two other bidegree-$(2,0)$ terms that arise when we apply the two differentials: writing $f=(f_{yz},f_{xz},f_{xy})$, they are $\hat{\delta}f$ and $f\cdot f$, defined by
      \[
        \begin{aligned}
          (\hat{\delta}f)_{xyz}
          &= f_{xz}
        \\(f\cdot f)_{xyz}
          &= f_{xy}\cdot f_{yz}
          \coloneqq f_{yz}\circ f_{xy}
        \end{aligned}
      \]
      (and here we are using functoriality of this collection of morphisms of simplicial sets in the definition of $\hat{\delta}$, for example).
      But then the defining equation for $\partial f$ tells us exactly that the Maurer--Cartan equation is satisfied in bidegree~$(2,0)$, since
      \[
        \underbrace{(\partial-\hat{\delta})f}_{\eqqcolon\,\,\operatorname{D}f} + f\cdot f=0.
      \]

      One important thing to mention is that the definition of twisting cochain that we recover from this construction is indeed the ``classical'' one --- we expand upon this comment in \cref{remark:equal-vs-chain-homotopic-to-identity}.

    \item This will follow immediately from \cref{theorem:tot0-stwist-is-simplicial-twcs} combined with \cref{remark:green-as-a-specific-case-of-stwist}.

    \item This proof consists solely of unravelling definitions; we spell out the full details in \cref{appendix:proof-of-tot0-stwist-is-simplicial-twcs}.
      The intuition is that, in a GTT-labelling, after Čech totalisation, the \emph{vertices} of $\pair{\Delta[p]}$ are labelled exactly with a twisting cochain;
      the \emph{edges} (and higher-dimensional cells) are labelled with the extra data of the elementary orthogonal complements;
      the extra conditions describe how different twisting cochains ${}^\sigma\mathfrak{a}$ and ${}^\tau\mathfrak{a}$ (in the notation of \cite{TT1986}) fit together in a compatible manner.
  \end{enumerate}
  \vspace{-2em}
\end{proof}

The natural question to ask next is what the $\Tot^1$ analogue of \cref{theorem:tot0-all-three} is, or to ask what the $\pi_0$ of these three simplicial presheaves are.
In the rest of this section, we give some partial and some complete answers to these questions.

Of course, it would be satisfying to have results pertaining to the higher-dimensional simplices (or homotopy groups) as well, but this lies beyond the scope of this paper.

\subsection{Edges in twisting cochains}
\label{subsection:1-simplices-in-complex-twist}

Twisting cochains have been well studied, especially in the language of dg-categories.
This means that there are, for example, definitions of morphisms and weak equivalences of twisting cochains that can be found elsewhere.
We can now study how our construction of twisting cochains via $\Tot\Twist(\cechnerve\cover_\bullet)$ relates to these other approaches.

\begin{remark}
\label{remark:equal-vs-chain-homotopic-to-identity}
  In earlier papers on twisting cochains (such as \cite{TT1978}), the condition imposed on the ${\alpha\alpha}$ term of a twisting cochain was that it be \emph{equal} to the identity map of the $E_\alpha$ complex.
  However, as pointed out in \cite{Wei2016}, if one wants to construct a pre-triangulated dg-category of twisting cochains, then, in order for mapping cones to exist, this condition needs to be weakened to only asking that the $\alpha\alpha$ term be \emph{chain homotopic} to the identity map.\footnote{In \cite{Wei2016} the terminology \emph{twisted (perfect) complex} is used instead, to differentiate between the classical and the more homotopical definitions. In this present article, we say e.g. ``morphism of twisting cochains'' to mean ``morphism of twisted perfect complexes'', i.e. considering twisting cochains as a full subcategory of twisted perfect complexes. Throughout the literature in general, the use of twisted/twisting cochain/complex is not entirely consistent.}

  One might ask whether we could modify the construction somehow in order to rectify this discrepancy, obtaining the more modern definition.
  But note that degenerate $1$-simplices are defined entirely by an \emph{object}, i.e. by a degeneracy map applied to a $0$-simplex.
  Inherent to the dg- (or, more generally, enriched-) categorical framework is the fact that we only enrich the \emph{morphisms}; the objects remain purely set-theoretical, and so cannot hope to describe any sort of homotopical information.

  This point aside, more specific to the question in hand, we note that it is not too surprising that we are not able to recover the pre-triangulated structure on twisting cochains.
  Indeed, we are not trying to construct from them a dg-category, but instead a \emph{space}, and so rather than expecting something resembling a pre-triangulated structure, one should instead look towards studying the higher $\pi_n$ of the resulting space to see what information it contains.
\end{remark}

Since $\Twist$ is a presheaf of Kan complexes by definition, \cref{lemma:tot-of-cech-of-kan-is-kan} says that $\Tot\Twist(\cechnerve\cover_\bullet)$ is a Kan complex.
Morally, this means that we should expect its $1$-simplices to describe \emph{invertible} morphisms.
We can make this observation formal by showing that the $1$-simplices are not merely morphisms of twisting cochains, but instead \emph{weak equivalences}.

\begin{theorem}
\label{theorem:1-simplices-in-complex-analytic-twist}
  Let $E=(E^\bullet,\varphi)$ and $F=(F^\bullet,\psi)$ be points in $\Tot^0\Twist(\cechnerve\cover_\bullet)$, where $(X,\cover)$ is as in \cref{theorem:tot0-all-three}.
  Then a path $\lambda\in\Tot^1\Twist(\cechnerve\cover_\bullet)$ from $E$ to $F$ gives a weak equivalence of twisting cochains $(F^\bullet,\psi)\simto(E^\bullet,\varphi)$ in the sense of \cite[Definition~2.27]{Wei2016}.
\end{theorem}

\begin{proof}
  This is a purely combinatorial calculation using an explicit description of the non-degenerate simplices of $\Delta[p]\times\Delta[1]$.
  We give the full details (including recalling the definition of a weak equivalence of twisting cochains) in \cref{appendix:proof-of-1-simplices-in-complex-analytic-twist}.
\end{proof}

As an immediate consequence, we can use the language of \cref{subsection:simplicial-homotopy-theory} (since $\Twist$ is a presheaf of Kan complexes, and is thus in particular such that $2$-horns fill) to say the following.

\begin{corollary}
\label{corollary:pi0-twist-is-twc-mod-we}
  Let $(X,\cover)$ be as in \cref{theorem:tot0-all-three}.
  Then $\pi_0\Tot\Twist(\cechnerve\cover_\bullet)$ consists of twisting cochains modulo weak equivalence of twisting cochains.
\end{corollary}

\begin{remark}
\label{remark:1-simplices-in-twist-and-cubical}
  One important thing to note here is what we mean when we say that a path ``gives'' a weak equivalence in the statement of \cref{theorem:1-simplices-in-complex-analytic-twist}.
  The data of a $1$-simplex in the totalisation is actually slightly more than just a weak equivalence, since it also contains factorisations of the $\lambda_{\alpha_0\ldots\alpha_p}$ terms for $p\geq1$.
  For example, in \cref{appendix:proof-of-1-simplices-in-complex-analytic-twist} we \emph{define} the $\lambda_{\alpha\beta}$ term of the purported weak equivalence associated to the path $\lambda$ as the composition of two homotopies, and in doing so thus forget about the common (co)domain through which this factors.
  This is because the product $\Delta[1]\times\Delta[1]$ is not simply a square, but instead a square with diagonal, and it is exactly the data associated to the diagonal that we forget when constructing a weak equivalence of twisting cochains.
  As a consequence, there are multiple $1$-simplices in $\Tot\Twist(\cechnerve\cover_\bullet)$ which describe the same weak equivalence of twisting cochains.

  Rather than quotienting the space by some equivalence relation in order to remedy this (harmless) situation, it seems more desirable to better understand how the dg-nerve has an inherent \emph{cubical} structure to it, by relating its combinatorics to that of the pair subdivision; the extra data floating around in the $1$-simplices in the totalisation comes exactly from the fact that we are working simplicially instead of cubically.
\end{remark}

\begin{remark}
\label{remark:turning-resolution-by-twisting-cochains-into-pi-0-statement}
  The study of twisting cochains in the dg-category setting is well explained in papers such as \cite{Wei2016,BHW2017}, where it is shown how they relate to the dg-category of perfect complexes.
  More precisely, in the language of this present paper, \cite[Proposition~4.9]{BHW2017} says that, for any ringed space $(X,\OO_X)$ with locally finite open cover $\cover$, the dg-category of twisting cochains is exactly the Čech totalisation of the presheaf that sends a ringed space to the dg-category of strictly perfect complexes.
  One can consider \cref{theorem:tot0-twist-is-twcs} and \cref{theorem:1-simplices-in-complex-analytic-twist} as a sort of space-theoretic analogue of these dg-categorical results: \cref{lemma:comparsion-of-tot-dg-and-tot-sset} tells us that if we applied $\core{\dgnerve(-)}$ to the dg-category described as the homotopy limit in \cite[Proposition~4.9]{BHW2017} then we would obtain the same space as given by our construction here.

  For example, if $M^\bullet$ and $N^\bullet$ are perfect complexes of quasi-coherent (or, in particular, coherent) $\OO_X$-modules on a connected complex-analytic manifold $X$, and $\cover$ is a locally finite Stein cover, then \cite[Proposition~3.21]{Wei2016} tells us that we can resolve $M^\bullet$ and $N^\bullet$ by twisting cochains $E$ and $F$ (respectively), in that their sheafifications \cite[Definition~3.1]{Wei2016} satisfy $\cal{S}(E)\simeq M^\bullet$ and $\cal{S}(F)\simeq N^\bullet$.
  Now, if $M^\bullet$ and $N^\bullet$ are quasi-isomorphic, then $\cal{S}(E)\simeq M^\bullet\simeq N^\bullet\simeq\cal{S}(F)$; we can then apply \cite[Corollary~3.10]{Wei2016} to show that $E$ and $F$ are weakly equivalent twisting cochains.
  But, by \cref{theorem:1-simplices-in-complex-analytic-twist}, this says that if two complexes are connected by a quasi-isomorphism, then we can resolve them by twisting cochains that are connected by a path in $\pi_0\Tot\Twist(\cechnerve\cover_\bullet)$.
  In the language of Kan complexes and their homotopy theory, this is exactly saying that there is an isomorphism between $\pi_0\Tot\Twist(\cechnerve\cover_\bullet)$ and the $\pi_0$ of the space whose points are perfect complexes of quasi-coherent sheaves and whose paths are quasi-isomorphisms between them, induced by the sheafification and resolution functors $\cal{S}$ and $\cal{T}$ from \cite{Wei2016}.
\end{remark}

\begin{remark}
\label{remark:constructing-spaces-for-perf-and-q-coh}
  Given \cref{remark:turning-resolution-by-twisting-cochains-into-pi-0-statement}, it seems natural to try to construct a space of perfect complexes using the same method of Čech totalisation, but the structure of perfect complexes does not really allow for this.
  Note that, when describing the $1$-simplices in the Čech totalisation, the degree-$0$ term is exactly a $1$-simplex in the same simplicial set as the degree-$1$ terms that constitute the $0$-simplices (see e.g. \cref{subsection:lines-in-bun-x} or \cref{appendix:proof-of-1-simplices-in-complex-analytic-twist}).
  For example, in the case of $\Tot\core{\nerve\Free(\cechnerve\cover_\bullet)}$, which is the space of complexes of locally free sheaves (though for simplicity here we will consider a complex concentrated in degree~$0$, i.e. a single locally free sheaf), the degree-$0$ part $\lambda_\alpha\colon E_\alpha\to F_\alpha$ of the $1$-simplices is of the same type as the degree-$1$ part $g_{\alpha\beta}$ of the $0$-simplices: the local data of morphisms of locally free sheaves and the transition functions are both $1$-simplices in $\core{\nerve\Free(U)}$ (for some $U$), i.e. isomorphisms of free sheaves.
  But in the case of perfect complexes, the gluing data (playing the role of the transition functions) consists of isomorphisms (since a perfect complex is a single global object), whereas the morphisms consist of \emph{quasi-}isomorphisms --- this mismatch means that we cannot describe a space whose points are perfect complexes and whose paths are quasi-isomorphisms via Čech totalisation.
  In terms of morphisms, this space looks like it arises from the Čech totalisation of some $\core{\dgnerve\cal{D}(-)}$; morally, in terms of objects, it sits somewhere between $\core{\nerve\cal{D}(-)}$ and $\core{\dgnerve\cal{D}(-)}$.

  However, this is not a defect of perfect complexes.
  Indeed, one reason that perfect complexes are so useful is the fact that they are global objects with homotopically weak local properties (i.e. they are locally quasi-isomorphic to complexes of locally free sheaves).
  This relates to the other key example that we cannot express in this language of Čech totalisation: the space of quasi-coherent sheaves.
\end{remark}

\subsection{Edges in Green complexes}
\label{subsection:1-simplices-in-complex-green}

When it comes to the $1$-simplices in $\Tot\Green(\cechnerve\cover_\bullet)$, there is not really a classical notion of morphism against which we can compare them.
The only definition that we know of is implicitly in \cite[Definition~3.2]{Hos2020a}, where the category of Green complexes is defined as a full subcategory of the homotopical category of cartesian locally free sheaves on the Čech nerve, meaning that the morphisms are simply chain maps, and the weak equivalences are quasi-isomorphisms.
It is with this structure that the $(\infty,1)$-category of Green complexes is shown \loccit\@ to be equivalent to that of locally coherent sheaves (after taking a homotopy colimit over refinements of covers).
The structure that we obtain here, from $\Tot\Green(\cechnerve\cover_\bullet)$, is very different: since we know that all $2$-horns fill in $\Green$ (\cref{corollary:filling-2-horns-in-green}), we know that the resulting morphisms will all be invertible.
But the difference is much more profound than this: as explained in \cref{remark:constructing-spaces-for-perf-and-q-coh}, the $1$-simplices in a Čech totalisation are of the same type as the gluing data for objects, so we should not expect to recover something like chain maps or quasi-isomorphisms, but instead something built from GTT-labellings.
However, it could be the case that these two notions happen to coincide, in that one can be strictified to recover the other, as we allude to in \cref{section:summary-and-future-work}.

Following the general description of $1$-simplices in the totalisation (cf. \cref{appendix:proof-of-1-simplices-in-complex-analytic-twist}), we can start to describe those in $\Tot\Green(\cechnerve\cover_\bullet)$.
Given Green complexes $C$ and $D$ in $\Tot^0\Green(\cechnerve\cover_\bullet)$, the degree-$0$ component of a path $f\colon C\to D$ consists of $1$-simplices
\[
  \begin{tikzpicture}[scale=2]
    \node (C) at (0,0) {$C_\alpha(\alpha)$};
    \node (D) at (0,2) {$D_\alpha(\alpha)$};
    \node (CD) at (0,1) {${C}'_\alpha(\alpha)\cong{D}'_\alpha(\alpha)$};
    \draw [-stealth] (C) to node[fill=white]{\scriptsize$\scr{L}_C$} (CD);
    \draw [-stealth] (D) to node[fill=white]{\scriptsize$\scr{L}_D$} (CD);
  \end{tikzpicture}
\]
in $\Green(U_\alpha)$.
In other words, the degree-$0$ component of the morphism $f$ is exactly a ``common generalisation'' of the degree-$0$ parts of a Green complex (the complexes of free sheaves over each open subset), namely isomorphic complexes $C'_\alpha(\alpha)\cong C_\alpha(\alpha)\oplus\scr{L}_C$ and $D'_\alpha(\alpha)\cong D_\alpha(\alpha)\oplus\scr{L}_D$ such that $\scr{L}_C$ and $\scr{L}_D$ are elementary.
In terms of the homotopy theory, this says that a necessary condition for there to exist a morphism between $C$ and $D$ is that they are built from resolutions of ``the same'' coherent sheaf (or complexes of coherent sheaves, cf. \cref{subsection:greens-resolution}), since $C_\alpha(\alpha)$ and $D_\alpha(\alpha)$ are forced to have the same homology.

The degree-$1$ component will then be some common generalisation of $C_\alpha(\alpha)$ and $D_\beta(\beta)$, along with two $2$-simplices in $\core{\nerve\Free(U_\alpha\beta)}$ mediating between this diagonal, the degree-$0$ parts of the morphism, and the $\alpha\beta$ parts of the Green complexes (cf. \cref{figure:abstract-diagonal-square}); but the nerve is $2$-coskeletal, so it is only the diagonal $1$-simplex that actually provides any new data.

Continuing on for higher simplices, we see that a path $f\colon C\to D$ will provide us with these common generalisations $C_{\alpha_i}(\alpha_i)\leftrightarrow D_{\alpha_j}(\alpha_j)$ for all $i<j$.
But since the $D_{\alpha_j}(\alpha_j)$ form a Green complex, we already have common generalisations $D_{\alpha_j}(\alpha_j)\leftrightarrow D_{\alpha_k}(\alpha_k)$ for \emph{all} $j,k$.
Composing with these allows us to find common generalisations $C_{\alpha_i}(\alpha_i)\leftrightarrow D_{\alpha_k}$ for \emph{all} $i,k$, and so it seems likely that we can invert the morphism $f$.
This (weakly) suggests the following conjecture, which is related to \cref{subsection:horn-filling-conditions}.

\begin{conjecture}
\label{conjecture:green-is-globally-fibrant}
  The simplicial presheaf $\Green$ is a presheaf of Kan complexes.
\end{conjecture}

Following on from \cref{remark:turning-resolution-by-twisting-cochains-into-pi-0-statement}, note that the existence of a $1$-simplex connecting two Green complexes built from resolutions of two quasi-isomorphic complexes would follow from \cref{conjecture:TTs-green-gives-morphism-stwist-to-green}.
We discuss the relation between $\Green$ and complexes of coherent sheaves further in \cref{section:summary-and-future-work}.

\subsection{Edges in simplicial twisting cochains}
\label{subsection:1-simplices-in-complex-stwist}

The $1$-simplices in $\Tot\sTwist(\cechnerve\cover_\bullet)$ can be described in exactly the same way as those of $\Tot\Green(\cechnerve\cover_\bullet)$ in \cref{subsection:1-simplices-in-complex-green}, but with two key differences: the isomorphisms become quasi-isomorphisms; and the structure is no longer $2$-coskeletal, so there is higher homotopy data than simply a collection of $1$-simplices (or of ``common generalisations'', in the language of \cref{subsection:1-simplices-in-complex-green}).
Since \cite{TT1986} does not define any notion of morphism of simplicial twisting cochains, we have nothing against which to compare the $1$-simplices in $\Tot\sTwist(\cechnerve\cover_\bullet)$.

\section{Relations between the three presheaves}
\label{section:morphisms-between-the-presheaves}

Recall that we have constructed three simplicial presheaves from $\core{\nerve\Free(-)}$, using the methods of GTT-labellings and ``upgrading'' the ordinary nerve to the dg-nerve, as shown in \cref{figure:square-of-simplicial-presheaves}.
We said in \cref{subsection:narrative} that $\core{\nerve\Free(-)}$ was somehow too discrete, and so we wanted to construct more subtle generalisations.
The aim of this section is to explain how the three generalisations that we have constructed ($\Twist$, $\Green$, and $\sTwist$) relate to one another, and to tie them in to the story told in \cite{TT1986} using the result of \cref{theorem:tot0-all-three}.
To be clear, we now return to working with the simplicial presheaves $\Green$, $\Twist$, and $\sTwist$ before Čech totalisation.

\begin{figure}[ht!]
  \centering
  \begin{tikzpicture}[xscale=3.5,yscale=2.5]
    \node (nfree) at (0,2) {$\core{\nerve\Free(-)}$};
    \node (twist) at (-1,1) {$\Twist=\core{\dgnerve\Free(-)}$};
    \node (green) at (1,1) {$\Green=\core{\nerve\Free(-)}^\mathrm{GTT}$};
    \node (stwist) at (0,0) {$\sTwist=\core{\dgnerve\Free(-)}^\mathrm{GTT}$};
    \draw [thick,-stealth,dashed] (nfree) to node[fill=white]{\footnotesize dg-nerve} (twist);
    \draw [thick,-stealth,dashed] (nfree) to node[fill=white]{\footnotesize GTT-labelling} (green);
    \draw [thick,-stealth,dashed] (twist) to node[fill=white]{\footnotesize GTT-labelling} (stwist);
    \draw [thick,-stealth,dashed] (green) to node[fill=white]{\footnotesize dg-nerve} (stwist);
  \end{tikzpicture}
  \caption{The three simplicial presheaves that we have so far constructed from $\core{\nerve\Free(-)}$, where we use a superscript $\mathrm{GTT}$ to denote the GTT-labelling construction of \cref{definition:GTT-labelling}. The dashed arrows show how each presheaf is constructed from another; they are \emph{not} morphisms of presheaves (although morphisms do exist, and we discuss them in \cref{definition:the-morphism-green-to-stwist,definition:the-morphism-twist-to-stwist}).}
  \label{figure:square-of-simplicial-presheaves}
\end{figure}

Using \cref{theorem:tot0-all-three}, we can rephrase the construction of Green's resolution, as described in \cite{TT1986}, as the existence of a composite map of sets
\[
  \Tot_0\Twist(\cechnerve\cover_\bullet)
  \to \Tot_0\sTwist(\cechnerve\cover_\bullet)
  \to \Tot_0\Green(\cechnerve\cover_\bullet).
\]
But now we have the language to study something much more general: morphisms of simplicial presheaves
\[
  \Twist \to \sTwist \from \Green
\]
and how they behave with respect to homotopy groups.
As a small note, the ``natural'' morphism between $\sTwist$ and $\Green$, given our definitions, goes in the opposite direction from the one in the composite morphism from \cite{TT1986}, but we conjecture (\cref{conjecture:green-into-stwist}) that these two morphisms are homotopy inverse to one another.

This is an important generalisation in two ways: first of all, this is \emph{independent of the geometry}, since we are working \emph{before} Čech totalisation; secondly, this contains information about \emph{higher structure}, since we are not just working with the $0$-simplices, but instead the entire space.

\subsection{Horn filling conditions}
\label{subsection:horn-filling-conditions}

By construction, $\Twist$ is a presheaf of Kan complexes, and thus globally fibrant.
This means that it satisfies many nice properties.
For example, its simplicial homotopy groups are naturally isomorphic to its topological homotopy groups; and its Čech totalisation is a Kan complex (\cref{lemma:tot-of-cech-of-kan-is-kan}), so we obtain a \emph{space} of twisting cochains.

It is not so simple to show whether or not $\Green$ and $\sTwist$ are presheaves of Kan complexes.
Indeed, in this paper, we only provide a partial result in this direction: we show that all $2$-horns in $\Green$ and $\sTwist$ fill.
Of course, it would be desirable to fully generalise this result and show that \emph{all} horns fill, but with the current definition of GTT-labelling this is not particularly easy (though we do still conjecture that $\Green$ at least is a presheaf of Kan complexes in \cref{conjecture:green-is-globally-fibrant}).
But, at the very least, showing that $2$-horns fill allows us to apply \cref{lemma:2-horns-fill-implies-simplicial-pi0-equals-topological-pi0}, which says that the simplicial $\pi_0$ is isomorphic to the topological $\pi_0$.

After proving that $2$-horns fill, we will remark on how one might try to generalise the proof to higher dimensions, and also on how the outer horns seem to be actually no more difficult to fill than the inner ones.

\begin{lemma}
\label{lemma:filling-2-horns-in-stwist}
  Any $2$-horn in $\sTwist(U)$ can be filled, for any $(U,\OO_U)\in\ConnRingSpace$.
\end{lemma}

We give here a proof that the outer horn $\Lambda_0[2]$ lifts, since the same argument can be applied to the other outer horn $\Lambda_2[2]$, and the argument for the inner horn $\Lambda_1[2]$ is strictly simpler (as we explain in the proof).

\begin{proof}
  Consider an arbitrary $2$-horn $\Lambda_0[2]\to\sTwist(U)$ as in \cref{figure:2-horn-in-stwist}.

  \begin{figure}[!ht]
    \centering
    \begin{tikzpicture}[scale=4]
      \node (0) at (210:1) {$C_0(0)$};
      \node (1) at (90:1) {$C_1(1)$};
      \node (2) at (330:1) {$C_2(2)$};
      \node (01) at (150:1) {\footnotesize$C_0(01)\xleftarrow{\varphi_{01}(01)} C_1(01)$};
      \node [vertex=white] (12) at (30:1) {};
      \node (02) at (270:1) {\footnotesize$C_0(02)\xleftarrow{\varphi_{02}(02)}C_2(02)$};
      \node [vertex=white] (012) at (0,0) {};
      \draw [-latex] (0) to node[fill=white]{\footnotesize$C_0^{\perp 01}(0)$} (01);
      \draw [-latex] (1) to node[fill=white]{\footnotesize$C_1^{\perp 01}(1)$} (01);
      \draw [dashed,-latex,shorten >=1em] (1) to (12);
      \draw [dashed,-latex,shorten >=1em] (2) to (12);
      \draw [-latex] (0) to node[fill=white]{\footnotesize$C_0^{\perp 02}(0)$} (02);
      \draw [-latex] (2) to node[fill=white]{\footnotesize$C_2^{\perp 02}(2)$} (02);
      \draw [dashed,-latex,shorten <=1em,shorten >=1em] (01) to (012);
      \draw [dashed,-latex,shorten <=1em,shorten >=1em] (12) to (012);
      \draw [dashed,-latex,shorten <=1em,shorten >=1em] (02) to (012);
    \end{tikzpicture}
    \caption{An arbitrary outer $2$-horn in $\sTwist(U)$, which we want to fill. For brevity, we write $ij$ instead of $\{i<j\}$.}
    \label{figure:2-horn-in-stwist}
  \end{figure}

  To fill this horn, we need, in particular, to construct a $2$-simplex in $\core{\dgnerve\Free(U)}$ to label the central vertex $\{0<1<2\}$, and we already have some restrictions on what the vertices of this $2$-simplex must look like in terms of the complexes that must sit inside it with elementary orthogonal complements.
  Labelling the three vertices of this $2$-simplex as $C_0(012)$, $C_1(012)$, and $C_2(012)$, we know that, for example, $C_0(012)$ must be such that both $C_0(01)$ and $C_0(02)$ must sit inside as a direct summand with elementary direct-sum complement.
  But we can write both of these as a direct sum of $C_0(0)$ with something elementary, since by assumption
  \[
    \begin{aligned}
      C_0(01) &\cong C_0(0) \oplus C_0^{\perp01}(0)
    \\C_0(02) &\cong C_0(0) \oplus C_0^{\perp02}(0).
    \end{aligned}
  \]
  This suggests that the ``minimal'' possibility for $C_0(012)$ is
  \[
    C_0(012) \coloneqq C_0(0) \oplus C_0^{\perp01}(0) \oplus C_0^{\perp02}(0)
  \]
  where we simply add both elementary complements, since then we can mediate between $C_0(01)$ and $C_0(02)$ to try to construct the rest of the $2$-simplex in a compatible way.

  Another condition of being a GTT-labelling (\cref{definition:GTT-labelling}) is that whatever quasi-isomorphisms $\varphi_{ij}(012)\colon C_j(012)\simto C_i(012)$ constitute the $2$-simplex that we construct must be compatible extensions of the already given $\varphi_{ij}(ij)$.
  Again, this suggest the ``minimal'' possibility for $\varphi_{ij}(012)$ to be simply given by taking the direct sum of $\varphi_{ij}(ij)$ with the identity on the remaining elementary component.
  Putting this all together, the given $2$-horn in $\sTwist(U)$ shown in \cref{figure:2-horn-in-stwist} induces a $2$-horn in $\core{\dgnerve\Free(U)}$, as shown in \cref{figure:induced-2-horn-in-dg-nerve}.

  \begin{figure}[!ht]
    \centering
    \begin{tikzpicture}[xscale=3.5,yscale=2.5]
      \node (0) at (0,0) {$C_0^{\perp01}(0)\oplus C_0(0)\oplus C_0^{\perp02}(0)$};
      \node (1) at (1,1.2) {$C_1^{\perp01}(1)\oplus C_1(1)\oplus C_0^{\perp02}(0)$};
      \node (2) at (2,0) {$C_0^{\perp01}(0)\oplus C_2(2)\oplus C_2^{\perp02}(2)$};
      \draw[-latex] (1) to node[above left]{\footnotesize$\widetilde{\varphi}_{01}(01)\oplus\id_{C_0^{\perp02}(0)}$} (0);
      \draw[dashed,-latex] (2) to (1);
      \draw[-latex] (2) to node[below]{\footnotesize$\id_{C_0^{\perp01}(0)}\oplus\widetilde{\varphi}_{02}(02)$} (0);
      \draw[dashed,double,double equal sign distance,-implies,shorten <=1em,shorten >=1.5em] (1) to (1,0);
    \end{tikzpicture}
    \caption{The ``minimal'' $2$-horn in $\core{\dgnerve\Free(U)}$ induced by the $2$-horn in $\sTwist(U)$ from \cref{figure:2-horn-in-stwist}. We write $\widetilde{\varphi}_{ij}(ij)$ to denote the conjugation of $\varphi_{ij}(ij)$ by the isomorphisms $\theta_k^{\perp ij}(k)$ between the direct sum $C_k(k)\oplus C_k^{\perp ij}(k)$ and $C_k(ij)$ for $k\in\{i,j\}$, i.e. $\widetilde{\varphi}_{ij}(ij)\coloneqq\theta_j^{\perp ij}(j)^{-1}\varphi_{ij}(ij)\theta_i^{\perp ij}(i)$.} 
    \label{figure:induced-2-horn-in-dg-nerve}
  \end{figure}

  But now we have a $2$-horn in a Kan complex (by definition, since we are labelling by the maximal Kan complex of the dg-nerve), and so we know that it can be filled.
  Here we can actually give a slightly more concrete description of how this works: we can apply the generalised Whitehead theorem to invert the quasi-isomorphism $\widetilde{\varphi}_{01}(01)\oplus\id$ and obtain a chain homotopy witnessing this quasi-inverse; pre-composition with $\id\oplus\widetilde{\varphi}_{02}(02)$ then gives the desired quasi-isomorphism $\varphi_{12}(012)$ along with the chain homotopy $\varphi_{012}(012)$.
  This gives us a $2$-simplex as shown in \cref{figure:filled-induced-2-horn-in-dg-nerve}.
  Note here that, if we were filling an inner horn $\Lambda_1[2]\to\sTwist(U)$ then we would not need to apply this argument, since we could simply compose the two existing quasi-isomorphisms and let $\varphi_{012}(012)$ be the identity chain homotopy.

  \begin{figure}[!ht]
    \centering
    \begin{tikzpicture}[xscale=4,yscale=3]
      \node (0) at (0,0) {$\underbrace{C_0^{\perp01}(0)\oplus C_0(0)\oplus C_0^{\perp02}(0)}_{\eqqcolon\,\, C_0(012)}$};
      \node (1) at (1,1.2) {$\overbrace{C_1^{\perp01}(1)\oplus C_1(1)\oplus C_0^{\perp02}(0)}^{\eqqcolon\,\, C_1(012)}$};
      \node (2) at (2,0) {$\underbrace{C_0^{\perp01}(0)\oplus C_2(2)\oplus C_2^{\perp02}(2)}_{\eqqcolon\,\, C_2(012)}$};
      \draw[-latex] (1) to node[above left]{\footnotesize$\overbrace{\widetilde{\varphi}_{01}(01)\oplus\id}^{\eqqcolon\,\,\varphi_{01}(012)}$} (0);
      \draw[-latex] (2) to node[above right]{\footnotesize$\varphi_{12}(012)$} (1);
      \draw[-latex] (2) to node[below]{\footnotesize$\underbrace{\id\oplus\widetilde{\varphi}_{02}(02)}_{\eqqcolon\,\,\varphi_{02}(012)}$} (0);
      \draw[double,double equal sign distance,-implies,shorten <=1em,shorten >=1.5em] (1) to node[fill=white]{\footnotesize$\varphi_{012}(012)$} (1,0);
    \end{tikzpicture}
    \caption{The filling of the induced $2$-horn from \cref{figure:induced-2-horn-in-dg-nerve} given by applying the generalised Whitehead theorem to invert $\widetilde{\varphi}_{01}(01)\oplus\id$ and then pre-composing with $\id\oplus\widetilde{\varphi}_{02}(02)$ to obtain $\varphi_{12}(012)$ and $\varphi_{012}(012)$.}
    \label{figure:filled-induced-2-horn-in-dg-nerve}
  \end{figure}

  Using this $2$-simplex from \cref{figure:filled-induced-2-horn-in-dg-nerve} to label the vertex $\{0<1<2\}$, all that remains to label is the collection of vertices and edges between $\{1\}$ and $\{2\}$, and then the three $2$-cells.
  But starting from this $2$-simplex we basically have no choices to make in how to label the lower-dimensional components.
  The only exception is $C_1(12)\simfrom C_2(12)$, since we could conceivably set $C_1(12)$ to be any of $C_1(1)$, $C_1(1)\oplus C_1^{\perp01}(1)$, or $C_1(1)\oplus C_1^{\perp01}(1)\oplus C_0^{\perp02}(0)$.
  However, since we want $\varphi_{12}(12)$ to be an \emph{isomorphism} in the case where we restrict to $\Green$ instead of $\sTwist$ (\cref{corollary:filling-2-horns-in-green}), the only option that makes sense is the last one.
  This gives us the complete labelling as shown in \cref{figure:filled-2-horn-in-stwist}.

  \begin{figure}[!ht]
    \centering
    \begin{tikzpicture}[scale=6]
      \node [gray] (0) at (210:1) {$C_0(0)$};
      \node [gray] (1) at (90:1) {$C_1(1)$};
      \node [gray] (2) at (330:1) {$C_2(2)$};
      \node [gray] (01) at (150:1) {\footnotesize$C_0(01)\xleftarrow{\varphi_{01}(01)} C_1(01)$};
      \node (12) at (30:1) {\footnotesize$C_1(012)\xleftarrow{\varphi_{12}(012)}C_2(012)$};
      \node [gray] (02) at (270:1) {\footnotesize$C_0(02)\xleftarrow{\varphi_{02}(02)}C_2(02)$};
      \node (012) at (0,0) {(\cref{figure:filled-induced-2-horn-in-dg-nerve})};
      \draw [gray,-latex] (0) to node[fill=white]{\footnotesize$C_0^{\perp01}(0)$} (01);
      \draw [gray,-latex] (1) to node[fill=white]{\footnotesize$C_1^{\perp01}(1)$} (01);
      \draw [-latex] (1) to node[fill=white]{\footnotesize$C_1^{\perp01}(1)\oplus C_0^{\perp02}(0)$} (12);
      \draw [-latex] (2) to node[fill=white]{\footnotesize$C_0^{\perp01}(0)\oplus C_2^{\perp02}(2)$} (12);
      \draw [gray,-latex] (0) to node[fill=white]{\footnotesize$C_0^{\perp02}(0)$} (02);
      \draw [gray,-latex] (2) to node[fill=white]{\footnotesize$C_2^{\perp02}(2)$} (02);
      \draw [-latex,shorten <=1em,shorten >=1em] (01) to node[fill=white]{\footnotesize$\Big(C_0^{\perp02}(0),C_0^{\perp02}(0)\Big)$} (012);
      \draw [-latex,shorten <=1em,shorten >=1em] (12) to node[fill=white]{\footnotesize$(0,0)$} (012);
      \draw [-latex,shorten <=1em,shorten >=1em] (02) to node[fill=white]{\footnotesize$\Big(C_0^{\perp01}(0),C_0^{\perp01}(0)\Big)$} (012);
      \draw [double,double equal sign distance,-implies,shorten <=1em,shorten >=1em] (0) to node[fill=white]{\footnotesize$C_0^{\perp01}(0)\oplus C_0^{\perp02}(0)$} (012);
      \draw [double,double equal sign distance,-implies,shorten <=1em,shorten >=1em] (1) to node[fill=white]{\footnotesize$C_1^{\perp01}(1)\oplus C_0^{\perp02}(0)$} (012);
      \draw [double,double equal sign distance,-implies,shorten <=1em,shorten >=1em] (2) to node[fill=white]{\footnotesize$C_2^{\perp02}(2)\oplus C_0^{\perp01}(0)$} (012);
    \end{tikzpicture}
    \caption{The filling of the $2$-horn from \cref{figure:2-horn-in-stwist}, where we use the $2$-simplex from \cref{figure:filled-induced-2-horn-in-dg-nerve} to label the central vertex. Note that the $2$-cells are labelled trivially, in that the two paths along their boundary (giving two choices of elementary complement) agree on the nose.}
    \label{figure:filled-2-horn-in-stwist}
  \end{figure}

  Now we need to check that this labelling does indeed satisfy the GTT-labelling conditions of \cref{definition:GTT-labelling} in order for it to be a $2$-simplex in $\sTwist(U)$.

  First of all, the data ``type-checks'': the vertices are labelled with simplices in $\core{\dgnerve\Free(U)}$ of the right degree, and the higher-dimensional cells are labelled with elementary complexes.\footnote{Recall \cref{remark:elementary-complements-in-GTT}, and note how here the elementary complements are not completely arbitrary, but instead built up from the $C_k^{\perp ij}(ij)$.}

  Next we need to check that we do indeed have direct-sum decompositions
  \[
    \theta_{j_m}^{\perp\sigma}(\tau)\colon C_{j_m}(\tau)\oplus C_{j_m}^{\perp\sigma}(\tau) \xrightarrow{\cong} C_{j_m}(\sigma)
  \]
  for all $\tau\subset\sigma$, but this is clear by construction.

  Finally, we need certain diagrams to commute, and there are three cases to check, corresponding to the three subsets $\{0<1\}$, $\{0<2\}$, and $\{1<2\}$ of $\{0<1<2\}$.
  The diagram for $\{0<1\}$ is
  \[
    \begin{tikzcd}[sep=large]
      C_1(01)
        \ar[r,hook]
        \ar[d,swap,"\varphi_{01}(01)"]
      & C_1^{\perp01}(1)\oplus C_1(1)\oplus C_0^{\perp02}(0)
        \ar[r,two heads]
        \ar[d,"\widetilde{\varphi}_{01}(01)\oplus\id_{C_0^{\perp02}(0)}"]
      & C_0^{\perp02}(0)
        \ar[d,"\id_{C_0^{\perp02}(0)}"]
    \\C_0(01)
        \ar[r,hook]
      & C_0^{\perp01}(0)\oplus C_0(0)\oplus C_0^{\perp02}(0)
        \ar[r,two heads]
      & C_0^{\perp02}(0)
    \end{tikzcd}
  \]
  where the $\hookrightarrow$ (resp. $\twoheadrightarrow$) are now not just the inclusion (resp. projection) of the direct sum, but instead the post-composition (resp. pre-composition) with the corresponding $\theta_k^{\perp ij}(k)^{-1}$ (resp. $\theta_k^{\perp ij}(k)$), but this expresses exactly the same condition as the diagram in \cref{definition:GTT-labelling}, since $\theta$ is an isomorphism.
  But this clearly commutes by definition: the left-hand square commutes since $\widetilde{\varphi}_{01}(01)$ is defined (in \cref{figure:induced-2-horn-in-dg-nerve}) exactly as the conjugation of $\varphi_{01}(01)$ by the corresponding $\theta_k^{\perp 01}(k)$, and the right-hand square commutes since the two horizontal arrows are identical.
  The same argument applies for the diagram corresponding to $\{0<2\}$.

  For $\{1<2\}$, the commutativity is even more immediate: the diagram is
  \[
    \begin{tikzcd}[sep=large]
      C_2(12)
        \ar[r,hook]
        \ar[d,swap,"\varphi_{12}(12)"]
      & C_2(012)
        \ar[r,two heads]
        \ar[d,"\varphi_{012}(012)"]
      & 0
        \ar[d]
    \\C_1(12)
        \ar[r,hook]
      & C_1(012)
        \ar[r,two heads]
      & 0
    \end{tikzcd}
  \]
  and this commutes by definition, since the $\hookrightarrow$ are simply identity morphisms, and $\varphi_{12}(12)\coloneqq\varphi_{12}(012)$.
\end{proof}

\begin{corollary}
\label{corollary:filling-2-horns-in-green}
  Any $2$-horn in $\Green(U)$ can be filled, for any $(U,\OO_U)\in\ConnRingSpace$.
\end{corollary}

\begin{proof}
  By \cref{remark:GTT-1-labelling-equivalent-definition}, we simply need to show that, in the proof of \cref{lemma:filling-2-horns-in-stwist}, if the $\varphi_J(\sigma)$ given in the horn are isomorphisms, then all other $\varphi_{ij}(ij)$ are also isomorphisms, and $\varphi_{012}(012)$ is zero.
  But if $\varphi_{ij}(ij)$ is an isomorphism, then so too is $\widetilde{\varphi}_{ij}(ij)\oplus\id$, and we don't need to apply generalised Whitehead in order to construct $\varphi_{12}(012)$, since it will simply be the isomorphism $(\widetilde{\varphi}_{01}(01)\oplus\id)^{-1}\circ(\id\oplus\widetilde{\varphi}_{02}(02))$, and the triangle in \cref{figure:filled-induced-2-horn-in-dg-nerve} will commute strictly, meaning that we can take $\varphi_{012}(012)$ to be zero.
\end{proof}

These proofs do not immediately generalise to higher dimensions: it is not clear, in the case of filling $\Lambda_0[3]$, for example, what one should choose for $C_0(0123)$.
However, it does seem likely that the only real difficulty is in constructing the $p$-simplex that labels the central vertex of $\pair{\Delta[p]}$, since then one should be able to simply apply topological face maps to label all vertices of lower virtual dimension (i.e. those corresponding to lower-dimensional sub-simplices of $\Delta[p]$), as we do in the proof of \cref{lemma:filling-2-horns-in-stwist}.
A consequence of this is that the argument for filling outer horns should be almost identical to that for filling inner horns, since we are labelling by a Kan complex, and so this ``induced horn'' (as in \cref{figure:induced-2-horn-in-dg-nerve}) can always be filled, whether inner or outer.

One final remark on the proof of \cref{lemma:filling-2-horns-in-stwist} is that the main idea --- adding all ``opposite'' elementary complements: defining $C_0(012)$ to be $C_0(0)\oplus \bigoplus_{i\neq0}C_0^{\perp 0i}(0)$ --- is taken directly from the construction of Green's simplicial resolution (\cite{Gre1980} or \cite{TT1986}).

\subsection{Inclusions}
\label{subsection:inclusions-of-simplicial-presheaves}

\begin{definition}
\label{definition:the-morphism-twist-to-stwist}
  The \emph{inclusion of $\Twist$ into $\sTwist$} is the injective morphism
  \[
    \Twist\hookrightarrow\sTwist
  \]
  given object-wise by labelling the central vertex, using face maps to label the other vertices, and then labelling all edges with zero elementary orthogonal complements.

  In more detail, given a $p$-simplex $\tau$ of $\Twist$, we need to construct a GTT-labelling of $\pair{\Delta[p]}$, since this is exactly a $p$-simplex of $\sTwist$.
  But $\Twist=\core{\dgnerve\Free(-)}$ is exactly the thing that we label by in a GTT-labelling.
  So we can simply label the central vertex $([p],[p])$ of $\pair{\Delta[p]}$ with $\tau$, use the face maps of $\Delta[p]$ to label all other vertices, and finally label all edges with zero elementary orthogonal complements (so that the inclusions into the direct sums are exactly identity maps).
  By construction, all the conditions of \cref{definition:GTT-labelling} are then trivially satisfied.

  For clarity, we spell this out explicitly in dimensions $0$, $1$, and $2$.
  \begin{itemize}
    \item A $0$-simplex of $\Twist(U)$ is a complex $A$ of $\OO_U$-modules, which gives us a labelling of \mbox{$\pair{\Delta[0]}=\{*\}$} by simply labelling the single vertex with this complex.
    \item A $1$-simplex of $\Twist(U)$ is a quasi-isomorphism $A\simfrom B$ of complexes.
      We label the vertex of $\pair{\Delta[1]}$ corresponding to $\{0<1\}$ with this element $\sigma$, label $\{0\}$ with $A$, and $\{1\}$ with $B$.
      Both edges $\{0\}\subset\{0<1\}$ and $\{1\}\subset\{0<1\}$ are labelled with $0$ (so the inclusion into direct-sum decomposition is given by the corresponding identity map), giving us
      \[
        \begin{tikzpicture}[xscale=2.5]
          \node (0) at (0,0) {$A$};
          \node (01) at (1,0) {$(A\simfrom B)$};
          \node (1) at (2,0) {$B$};
          \draw[-latex] (0) to node[fill=white]{\footnotesize$0$} (01);
          \draw[-latex] (1) to node[fill=white]{\footnotesize$0$} (01);
          \node[vertex] at (0,-0.5) {};
          \node[vertex] at (1,-0.5) {};
          \node[vertex] at (2,-0.5) {};
          \draw[thick] (0,-0.5) to (2,-0.5);
        \end{tikzpicture}
      \]
    \item A $2$-simplex $\tau$ of $\Twist(U)$ consists of three quasi-isomorphisms and a chain homotopy, of the form
      \[
        \begin{tikzpicture}
          \node (A) at (0,0) {$A$};
          \node (B) at (1,1.2) {$B$};
          \node (C) at (2,0) {$C$};
          \draw[-latex] (B) to node[sloped,above]{$\sim$} (A);
          \draw[-latex] (C) to node[sloped,above]{$\sim$} (B);
          \draw[-latex] (C) to node[sloped,below]{$\sim$} (A);
          \draw[double,double equal sign distance,-implies,shorten >=5pt,shorten <=2pt] (B) to (1,0);
        \end{tikzpicture}
      \]
      We label the vertex of $\pair{\Delta[2]}$ corresponding to $\{0<1<2\}$ with $\tau$, and then label the vertex corresponding to $\{i<j\}$ with the face $\{i<j\}$ of $\tau$, and finally the vertex corresponding to $\{i\}$ with the vertex $\{i\}$ of $\tau$, giving us
      \[
        \begin{tikzpicture}[scale=3.5]
          \node (0) at (210:1) {$A$};
          \node (1) at (90:1) {$B$};
          \node (2) at (330:1) {$C$};
          \node (01) at (150:1) {$A\simfrom B$};
          \node (12) at (30:1) {$B\simfrom C$};
          \node (02) at (270:1) {$A\simfrom C$};
          \node (012) at (0,0) {};
          \draw[thick,-latex] (0) to node[fill=white]{\footnotesize$0$} (01);
          \draw[thick,-latex] (1) to node[fill=white]{\footnotesize$0$} (01);
          \draw[thick,-latex] (1) to node[fill=white]{\footnotesize$0$} (12);
          \draw[thick,-latex] (2) to node[fill=white]{\footnotesize$0$} (12);
          \draw[thick,-latex] (0) to node[fill=white]{\footnotesize$0$} (02);
          \draw[thick,-latex] (2) to node[fill=white]{\footnotesize$0$} (02);
          \draw[thick,-latex,shorten >=18pt] (01) to node[fill=white,pos=0.35]{\footnotesize$(0,0)$} (012);
          \draw[thick,-latex,shorten >=18pt] (12) to node[fill=white,pos=0.35]{\footnotesize$(0,0)$} (012);
          \draw[thick,-latex,shorten >=22pt] (02) to node[fill=white,pos=0.4]{\footnotesize$(0,0)$} (012);
          \draw[double,double equal sign distance,-implies,shorten <=1em,shorten >=4em] (0) to node[pos=0.3,fill=white]{\footnotesize$0$} (012);
          \draw[double,double equal sign distance,-implies,shorten <=1em,shorten >=3em] (1) to node[pos=0.4,fill=white]{\footnotesize$0$} (012);
          \draw[double,double equal sign distance,-implies,shorten <=1em,shorten >=4em] (2) to node[pos=0.3,fill=white]{\footnotesize$0$} (012);
          \begin{scope}[scale=0.27,shift={(-1,-0.6)}]
            \node (A) at (0,0) {$A$};
            \node (B) at (1,1.2) {$B$};
            \node (C) at (2,0) {$C$};
            \draw[-latex] (B) to node[sloped,above]{$\sim$} (A);
            \draw[-latex] (C) to node[sloped,above]{$\sim$} (B);
            \draw[-latex] (C) to node[sloped,below]{$\sim$} (A);
            \draw[double,double equal sign distance,-implies,shorten >=5pt,shorten <=2pt] (B) to (1,0);
          \end{scope}
        \end{tikzpicture}
      \]
  \end{itemize}

  The fact that this defines a morphism follows from the fact that this construction commutes with the face and degeneracy maps in $\Twist$ and $\sTwist$.
\end{definition}

\begin{definition}
\label{definition:the-morphism-green-to-stwist}
  The \emph{inclusion of $\Green$ into $\sTwist$} is the injective morphism
  \[
    \Green\hookrightarrow\Twist
  \]
  induced by the inclusion of the nerve into the dg-nerve.
  In other words, any GTT-$1$-labelling of $\Delta[p]$ by $\core{\nerve\Free(U)}$ is a specific example of a GTT-labelling of $\Delta[p]$ by $\core{\dgnerve\Free(U)}$ where the $\varphi_J$ are zero for $|J|\geq3$, by definition
\end{definition}

\subsection{Equivalences}
\label{subsection:equivalences}

For presheaves of Kan complexes, \cref{lemma:tot-of-we-of-kan-is-we} tells us that weak equivalences are preserved by Čech totalisation.
Just as it is easier to work with simplicial homotopy groups instead of the homotopy groups of the geometric realisation, it can be easier to work with globally fibrant simplicial presheaves directly instead of their Čech totalisations.
If we already know a concrete description of $\Tot\scr{F}(\cechnerve\cover_\bullet)$ and $\Tot\scr{G}(\cechnerve\cover_\bullet)$ (as we do for $\Twist$, $\Green$, and $\sTwist$ on any complex-analytic $(X,\cover)$, thanks to \cref{theorem:tot0-all-three}), then to show they are equivalent as spaces it suffices to show the stronger statement that $\scr{F}\simeq\scr{G}$.

Since we are unable to prove whether or not $\Green$ and $\sTwist$ are indeed presheaves of Kan complexes, we cannot even work with their simplicial $\pi_n$ for $n\geq1$, let alone hope to see the existence of, or obstruction to, an equivalence between $\Tot\Green(\cechnerve\cover_\bullet)$ and $\Tot\sTwist(\cechnerve\cover_\bullet)$.
However, since we know that $2$-horns fill (\cref{lemma:filling-2-horns-in-stwist} and \cref{corollary:filling-2-horns-in-green}), their simplicial $\pi_0$ are indeed well defined and calculate the topological $\pi_0$ (\cref{lemma:2-horns-fill-implies-simplicial-pi0-equals-topological-pi0}).
Using this, we can show that $\pi_0\Twist$, $\pi_0\Green$, and $\pi_0\sTwist$ are all isomorphic.

\begin{theorem}
\label{theorem:pi0-of-twist-and-pi0-of-stwist}
  The inclusion $i\colon\Twist\hookrightarrow\sTwist$ induces an isomorphism
  \[
    i_0\colon\pi_0\Twist(U) \cong \pi_0\sTwist(U)
  \]
  for any $(U,\OO_U)\in\ConnRingSpace$.
\end{theorem}

\begin{proof}
  Since $\pi_0$ is merely a set, we simply need to show that $i_0$ is a bijection.
  Firstly, note it is surjective, since $\Twist(U)_0=\sTwist(U)_0$.
  To prove that it is injective, we need to show that, if (the image under $i$ of) two $0$-simplices of $\Twist(U)$ are connected by some $1$-simplex in $\sTwist(U)$, then they are already connected by some $1$-simplex in $\Twist(U)$.
  So let $A,B\in\Twist(U)_0$ be such that there exists some $1$-simplex
  \[
    \begin{tikzpicture}[xscale=2.5]
      \node (0) at (0,0) {$A$};
      \node (01) at (1,0) {$(A'\simfrom B')$};
      \node (1) at (2,0) {$B$};
      \draw[-latex] (0) to (01);
      \draw[-latex] (1) to (01);
      \node[vertex] at (0,-0.5) {};
      \node[vertex] at (1,-0.5) {};
      \node[vertex] at (2,-0.5) {};
      \draw[thick] (0,-0.5) to (2,-0.5);
    \end{tikzpicture}
  \]
  in $\sTwist(U)$ (where we omit the elementary orthogonal complements labelling the $1$-simplices since they do not play a role in this proof).
  Note that, by definition, the morphisms $A\to A'$ and $B\to B'$ are not just quasi-isomorphisms, but also have quasi-inverses (\cref{lemma:adding-elementary-is-quasi-iso}).\footnote{We could also appeal to the generalised Whitehead theorem in order to invert the quasi-isomorphism $A\simto A'$, but we do not need such heavy machinery here.}
  Finding some $1$-simplex in $\Twist(U)$ connecting $A$ and $B$ just means, by definition, finding some quasi-isomorphism $A\simfrom B$, and so we can simply take the composite $A\simfrom A'\simfrom B'\simfrom B$.
\end{proof}

\begin{remark}
\label{remark:green-lets-you-turn-quasi-isos-into-isos}
  One particularly useful special case of Green's construction is given by looking at the Čech-degree-$1$ part (\cite[(8.3.6)]{Hos2020} or \cite[pp.~23--25]{Gre1980}).
  This shows that any quasi-isomorphism of bounded complexes of free modules can be ``strictified'' to an isomorphism in the following sense: if $A\simfrom_f B$ is a quasi-isomorphism of bounded complexes of free modules, then there exists complexes $\widetilde{A}$ and $\widetilde{B}$ such that
  \begin{enumerate}[(i)]
    \item $\widetilde{A}=A\oplus E_A$, where $E_A$ is a (bounded) $B$-elementary complex;
    \item $\widetilde{B}=B\oplus E_B$, where $E_B$ is a (bounded) $A$-elementary complex;
    \item $\widetilde{A}\xleftarrow{\widetilde{f}}\widetilde{B}$ is an isomorphism;
    \item the restriction of $\widetilde{f}$ to $B$ is exactly $A\simfrom_f B$.
  \end{enumerate}
  Note that (i) and (ii) imply, in particular, that $A\simeq\widetilde{A}$ and $B\simeq\widetilde{B}$ are chain homotopy equivalent.

  In fact, Green proves something stronger, even in Čech-degree~$1$.
  He shows that, if we further have coherent homotopies $p_i$ and $q_i$ for $i\geq1$ (so, for example, $p_1$ and $q_1$ are the homotopies witnessing that $f$ has a quasi-inverse, say $g$; $p_2$ and $q_2$ witness the failure of $p_1$ and $q_1$ to commute with $f$ and $g$; etc.), then this construction can still be applied, resulting in a strict isomorphism, with all higher homotopies $p_i$ and $q_i$ being strictly zero.
\end{remark}

\begin{theorem}
\label{theorem:pi0-of-green-and-pi0-of-stwist}
  The inclusion $i\colon\Green\hookrightarrow\sTwist$ induces an isomorphism
  \[
    i_0\colon\pi_0\Green(U) \cong \pi_0\sTwist(U)
  \]
  for any $(U,\OO_U)\in\ConnRingSpace$.
\end{theorem}

\begin{proof}
  Since $\pi_0$ is merely a set, we simply need to show that $i_0$ is a bijection.
  Firstly, note that it is surjective, since $\Green(U)_0=\sTwist(U)_0$.
  To prove that it is injective, we need to show that, if (the image under $i$ of) two $0$-simplices of $\Green(U)$ are connected by some $1$-simplex in $\sTwist(U)$, then they are already connected by some $1$-simplex in $\Green(U)$.
  So let $A,B\in\Green(U)_0$ be such that there exists some $1$-simplex
  \[
    \begin{tikzpicture}[xscale=2.5]
      \node (0) at (0,0) {$A$};
      \node (01) at (1,0) {$(A'\simfrom B')$};
      \node (1) at (2,0) {$B$};
      \draw[-latex] (0) to node[fill=white]{\footnotesize$A^\perp$} (01);
      \draw[-latex] (1) to node[fill=white]{\footnotesize$B^\perp$} (01);
      \node[vertex] at (0,-0.5) {};
      \node[vertex] at (1,-0.5) {};
      \node[vertex] at (2,-0.5) {};
      \draw[thick] (0,-0.5) to (2,-0.5);
    \end{tikzpicture}
  \]
  in $\sTwist(U)$.

  Since $A'$ and $B'$ are bounded complexes of free modules, the generalised Whitehead theorem tells us that the quasi-isomorphism $A'\simfrom B'$ is actually a chain homotopy equivalence, with chain homotopy inverse.\footnote{To show this, we appeal to the classical abstract argument: we endow $\Free(U)$ with the projective model structure; since the objects of $\Free(U)$ are \emph{bounded} complexes, they are fibrant; since they are complexes of \emph{free} modules, they are also cofibrant; by the generalised Whitehead theorem \cite[\S II,~Theorem~1.10]{GJ2009}, every quasi-isomorphism in $\Free(U)$ is a homotopy equivalence, and thus has a homotopy inverse.}
  But this then puts us in the setting of \cref{remark:green-lets-you-turn-quasi-isos-into-isos}, and so obtain an isomorphism
  \[
    \widetilde{A} \coloneqq A\oplus\widetilde{A}^\perp
    \cong B\oplus\widetilde{B}^\perp \eqqcolon \widetilde{B}
  \]
  where
  \[
    \begin{aligned}
      \widetilde{A}^\perp
      &\coloneqq
      A^\perp \oplus
      \left(
        \bigoplus_{(B')^i\neq0} (B\oplus B^\perp)^i\xrightarrow{\id}(B\oplus B^\perp)^i
      \right)
    \\\widetilde{B}^\perp
      &\coloneqq
      B^\perp \oplus
      \left(
        \bigoplus_{(A')^i\neq0} (A\oplus A^\perp)^i\xrightarrow{\id}(A\oplus A^\perp)^i
      \right)
    \end{aligned}
  \]
  and where the resulting isomorphism $\widetilde{B}\cong_f\widetilde{A}$ is such that its restriction $B'\to A'$ is exactly the initial quasi-isomorphism, but it is non-trivially modified in its other three components.
  By definition, $A^\perp$ is elementary in $A'$, and so $\widetilde{A}^\perp$ is elementary in $A'$ and $B'$, since we simply take the direct sum of $A^\perp$ with $B'$-elementary complexes.
  So there are no further conditions to check: we have constructed a $1$-simplex
  \[
    \begin{tikzpicture}[xscale=2.5]
      \node (0) at (0,0) {$A$};
      \node (01) at (1,0) {$(\widetilde{A}\cong\widetilde{B})$};
      \node (1) at (2,0) {$B$};
      \draw[-latex] (0) to node[fill=white]{\footnotesize$\widetilde{A}^\perp$} (01);
      \draw[-latex] (1) to node[fill=white]{\footnotesize$\widetilde{B}^\perp$} (01);
      \node[vertex] at (0,-0.5) {};
      \node[vertex] at (1,-0.5) {};
      \node[vertex] at (2,-0.5) {};
      \draw[thick] (0,-0.5) to (2,-0.5);
    \end{tikzpicture}
  \]
  in $\Green(U)$.
\end{proof}

\begin{corollary}
\label{corollary:pi0-all-three-equivalent}
  Let $(X,\OO_X)\in\ConnRingSpace$ with cover $\cover$.
  Then
  \[
    \pi_0\Tot\Twist(\cechnerve\cover_\bullet) \cong \pi_0\Tot\Green(\cechnerve\cover_\bullet) \cong \pi_0\Tot\sTwist(\cechnerve\cover_\bullet).
  \]
\end{corollary}

\begin{proof}
  This is \cref{lemma:tot-of-we-of-kan-is-we} applied to \cref{theorem:pi0-of-twist-and-pi0-of-stwist} and \cref{theorem:pi0-of-green-and-pi0-of-stwist}.
\end{proof}

Knowing that $\Green$, $\Twist$, and $\sTwist$ all have equivalent $\pi_0$ leads to the natural question: are all higher $\pi_n$ equivalent as well?
In other words, \emph{are these simplicial presheaves globally weakly equivalent}?
We do not claim to have an answer to this question, as we are largely hindered by the fact that we are unable to prove global fibrancy of $\Green$ and $\sTwist$.
\emph{If}, however, one could prove global fibrancy, then it is not too difficult to show by hand that, for example, $i\colon\pi_1\Twist\to\pi_1\sTwist$ is a surjection, and it seems possible that the method of proof could generalise to higher dimensions.
We further discuss what the implications of these global weak equivalences would be in \cref{section:summary-and-future-work}.
For now, we provide a single conjecture.

\begin{conjecture}
\label{conjecture:green-into-stwist}
  The inclusion $i\colon\Green\hookrightarrow\sTwist$ induces a weak equivalence of simplicial presheaves.
  Furthermore, this weak equivalence is witnessed by a homotopy inverse, given by the construction of \cref{conjecture:TTs-green-gives-morphism-stwist-to-green}.
\end{conjecture}

Note that this conjecture assumes that both $\Green$ and $\sTwist$ are globally fibrant, and is thus dependent on both \cref{conjecture:green-is-globally-fibrant} and the statement (which we do not separately conjecture) that $\sTwist$ is a presheaf of Kan complexes.
The justification for \cref{conjecture:green-into-stwist} is that one can try to show by hand that the induced map $i_n\colon\pi_n\Green(U)\to\pi_n\sTwist(U)$ is an isomorphism for each $n\in\mathbb{N}$, and showing surjectivity amounts to applying Green's resolution (specifically the version described in \cite[\S3]{TT1986}).
This suggests that realising Green's resolution as a morphism $\sTwist\to\Green$ of simplicial presheaves would provide an explicit homotopy inverse to the inclusion $\Green\hookrightarrow\sTwist$.
But the construction of this morphism is not so trivial, and requires certain modifications to the theory that we have developed, as we now explain.

\subsection{Green's resolution}
\label{subsection:greens-resolution}

We \emph{would like} to be able to say that Green's resolution as described in \cite[\S3]{TT1986} defines a morphism of simplicial presheaves
\[
  \sTwist\to\Green
\]
and, indeed, this is the content of \cref{conjecture:TTs-green-gives-morphism-stwist-to-green}.
Although we are \emph{almost} able to construct this morphism, we will show how there is a technical problem that requires us to enrich our framework with the addition of \emph{cyclic structure} on our simplices.
In the name of brevity and clarity, we are reluctant to introduce the necessary extra details here, and instead give a sketch of the idea and leave the formalism to appear in future work.

For the convenience of the reader, we reproduce here conditions (STC~1) to (STC~4) from \cite[\S3]{TT1986}; for an explanation of the notation, see \cref{appendix:proof-of-tot0-stwist-is-simplicial-twcs}.
\begin{quote}
  \itshape
  \begin{enumerate}
    \item[(STC~1)]
      For each vertex $\alpha$ and each $\sigma\geqslant\langle\alpha\rangle$ a free resolution $E_{\sigma,\alpha}^\bullet$ of $\mathscr{S}|U_\sigma$:
      \[
        0
        \to E_{\sigma,\alpha}^{-n}
        \xrightarrow{{}^\sigma\mathfrak{a}_\alpha^{0,1}} \ldots
        \xrightarrow{{}^\sigma\mathfrak{a}_\alpha^{0,1}} E_{\sigma,\alpha}^0
        \to \mathscr{S}|U_\sigma
        \to 0
      \]
    \item[(STC~2)]
      For each $\sigma=\langle\alpha_0\ldots\alpha_p\rangle$ a twisting cochain
      \[
        {}^\sigma\mathfrak{a}
        = {}^\sigma\mathfrak{a}^{0,1}
        + {}^\sigma\mathfrak{a}^{1,0}
        + {}^\sigma\mathfrak{a}^{2,-1}
        + \ldots
      \]
      acting on the resolutions $E_{\sigma,\alpha_i}^\bullet$
    \item[(STC~3)]
      Whenever $\sigma<\tau$ and $\alpha$ a vertex of $\sigma$, a short exact sequence of free modules
      \[
        0
        \to E_{\sigma,\alpha}^\bullet|U_\tau
        \to E_{\tau,\alpha}^\bullet
        \to E_{\sigma,\tau,\alpha}^\bullet
        \to 0
      \]
      as in (SB~4) (\loccit)
    \item[(STC~4)]
      Whenever $\sigma<\tau$ and $\alpha_0,\ldots,\alpha_k$ any vertices of $\sigma$ (repetitions allowed!), the matrix ${}^\tau\mathfrak{a}_{\alpha_0\ldots\alpha_k}^{k,1-k}$ is of the form
      \begin{enumerate}[(i)]
        \item for $k=0$, ${}^\tau\mathfrak{a}_{\alpha}^{0,1}={}^\sigma\mathfrak{a}_{\alpha}^{0,1}\oplus\text{elementary sequences}$
        \item for $k=1$, ${}^\tau\mathfrak{a}_{\alpha\beta}^{1,0}=\begin{pmatrix}{}^\sigma\mathfrak{a}_{\alpha\beta}^{1,0}&*\\0&1\end{pmatrix}$
        \item for $k>1$, ${}^\tau\mathfrak{a}_{\alpha_0\ldots\alpha_k}^{k,1-k}=\begin{pmatrix}{}^\sigma\mathfrak{a}_{\alpha_0\ldots\alpha_k}^{k,1-k}&*\\0&0\end{pmatrix}$
      \end{enumerate}
  \end{enumerate}
\end{quote}

By design, it is \emph{not} true that simplices of $\sTwist$ are exactly simplicial twisting cochains: to get such a result, we need to apply Čech totalisation, and we then recover simplicial twisting cochains as the $0$-simplices (\cref{theorem:tot0-stwist-is-simplicial-twcs}).
If, however, we apply a suitable change of notation, then the simplices of $\sTwist$ behave ``sufficiently similar' to simplicial twisting cochains that we could try to directly apply the version of Green's construction given in \cite[\S3]{TT1986}, since they satisfy conditions (STC~1) to (STC~4) \loccit\@, which is all that is necessary.
Indeed, as we will see, the $p$-simplices of $\sTwist(U)$ are \emph{local} and \emph{truncated} simplicial twisting cochains: they live over a single space $U$ instead of a cover, and their bidegree-$(k,1-k)$ terms are zero for $k>p$.
This is exactly why we can prove \cref{theorem:tot0-stwist-is-simplicial-twcs}: pulling back along the opposite of the Čech nerve lets us remove the adjective ``local'' by resolving $U$ by a cover, and taking the totalisation (which is a way of computing the homotopy limit) lets us remove the adjective ``truncated'' by gluing together infinitely many $p$-simplices for all $p\in\mathbb{N}$.
However, there is one very important caveat hidden in this adjective ``local'' which is what prevents us from formally constructing this morphism, namely that our change of notation will only give us \emph{ordered} and \emph{non-degenerate} terms, as we will explain.

So let $\tau\in\sTwist(U)_p$ be a $p$-simplex, which is exactly a GTT-labelling of $\Delta[p]$ by $\core{\dgnerve\Free(U)}$.
For $\alpha\in[p]$ and $\sigma$ a subset of $[p]$ containing $\alpha$, set
\[
  E_{\sigma,\alpha}^\bullet
  \coloneqq C_\alpha(\sigma)
\]
and, for $\tau\subseteq\sigma$ a sub-face also containing $\alpha$, set
\[
  E_{\sigma,\tau,\alpha}^\bullet
  \coloneqq C_\alpha^{\perp\sigma}(\tau).
\]
For $\sigma=\{\alpha_0<\ldots<\alpha_k\}$ and any non-empty subset $J=\{\beta_0<\ldots<\beta_\ell\}\subseteq\sigma$, set
\[
  {}^{\sigma}\mathfrak{a}_{\beta_0\ldots\beta_\ell}^{\ell,1-\ell}
  \coloneqq \varphi_J(\sigma)
\]
and define
\[
  {}^\sigma\mathfrak{a}
  \coloneqq \sum_{i=0}^k {}^{\sigma}\mathfrak{a}^{i,1-i}.
\]

The fact that $\tau$ is a GTT-labelling (\cref{definition:GTT-labelling}) immediately tells us that this change of notation gives objects that satisfy conditions (STC~3) and (STC~4) (which posit, respectively, the existence of elementary orthogonal complements and that higher homotopies be upper triangular, as in \cref{remark:comments-on-GTT-definition}).
It remains only to check three things:
\begin{enumerate}
  \item that conditions (STC~1) and (STC~2) are also satisfied (where the former asks that each $E_{\sigma,\alpha}^\bullet$ be a free resolution of a given coherent sheaf, and the latter asks that the ${}^\sigma\mathfrak{a}$ satisfy the Maurer--Cartan equation);
  \item that the result of the inductive construction defined there gives a $p$-simplex of $\Green(U)$, i.e. that the codomain of the morphism thus constructed is indeed $\Green$; and
  \item that the construction is functorial, and thus defines a morphism of simplicial presheaves.
\end{enumerate}
The second would follow rather immediately from \cref{remark:green-as-a-specific-case-of-stwist}, and the third would be a lengthy but uninspired explicit calculation; it is the first that poses a technical problem.

Condition~(STC~1) asks that the $E_{\sigma,\alpha}^\bullet$ be local resolutions of some given coherent sheaf $\scr{S}$ restricted to $U_\sigma$.
In particular, this implies that the $E_{\sigma,\alpha}^\bullet$ be \emph{exact} in all but the highest degree.
However, this condition is not actually necessary for the construction of Green's resolution: what is important is that the construction does not modify the internal homology of the complexes, i.e. the resulting object still resolves the same coherent sheaf $\scr{S}$.
Indeed, \cite{Wei2016} deals with twisting cochains that are not exact --- in general, these correspond to resolutions of \emph{complexes} of coherent sheaves.
As mentioned in \cref{section:summary-and-future-work}, a better understanding of complexes of coherent (analytic) sheaves in one of the main motivations for this present work.

Condition~(STC~2) is where the real problem lies.
This condition concerns two types of terms outside of those that arise from our tentative morphism: unordered ones, such as $\mathfrak{a}_{\beta\alpha}^{0,1}$, and degenerate ones, such as $\mathfrak{a}_{\alpha\beta\alpha}^{2,-1}$.
It seems possible to recover the degenerate terms by applying degeneracy maps, but the unordered ones are simply extra information that is not given in our current framework.
However, by replacing $\Delta[p]$ with some ``thicker'' structures in our definitions, it seems that we could resolve this problem.
Before formalising this, let us first motivate what we mean.

If we wish to describe a morphism $f$ between two objects $x$ and $y$, then we might label a $1$-simplex accordingly: label $\{0\}$ with $x$, $\{1\}$ with $y$, and $\{0<1\}$ with $f$.
Now suppose we wish to record the fact that $f$ is actually invertible up to homotopy, such as in the case of a quasi-isomorphism between complexes of free modules, or a chain homotopy in general.
Then it makes sense to actually introduce \emph{another} $1$-cell connecting $\{0\}$ and $\{1\}$ with the opposite orientation, and to label it with $f^{-1}$.
But now we have a non-contractible space: we have constructed a copy of the circle $S^1$.
So to ensure that we have something homotopically equivalent to our original description, we might try introducing a $2$-cell bounded by the two $1$-simplices, filling in $S^1$ with a disc to obtain $D^2$.
This is indeed now contractible, but what information does this $2$-cell describe?
Since it goes from one $1$-cell to the other, it seems justifiable to label it with the data of the homotopy $f^{-1}f\Rightarrow\id_x$.
But now we have the same asymmetry as we did at the start, and we should introduce another $2$-cell with the opposite orientation, labelled with $\id_y\Rightarrow ff^{-1}$.
Again, however, we have created a sphere, namely $S^2$, and want to fill this with a $3$-cell in order for it to remain contractible.
Repeating this process indefinitely, labelling with higher and higher homotopical data, we see that to ``really'' describe the data of a morphism with homotopy inverse, we want to replace $\Delta[1]$ by $S^\infty$.
More generally, this leads to the idea of replacing $\Delta[p]$ by the nerve of the groupoid with $(p+1)$-many objects and a unique (invertible) morphism between any two objects.

\begin{conjecture}
\label{conjecture:TTs-green-gives-morphism-stwist-to-green}
  If we define $\Green'$ and $\sTwist'$ by replacing $\Delta[1]$ with $S^\infty$, as described above (and analogously for higher simplices), after taking the pair subdivision in the construction of $\Green$ and $\sTwist$ (respectively), then the construction of Green's resolution described in \cite{TT1986} defines a morphism of simplicial presheaves $\sTwist\simeq\sTwist'\to\Green'\simeq\Green$.
\end{conjecture}

\section{Future work}
\label{section:summary-and-future-work}

In this paper we have shown how to construct generalisations of $\core{\nerve\Free(U)}$, as motivated in \cref{subsection:narrative} --- one version using the dg-nerve and one using a simplicial labelling construction.
The objects resulting from the former can be thought of as abstract twisting cochains, and from the latter as abstract Green complexes; this is justified by \cref{theorem:tot0-all-three} which tells us that we do indeed recover the corresponding classical objects when passing to the Čech totalisation in the holomorphic setting.

We have shown that the connected components of these two simplicial presheaves, $\Twist$ and $\Green$, are in bijection (\cref{theorem:pi0-of-twist-and-pi0-of-stwist} and \cref{theorem:pi0-of-green-and-pi0-of-stwist}), via the construction of their common generalisation $\sTwist$, which can be thought of as the abstractification of simplicial twisting cochains.
This leads us to possibly the largest unanswered question that we have raised:
\begin{quote}
  \itshape
  What is the full description of the relationship between $\Twist$, $\Green$, and $\sTwist$ in terms of the homotopy theory of simplicial presheaves?
\end{quote}
We hope, as described in \cref{conjecture:green-into-stwist}, that $\Green$ and $\sTwist$ can be shown to be weakly equivalent, using Green's resolution, as described in \cref{conjecture:TTs-green-gives-morphism-stwist-to-green}.
However, the question of whether or not $\Twist$ and $\sTwist$ are weakly equivalent is much more open; if it is indeed the case that $\Green\simeq\sTwist$, then the weak equivalence $\Twist\simeq\sTwist$ would give us a weak equivalence between $\Twist$ and $\Green$, and this would seem to have rather far-reaching implications, as we now explain.

\bigskip

We know that weak equivalence of globally fibrant simplicial presheaves is preserved by Čech totalisation (\cref{lemma:tot-of-we-of-kan-is-we}), and so if $\Twist$, $\Green$, and $\sTwist$ are indeed all weakly equivalent \emph{and also all globally fibrant}, then so too are their Čech totalisations.
Although we do not prove (nor even claim) that this is indeed the case, we have seen a very weak version of this statement, namely \cref{corollary:pi0-all-three-equivalent}, which says that
\[
  \pi_0\Tot\Twist(\cechnerve\cover_\bullet) \cong \pi_0\Tot\Green(\cechnerve\cover_\bullet) \cong \pi_0\Tot\sTwist(\cechnerve\cover_\bullet).
\]
In the case where $(X,\OO_X)$ is a complex-analytic manifold with Stein cover $\cover$, this first equivalence tells us that twisting cochains up to weak equivalence are ``the same as'' Green complexes up to weak equivalence (whatever this might mean for Green complexes, cf. \cref{subsection:1-simplices-in-complex-green}).

The relation between twisting cochains and perfect complexes is rather well studied, with \cite{Wei2016} showing that the dg-category of the former gives a dg-enhancement of the latter (see also \cite{GMTZ2022a}), and it is a classical fact that perfect complexes allow access to the derived category of coherent sheaves: if $(X,\OO_X)$ is smooth, then there is an equivalence of triangulated categories between that of perfect complexes on $X$ and the bounded derived category of coherent sheaves on $X$ (\cite[Exposé~I, Corollaire~5.10 and Exemples~5.11]{BGI1971}).
On the other hand, Green complexes allow us to resolve not just coherent sheaves, but actually complexes of sheaves of $\OO_X$-modules whose cohomology consists of coherent modules: if $(X,\OO_X)$ is a complex-analytic manifold, then there is an equivalence of $(\infty,1)$-categories between Green complexes on $X$ (after taking a homotopy colimit over refinements of covers) and complexes of sheaves of $\OO_X$-modules with coherent cohomology (\cite[Corollary~3.21 and Lemma~4.36]{Hos2020a}).

The question of whether or not the derived category of coherent sheaves is equivalent to the category of complexes of sheaves with coherent cohomology is a long-standing problem in the complex-analytic setting, and it seems as though this framework describing the simplicial presheaves $\Twist$ and $\Green$ (and their common generalisation $\sTwist$) might provide another way of approaching this question.

One specific aspect of $\Green$ that first needs to be better understood is whether or not it is globally fibrant (or, at least, for which specific spaces its Čech totalisation is fibrant).
Then one can try to relate this simplicial presheaf to the $(\infty,1)$-category of Green complexes described in \cite{Hos2020a}, with the hope being that there is an equivalence between the cores of the two.
At the moment, however, one can still start to phrase statements of this flavour: for example, \cite[Corollary~3.21 and Lemma~4.36]{Hos2020a} tell us that, given a complex with coherent cohomology, we can resolve it (up to quasi-isomorphism) by a Green complex, i.e. there is a surjection \emph{of sets} from $\Tot^0\Green(\cechnerve\cover_\bullet)$ to the set of quasi-isomorphism classes of complexes with coherent cohomology.

In other words, it would be useful to construct the two horizontal morphisms in the diagram
\[
  \begin{tikzcd}
    \Tot\Twist(\cechnerve\cover_\bullet)
      \ar[r] \ar[d]
    & D^\mathrm{b}(\mathsf{Coh}(X))
  \\\Tot\sTwist(\cechnerve\cover_\bullet)
  \\\Tot\Green(\cechnerve\cover_\bullet)
      \ar[r] \ar[u]
    & \mathsf{CCoh}(X)
  \end{tikzcd}
\]
and study their properties.

\bigskip

Finally, this machinery should allow us to study equivariant theories.
Given a group $G$ and a $G$-space $X$ with suitable cover $\cover$, we can consider the bisimplicial space given by the natural combination of the Čech nerve of $\cover$ and the bar construction of the $G$-action;
the diagonal of this should give a suitable simplicial space on which to define simplicial presheaves analogous to those considered in this present work.
However, in order to provide the full technical details, one first needs to understand, for example, what cofibrant replacements look like in this setting.


\part*{Appendices}
\addcontentsline{toc}{part}{Appendices}
\appendix

\renewcommand{\thefigure}{\Alph{section}.\arabic{subsection}.\roman{figure}}

\section{Motivation: principal bundles}
\label{appendix:example}

In this appendix we show how the abstract machinery of \cref{subsection:abstract-machinery-for-example} can be used to recover principal $\GL_n(\mathbb{R})$-bundles.

Using the notation of \cref{subsection:abstract-machinery-for-example}, take $G=\GL_n(\mathbb{R})\in\LieGroup$ (which we write as $\GL_n$ for brevity), and consider the presheaf of simplicial sets
\[
  \Bun \coloneqq \Tot(\bigfunctor)(\GL_n) \in [\SmthC^\op,\sSet].
\]
Let $\underline{X}=(X,\{U_\alpha\})\in\SmthC$.

\bigskip
\textbf{Question.}
\emph{What does the space $\Bun(\underline{X})$ look like?}

\smallskip

\textbf{Answer.}
{It classifies principal $\GL_n(\mathbb{R})$-bundles on $X$.}

\bigskip

To see why this is the answer (and to understand what exactly we mean by ``classifies''), we can use \cref{subsection:totalisation-and-holim} to explicitly describe the points and lines in $\Bun(\underline{X})$, as well as higher simplices.
Since we are working with simplicial sets as a model for spaces, we should gain some understanding from calculating the \emph{homotopy groups} of $\Bun(\underline{X})$.

\begin{remark}
  Our choice of $G=\GL_n(\mathbb{R})$ is somewhat arbitrary: the following construction works regardless of the choice of group, and would give principal $G$-bundles in general.
  Indeed, we could take $G=\GL_n(\mathbb{C})$, in the category of \emph{Stein} (instead of smooth) manifolds, and recover \emph{holomorphic} vector bundles.
\end{remark}

\subsection{Points in $\Bun(\underline{X})$}
\label{subsection:points-in-bun-x}

Let us start by considering what a \emph{point}  $v\in\Bun(\underline{X})_0$ is.
It consists of components $v^p\in\bigfunctor(\GL_n)(\underline{X})_p^p$ for $p\in\mathbb{N}$, subject to certain conditions.
First we describe the data of each $v^p$, and then we describe the conditions.

\begin{enumerate}
  \item[($p=0$)] By definition, $v^0\in\bigfunctor(\GL_n)(\underline{X})_0^0$.
    Unravelling the definition of $\bigfunctor$, and using the fact that $\mathbb{B}$ is a right adjoint, and thus preserves limits, this means that $v^0$ is a collection of elements
    \[
      \begin{aligned}
        \left(\nerve\mathbb{B}\yon(\GL_n)(\sqcup_\alpha U_\alpha)\right)_0
        &= \mathbb{B}\yon(\GL_n)(\sqcup_\alpha U_\alpha)
      \\&= \mathbb{B}\left(\Smth(\sqcup_\alpha U_\alpha,\GL_n)\right)
      \\&= \prod_\alpha\mathbb{B}\left(\Smth(U_\alpha,\GL_n)\right)
      \end{aligned}
    \]
    for all $\alpha$.
    But this is a product over one-element groupoids, and so $v^0$ is simply the (uniquely determined) collection
    \[
      v^0 =
      \bigg\{
        *_\alpha \in \mathbb{B}\big(\Smth(U_\alpha,\GL_n)\big)
      \bigg\}_\alpha.
    \]
    However, since the objects of a category are in bijection with the identity morphisms, we can instead think of this as
    \[
      v^0 =
      \bigg\{
        \id_{*_\alpha} \in \Smth(U_\alpha,\GL_n)
      \bigg\}_\alpha
    \]
    where $\id_{*_\alpha}$ is the constant map to the identity, i.e. $\id_{*_\alpha}\colon x\mapsto\id$ for all $x\in U_\alpha$.
  \item[($p=1$)] Again, unravelling the definition of $\bigfunctor$, we see that $v^1$ is an element of
    \[
      \begin{aligned}
        \left(\bigfunctor(\GL_n)(\underline{X})\right)_1^1
        &=\left(\nerve\mathbb{B}\yon(\GL_n)(\sqcup_\alpha U_\alpha)\right)_1
      \\&= \operatorname{Mor}\left(\mathbb{B}\yon(\GL_n)(\sqcup_{\alpha,\beta} U_{\alpha\beta})\right)
      \\&= \yon(\GL_n)(\sqcup_{\alpha,\beta} U_{\alpha\beta})
      \\&= \Smth(\sqcup_{\alpha,\beta}U_{\alpha\beta},\GL_n)
      \\&\cong \prod_{\alpha,\beta} \Smth(U_{\alpha\beta},\GL_n),
      \end{aligned}
    \]
    i.e. $v^1$ is a collection $\{g_{\alpha\beta}\colon U_{\alpha\beta}\to\GL_n\}_{\alpha,\beta}$ of smooth maps.
  \item[($p=2$)]
    Since the nerve of a category is generated by its $1$-simplices, an element of $(\bigfunctor(\GL_n)(\underline{X}))_2^2$ is given by a pair of composible morphisms, along with their composite:
    \[
      \left(\bigfunctor(\GL_n)(\underline{X})\right)_2^2
      = \bigg\{ (g_{\alpha\beta\gamma}^{(1)},g_{\alpha\beta\gamma}^{(2)},g_{\alpha\beta\gamma}^{(2)}\circ g_{\alpha\beta\gamma}^{(1)}) \mid g_{\alpha\beta\gamma}^{(1)},g_{\alpha\beta\gamma}^{(2)}\in \operatorname{Mor}\left(\mathbb{B}\yon(\GL_n)(\sqcup_{\alpha,\beta,\gamma} U_{\alpha\beta\gamma})\right) \bigg\}.
    \]
    But note that, since we are working with a one-element groupoid, all morphisms are composable, and the composition is uniquely determined;
    just as for the case for $v^1$, this means that $v^2$ is simply a collection of pairs of smooth maps (along with their composition)
    \[
      v^2 = \Big\{
        \left(
          g_{\alpha\beta\gamma}^{(1)},g_{\alpha\beta\gamma}^{(2)},g_{\alpha\beta\gamma}^{(2)}\circ g_{\alpha\beta\gamma}^{(1)}
        \right)
        \mid
        g_{\alpha\beta\gamma}^{(1)},g_{\alpha\beta\gamma}^{(2)}\colon U_{\alpha\beta\gamma}\to\GL_n
      \Big\}_{\alpha,\beta,\gamma}.
    \]
  \item[($p\geq3$)] In general, each $v^p$ for $p\geq3$ will be a collection of $p$-tuples of smooth maps
    \[
      v^p = \Big\{
        \left(
          g_{\alpha_0\ldots\alpha_p}^{(1)},\ldots,g_{\alpha_0\ldots\alpha_p}^{(p)}
        \right)
        \mid
        g_{\alpha_0\ldots\alpha_p}^{(1)},\ldots,g_{\alpha_0\ldots\alpha_p}^{(p)}\colon U_{\alpha_0\ldots\alpha_p}\to\GL_n
      \Big\}_{\alpha_0,\ldots,\alpha_p}
    \]
    where we omit the compositions simply for ease of notation.
\end{enumerate}

Now, as mentioned in \cref{subsection:totalisation-and-holim}, the $v^p$ also satisfy some conditions, given by the face and degeneracy maps in the \emph{opposite} of the Čech nerve.
The important thing here is indeed the word ``opposite''.
For example, one condition is that $f_1^0(v^0)$ and $f_1^1(v^0)$ must be the two vertices of $v^1$.
We know that the usual Čech face maps act via $U_{\alpha\beta}\mapsto U_\alpha$ and $U_{\alpha\beta}\mapsto U_\beta$, but, since we are considering the \emph{opposite} of the Čech nerve, we understand the condition in question \emph{not} by looking at the image of some $g_{\alpha\beta}$, but \emph{instead} by fixing some $\alpha\beta$, and then looking at the $g_{\alpha'}$ for all $\alpha'$ that are mapped to from $\alpha\beta$ under the usual face maps, i.e. exactly for $\alpha'\in\{\alpha,\beta\}$.

We start with the degeneracy maps $s_i^p$.

\begin{itemize}
  \item[($p=0$)]
    Here we have only one condition: $s_0^0(v^0)=v^1$, corresponding to the degeneracy map given by $\alpha\mapsto\alpha\alpha$.
    This condition thus reduces to asking that
    \[
      g_{\alpha\alpha} = \id_{*_\alpha}
    \]
    for all $\alpha$.
  \item[($p=1$)]
    The degeneracy map $s_0^1$ acts via $\alpha\beta\mapsto\alpha\alpha\beta$; the map $s_1^1$ via $\alpha\beta\mapsto\alpha\beta\beta$.
    The corresponding conditions thus reduce to asking that
    \[
      g_{\alpha\alpha\beta} = g_{\alpha\beta\beta} = g_{\alpha\beta}
    \]
    for all $\alpha,\beta$.
  \item[($p\geq2$)]
    More generally, the degeneracy maps $s_i^p$ for $p\geq2$ will give analogous conditions to the $p=1$ case: any $g_{\alpha_0\ldots\alpha_i\alpha_i\ldots\alpha_p}$ is equal to $g_{\alpha_0\ldots\alpha_i\ldots\alpha_p}$, i.e. we can always remove repeated indices without changing the map.
\end{itemize}

Now for the face maps $f_p^i$.

\begin{itemize}
  \item[($p=1$)]
    The face map $f_1^0$ acts via $\alpha\beta\mapsto\beta$; the map $f_1^1$ via $\alpha\beta\mapsto\alpha$.
    The condition given by these face maps is that $f_1^0(v^0)$ and $f_1^1(v^0)$ should be the endpoints of $v^1$, i.e. that the line labelled by $g_{\alpha\beta}$ should go from $*_\alpha$ to $*_\beta$.
  \item[($p=2$)]
    The face map $f_2^0$ acts via $\alpha\beta\gamma\mapsto\beta\gamma$; the map $f_2^1$ via $\alpha\beta\gamma\mapsto\alpha\gamma$; and the map $f_2^2$ via $\alpha\beta\gamma\mapsto\beta\gamma$.
    Asking for the three edges of $v^2$ to be given by the images of $v^1$ under these face maps thus reduces to asking for
    \[
      \begin{aligned}
        g_{\alpha\beta\gamma}^{(1)}
        &= f_2^2(g_{\alpha\beta})
        \coloneqq g_{\alpha\beta}|U_{\alpha\beta\gamma}
      \\g_{\alpha\beta\gamma}^{(2)}
        &= f_2^0(g_{\beta\gamma})
        \coloneqq g_{\beta\gamma}|U_{\alpha\beta\gamma}
      \\g_{\alpha\beta\gamma}^{(2)}\circ g_{\alpha\beta\gamma}^{(1)}
        &= f_2^1(g_{\alpha\gamma})
        \coloneqq g_{\alpha\gamma}|U_{\alpha\beta\gamma}.
      \end{aligned}
    \]
    But note that composition in $\mathbb{B}\Smth(U_{\alpha\beta\gamma},\GL_n)$ is given by multiplication in $\GL_n$:
    \[
      (h\circ g)(x) \coloneqq h(x)\cdot g(x)
    \]
    and so the above conditions simplify to
    \[
      g_{\beta\gamma}g_{\alpha\beta} = g_{\alpha\gamma}
    \]
    (where we omit the restrictions from our notation).
  \item[($p\geq3$)]
    Since the nerve is generated by $1$-simplices, and since the multiplication in $\GL_n$ is (strictly) associative, the higher face maps give us no further conditions.
\end{itemize}

In summary, a point in $\Bun(\underline{X})$ consists of smooth maps $g_{\alpha\beta}\colon U_{\alpha\beta}\to\GL_n$ that satisfy the identity condition $g_{\alpha\alpha}=\id_{*_\alpha}$ and the cocycle condition $g_{\beta\gamma}g_{\alpha\beta}=g_{\alpha\gamma}$.
More concisely,
\begin{quote}
  \itshape
  a point in $\Bun(\underline{X})$ is exactly the data of a principal $\GL_n$-bundle \mbox{on $X$}.
\end{quote}

\subsection{Edges in $\Bun(\underline{X})$}
\label{subsection:lines-in-bun-x}

We can describe \emph{lines} in a similar way: a line $l\in\Bun(\underline{X})$ is given by
\[
  \left(
    l^p\colon\Delta[1]\times\Delta[p]\to\bigfunctor(\GL_n)(\underline{X})_\bullet^p
  \right)_{p\in\mathbb{N}}
\]
satisfying some conditions.
Geometrically,
\begin{itemize}
  \item $l^0$ is a line in $\bigfunctor(\GL_n)(\underline{X})^0$;
  \item $l^1$ is a square (with diagonal) in $\bigfunctor(\GL_n)(\underline{X})^1$;
  \item $l^2$ is a (triangulated) triangular prism in $\bigfunctor(\GL_n)(\underline{X})^2$;
  \item \ldots and ``so on''.
\end{itemize}
We now make this more precise, although in a slightly more condensed way than we did for simplices, studying the face and degeneracy conditions as we go, rather than separately afterwards.

\begin{enumerate}
  \item[($p=0$)]
    Note that
    \[
      \left[
        \Delta[1]\times\Delta[0],\bigfunctor(\GL_n)(\underline{X})_\bullet^0
      \right]
      \cong\bigfunctor(\GL_n)(\underline{X})_1^0
    \]
    and, by definition,
    \[
      \begin{aligned}
        \bigfunctor(\GL_n)(\underline{X})_1^0
        &= \left(\nerve\mathbb{B}\yon(\GL_n)(\sqcup_\alpha U_\alpha)\right)_1
      \\&= \operatorname{Mor}\left(\mathbb{B}\yon(\GL_n)(\sqcup_\alpha U_\alpha)\right)
      \\&=\prod_\alpha\Smth(U_\alpha,\GL_n)
      \end{aligned}
    \]
    so that $l^0$ consists of a morphism $\lambda_\alpha\colon U_\alpha\to\GL_n$ for all $U_\alpha$.
    As for vertices in the totalisation, the face maps tell us that the $1$-simplex $\lambda_\alpha$ has repeated endpoints $*_\alpha$ and $*_\alpha$.
  \item[($p=1$)]
    At this point, it becomes simpler to draw a diagram describing $l^1$, which is a map from $\Delta[1]\times\Delta[1]$, and thus looks like a square with diagonal, or two $2$-simplices glued together along a common edge:
    \[
      \begin{tikzpicture}
        \node (v00) at (0,0) {$*_\alpha$};
        \node (v10) at (2,0) {$*_\beta$};
        \node (v01) at (0,2) {$*_\alpha$};
        \node (v11) at (2,2) {$*_\beta$};
        \draw[thick,-latex] (v00) to node[label=left:{$\lambda_\alpha$}]{} (v01);
        \draw[thick,-latex] (v10) to node[label=right:{$\lambda_\beta$}]{} (v11);
        \draw[thick,-latex] (v01) to node[label=above:{$h_{\alpha\beta}$}]{} (v11);
        \draw[thick,-latex] (v00) to node[label=below:{$g_{\alpha\beta}$}]{} (v10);
        \draw[thick,-latex] (v00) to (v11);
      \end{tikzpicture}
    \]
    where the face maps tell us that the vertical edges are exactly the $\lambda_\alpha$ and $\lambda_\beta$ from the $(p=0)$ case above, and where we already know that the horizontal edges are of the form $g_{\alpha\beta}$ and $h_{\alpha\beta}$ from our study of the vertices.
    Indeed, this is the general pattern that arises: a $1$-simplex in the totalisation contains, in particular, the data of two $0$-simplex, namely its endpoints.
    This makes precise the sense in which a $1$-simplex should describe a morphism between two bundles $g_{\alpha\beta}$ and $h_{\alpha\beta}$.
    Since this diagram takes values in the nerve, the diagonal is simply labelled with the common composite $h_{\alpha\beta}\circ\lambda_\alpha=\lambda_\beta\circ g_{\alpha\beta}$, where composition is given by multiplication in $\GL_n$, so that
    \[
      h_{\alpha\beta}\lambda_\alpha = \lambda_\beta g_{\alpha\beta}.
    \]
    Note that defining $\lambda_\alpha^{-1}(x)$ to be $\lambda_\alpha(x)^{-1}$ shows that $\lambda_\alpha$ is invertible.
    The degeneracy map gives no extra information here, since it simply says that $\id_\alpha\circ\lambda_\alpha=\lambda_\alpha\circ\id_\alpha$.
  \item[($p\geq2$)]
    For $p\geq2$, there are no extra non-trivial conditions or data.
    This is entirely analogous to how we only needed to study $p\leq2$ in the case of vertices: since the nerve is generated by its $1$-simplices, all the interesting information is concentrated in degrees $0$, $1$, and $2$; since here the degrees shift by $1$ (i.e. $l^0$ consists of a line, not just a point), all the interesting information is concentrated in degrees $0$ and $1$.
\end{enumerate}

In summary then, a line in $\Bun(\underline{X})$ consists of smooth maps $U_\alpha\to\GL_n$ that are invertible and commute with transition maps $g_{\alpha\beta},h_{\alpha\beta}$ given by its endpoints.
More concisely,
\begin{quote}
  \itshape
  a line in $\Bun(\underline{X})$ is exactly the data of a morphism (which is necessarily an isomorphism) between two principal $\GL_n$-bundles on $X$.
\end{quote}

\begin{remark}
  If we did not take the maximal Kan complex in the definition of $\Bun$, then the objects would be exactly the same: the fact that the $g_{\alpha\beta}$ are isomorphisms is implied by the cocycle condition $g_{\beta\gamma}g_{\alpha\beta}=g_{\alpha\gamma}$ and the degeneracy condition $g_{\alpha\alpha}=\id$.
  The morphisms, however, would then simply be morphisms of bundles, not isomorphisms.
  This would give us a \emph{quasi-category} of principal $\GL_n$-bundles instead of a space, which is perfectly usable, but for which the simplicial homotopy groups (\cref{subsection:simplicial-homotopy-theory}) are not a priori well defined: one would have to first pass to a fibrant replacement.
\end{remark}

\subsection{Higher simplices in $\Bun(\underline{X})$}

Since the nerve of a groupoid is $2$-coskeletal, we only need to study the $p$-simplices for $p\leq2$.
But, as already mentioned, the $2$-simplices are uniquely determined by the $1$-simplices on their boundary, which means that
\[
  \pi_1(\Bun(\underline{X})) \cong \Bun(\underline{X})_1
\]
and so we can simply ignore all higher simplices in $\Bun(\underline{X})$.

\subsection{Homotopy groups of $\Bun(\underline{X})$}

In summary, points in $\Bun(\underline{X})$ are principal $\GL_n$-bundles on $X$, and lines in $\Bun(\underline{X})$ are bundle isomorphisms.
Thus\footnote{Here we are making implicit use of the fact that, for Kan complexes, it suffices to work with \emph{simplicial} homotopy groups instead of taking the geometric realisation, as explained in \cref{subsection:simplicial-homotopy-theory}.}
\begin{quote}
  \itshape
  $\pi_0(\Bun(\underline{X}))$ consists of isomorphism classes of principal $\GL_n$-bundles on $X$.
\end{quote}
Furthermore, for any principal $\GL_n$-bundle $E$ on $X$,
\begin{quote}
  \itshape
  $\pi_1(\Bun(\underline{X}),E)$ is the gauge group $\operatorname{Aut}(E)$ of $E$.
\end{quote}
Finally, since $\Bun(\underline{X})$ is $2$-coskeletal, we know that all higher $\pi_n$ (i.e. for $n\geq2$) are zero.

\section{Technical proofs}
\label{appendix:technical-proofs}

Some of the proofs in this paper are so laden with notation (multiple indices, lower-dimensional faces of simplices, etc.) that they obfuscate the actual ideas from which they were constructed.
Because of this, we have decided to place them in an appendix;
the reader is invited to carefully check all details, and to verify that we have not mislead them in claiming that these proofs are more technical than they are interesting.

\subsection{Proof of \cref{theorem:tot0-stwist-is-simplicial-twcs} --- $\Tot^0\sTwist(\cechnerve\cover_\bullet)$ recovers simplicial twisting cochains}
\label{appendix:proof-of-tot0-stwist-is-simplicial-twcs}

The data of a $0$-simplex $v$ in $\Tot\sTwist(\cechnerve\cover_\bullet)$ is the data of $v^p\in\sTwist(\cechnerve\cover_\bullet)_p^p$ for $p\in\mathbb{N}$, subject to face and degeneracy conditions.

By definition,
\[
  v^0\in
  \sTwist(\cechnerve\cover_\bullet)_0^0
  = \sTwist(\sqcup_\alpha U_\alpha)_0
\]
which is the set of GTT-labellings of $\Delta[0]$ by $\core{\dgnerve\Free(\sqcup_\alpha U_\alpha)}$.
But, as in the proof of \cref{theorem:tot0-twist-is-twcs}, since $\OO(\sqcup_\alpha U_\alpha)\cong\prod_\alpha\OO(U_\alpha)$, a free $\OO(\sqcup_\alpha U_\alpha)$-module is exactly the data of a free $\OO(U_\alpha)$-module for all $\alpha$.
Since both the dg-nerve and the core functor are right adjoints, their composition preserves all limits.
In summary then,
\[
  \core{\dgnerve\Free(\sqcup_\alpha U_\alpha)}
  \cong \prod_\alpha\core{\dgnerve\Free(U_\alpha)}
\]
and so $v^0$ is an element of the set of GTT-labellings of $\Delta[0]$ by $\prod_\alpha\core{\dgnerve\Free(U_\alpha)}$, which simply means a choice of complex $C_\alpha\in\Free(U_\alpha)$ for each $U_\alpha$.

Next,
\[
  v^1\in
  \sTwist(\cechnerve\cover_\bullet)_1^1
  =\sTwist(\sqcup_{\alpha\beta}U_{\alpha\beta})_1
\]
which is the set of GTT-labellings of $\Delta[1]$ by $\core{\dgnerve\Free(\sqcup_{\alpha\beta} U_{\alpha\beta})}$.
By the same argument as above,
\[
  \core{\dgnerve\Free(\sqcup_{\alpha\beta} U_{\alpha\beta})}
  \cong \prod_{\alpha\beta}\core{\dgnerve\Free(U_{\alpha\beta})}.
\]
Now, a GTT-labelling of $\Delta[1]$ by this simplicial set consists of a labelling of the three $0$-simplices $\{0\}$, $\{1\}$, and $\{0<1\}$, as well as of the two $1$-simplices $\{0\}\subset\{0<1\}$ and $\{1\}\subset\{0<1\}$, of $\pair{\Delta[1]}$, all such that the conditions of \cref{definition:GTT-labelling} are satisfied.
The face conditions of the totalisation tell us that, for each $U_{\alpha\beta}$, the two $0$-simplices $\{0\}$ and $\{1\}$ are labelled with the complexes $C_0(\alpha)\coloneqq C_\alpha$ and $C_1(\beta)\coloneqq C_\beta$ (respectively) from the $v^0$ above;
the remaining $0$-simplex $\{0<1\}$ is labelled with a quasi-isomorphism $C_0(\alpha\beta)\simfrom C_1(\alpha\beta)$ of complexes in $\Free(U_{\alpha\beta})$.
The two $1$-simplices $\{0\}\subset\{0<1\}$ and $\{1\}\subset\{0<1\}$ are labelled with elementary complexes $C_0^{\perp\alpha\beta}(\alpha)$ and $C_1^{\perp\alpha\beta}(\beta)$ (respectively) such that
\[
  \begin{aligned}
    C_0(\alpha\beta)
    &\cong C_0(\alpha) \oplus C_0^{\perp\alpha\beta}(\alpha)
  \\C_1(\alpha\beta)
    &\cong C_1(\beta) \oplus C_1^{\perp\alpha\beta}(\beta).
  \end{aligned}
\]
Note that the final condition of \cref{definition:GTT-labelling} is irrelevant here, since $k=1$.
The degeneracy map $\alpha\mapsto\alpha\alpha$ imposes the condition that the quasi-isomorphism $C_0(\alpha\alpha)\simfrom C_1(\alpha\alpha)$ be the identity, and that $C_0^{\perp\alpha\alpha}(\alpha)=C_1^{\perp\alpha\alpha}(\alpha)=0$.

We now pass immediately to $v^p\in\sTwist(\cechnerve\cover_\bullet)_p^p$ for arbitrary $p\geq2$.
The face conditions will tell us that all of the $(p-1)$-dimensional data of $v^p$ coincides with that already given by $v^0,\ldots,v^{p-1}$;
the degeneracy conditions will tell us that it suffices to consider non-degenerate intersections $U_{\alpha_0\ldots\alpha_p}$, since if $\alpha_i=\alpha_{i+1}$ for some $0\leq i<p$ then the corresponding edge (containing a $0$-simplex and two $1$-simplices) will be trivially labelled.
As a notational shorthand, given a simplex $I=\{i_0<\ldots<i_k\}$ of indices, we write $\alpha_I\coloneqq\alpha_{i_0}\ldots\alpha_{i_k}$.
Since we are evaluating on the Čech nerve, rather than labelling the simplices of $\pair{\Delta[p]}$ by subsets of $[p]$, it makes sense to label them with subsets of $\{\alpha_0<\ldots<\alpha_p\}$ for each fixed $U_{\alpha_0\ldots\alpha_p}\in\cechnerve\cover_p$.
Following \cref{definition:GTT-labelling}, each $0$-simplex $\alpha_{i_0}\ldots\alpha_{i_k}\leq\alpha_0\ldots\alpha_p$ is labelled with a $k$-simplex of $\core{\dgnerve\Free(U_{\alpha_I})}$, i.e. by complexes
\[
  C_{i_0}(\alpha_I),\ldots,C_{i_k}(\alpha_I)
\]
along with, for all non-empty subsets $J=\{{j_0}<\ldots<{j_\ell}\}\subseteq\{{i_0}<\ldots<{i_k}\}$, morphisms
\[
  \varphi_J(\alpha_I) \in \Hom_{\Free(U_{\alpha_I})}^{1-\ell}\big(C_{j_\ell}(\alpha_I),C_{j_0}(\alpha_I)\big)
\]
such that, for all $J$ with $|J|\geq3$,
\[
  \partial\varphi_J(\alpha_I)
  = \sum_{m=1}^{\ell-1} (-1)^{m-1} \varphi_{J\setminus\{j_m\}}(\alpha_I)
  + \sum_{m=1}^{\ell-1} (-1)^{\ell(m-1)+1} \varphi_{j_0<\ldots<j_m}(\alpha_I)\circ\varphi_{j_m<\ldots<j_k}(\alpha_I).
\]
Next, setting $I=\{i_0<\ldots<i_k\}$ and $J=\{j_0<\ldots<j_\ell\}\subset I$, each $(k-\ell)$-cell $\alpha_J<\alpha_I$ in $\pair{\Delta[p]}$ is labelled with an $(\ell+1)$-tuple of objects
\[
  \Big(
    C_{j_m}^{\perp\alpha_I}(\alpha_J)
    \in\Free(U_{\alpha_I})
  \Big)_{0\leq m\leq\ell}
\]
where each $C_{j_m}^{\perp\alpha_I}(\alpha_J)$ is elementary, and such that we obtain a direct-sum decomposition
\[
  C_{J_m}(\alpha_I) \cong C_{J_m}(\alpha_J)\oplus C_{J_m}^{\perp\alpha_I}(\alpha_J)
\]
for all $0\leq m\leq\ell$.

Now we introduce the change of notation that will recover the definition of simplicial twisting cochain from \cite[\S3]{TT1986}, which we have already seen in \cref{subsection:greens-resolution}.
Given $U_{\alpha_0\ldots\alpha_p}$, some specific $\alpha\coloneqq\alpha_i$, and some subset $I\subseteq[p]$ containing $i$, let $\sigma=\alpha_I$;
for $J\subseteq I$ also containing $i$, let $\tau=\alpha_J$, so that
\[
  \{\alpha_i\}
  =\{\alpha\}
  \subseteq\tau
  \subseteq\sigma
  \subseteq\alpha_{[p]}
  =\alpha_0\ldots\alpha_p.
\]
We then define
\[
  \begin{aligned}
    E_{\sigma,\alpha}^\bullet
    &\coloneqq C_i(\sigma)
  \\E_{\sigma,\tau,\alpha}^\bullet
    &\coloneqq C_i^{\perp\sigma}(\tau)
  \\{}^\sigma\mathfrak{a}_{\alpha_J}^{\ell,1-\ell}
    &\coloneqq \varphi_J(\sigma)
  \\{}^\sigma\mathfrak{a}
    &\coloneqq \sum_{i=0}^{k-1} {}^\sigma\mathfrak{a}^{i,1-i}
  \end{aligned}
\]
where $k=|\sigma|-1$ and $\ell=|\tau|-1$.
It remains to show that conditions (STC~1) to (STC~4) in \cite[\S3]{TT1986} are satisfied.
As mentioned in \cref{subsection:greens-resolution}, conditions~(STC~3) and (STC~4) are satisfied by construction, following the definition of a GTT-labelling, and condition~(STC~1) is not really necessary, if we consider simplicial twisting cochains that resolve \emph{complexes} of coherent sheaves.
What remains to show is that (STC~2) is satisfied, but this is exactly the content of \cref{theorem:tot0-twist-is-twcs}, i.e. that the $\varphi_J(\sigma)$ (which constitute the ${}^\sigma\mathfrak{a}$) satisfy the Maurer--Cartan equation.

\subsection{Proof of \cref{theorem:1-simplices-in-complex-analytic-twist} --- $\Tot^1\Twist(\cechnerve\cover_\bullet)$ recovers weak equivalences}
\label{appendix:proof-of-1-simplices-in-complex-analytic-twist}

We start by unravelling the definition of a weak equivalence of twisting cochains (\cite[Definition~2.27]{Wei2016}) and spelling out the explicit conditions in the first three degrees;
we then do the same for the definition of a $1$-simplex in the totalisation, and show that the conditions agree with those of a weak equivalence.
Finally, we give a general combinatorial argument that applies in arbitrary degree.

Let $(E^\bullet,\varphi)$ be a twisting cochain: the data of $E_\alpha^\bullet\in\Free(U_\alpha)$ for all $U_\alpha\in\cover$, and $\varphi=\sum_{p\in\mathbb{N}}\varphi^{p,1-p}$ satisfying the Maurer--Cartan equation, where
\[
  \varphi_{\alpha_0\ldots\alpha_p}^{p,1-p}\colon E_{\alpha_p}^\bullet\to E_{\alpha_0}^\bullet[p-1]
\]
and with $\varphi_{\alpha}^{0,1}$ being exactly the differential of $E_\alpha$.
If $(F^\bullet,\psi)$ is another twisting cochain, then a degree-$0$ morphism $\Lambda\colon(F^\bullet,\psi)\to(E^\bullet,\varphi)$ is, following \cite[Definition~2.12]{Wei2016}, the data of maps (that do not necessarily commute with the differentials)
\[
  \Lambda_{\alpha_0\ldots\alpha_p}^{p,-p}\colon F_{\alpha_p}^\bullet\to E_{\alpha_0}^\bullet[p].
\]
This morphism is a \emph{weak equivalence} (\cite[Definition~2.27]{Wei2016}) if the $\Lambda_\alpha^{0,0}$ are quasi-isomorphisms, and if the $\Lambda^{p,-p}$ satisfy
\[
  \hat{\delta}\Lambda + \varphi\cdot\Lambda - \Lambda\cdot\psi = 0.
\]
(Morally, this is asking that $\operatorname{D}\Lambda\coloneqq(\hat{\delta}+\partial)\Lambda=0$, where we think of $\partial$ as an analogue to the differential in the category of chain complexes, given by the difference between pre- and post-composition with the two ``differentials'', which are now the twisting cochains $\varphi$ and $\psi$).
All the terms in this equation are of total degree~$1$, and so we can consider what happens in each different bidegree.\footnote{We omit explicitly writing the degrees of the $\varphi$, $\psi$, and $\Lambda$ terms from now on, since they can be deduced from the degree of the simplex $\alpha_0\ldots\alpha_p$ in the subscript, knowing that $\deg\varphi=\deg\psi=1$ and $\deg\Lambda=0$.}
As is often the case, ensuring that the signs are correct is the majority of the work, and there are two things to be aware of: the composition of a $(p,q)$-term with an $(r,s)$-term has a sign of $(-1)^{qr}$ (\cite[\S2.2, Equation~3]{Wei2016}); and the differential of a morphism $f\colon A\to B$ of degree~$n$ in a dg-category of chain complexes is given by $\partial f=f\circ\mathrm{d}_A+(-1)^{n+1}\mathrm{d}_B\circ f$.

\subsubsection{The first three terms}

To simplify notation, we will write $E_{\alpha_0\ldots\alpha_p}$ (resp. $F_{\alpha_0\ldots\alpha_p}$) instead of $\varphi_{\alpha_0\ldots\alpha_p}$ (resp. $\psi_{\alpha_0\ldots\alpha_p}$).
For the sake of clarity, we usually denote the complex by $E_{\alpha_i}^\bullet$ so as not to confuse it with its differential $E_{\alpha_i}$.
\begin{itemize}
  \item The $(0,1)$-terms tell us that
    \[
      E_\alpha\circ\Lambda_\alpha - \Lambda_\alpha\circ F_\alpha
      = 0
      \tag{$\star_0$}
    \]
    which simply says that $\Lambda_\alpha$ is a chain map from $F_\alpha^\bullet$ to $E_\alpha^\bullet$.

  \item The $(1,0)$-terms tell us that
    \[
      \begin{aligned}
        0
        =\phantom{+}& E_{\alpha\beta}\Lambda_\beta - \Lambda_\alpha F_{\alpha\beta}
      \\-& E_\alpha\Lambda_{\alpha\beta} - \Lambda_{\alpha\beta}F_\beta
      \end{aligned}
    \]
    which we can rewrite as
    \[
      \partial\Lambda_{\alpha\beta} = E_{\alpha\beta}\Lambda_\beta - \Lambda_\alpha F_{\alpha\beta}.
      \tag{$\star_1$}
    \]

  \item The $(2,-1)$ terms tell us that
    \[
      \begin{aligned}
        0
        =\phantom{+}& E_{\alpha\beta\gamma}\Lambda_{\gamma} - \Lambda_\alpha F_{\alpha\beta\gamma}
      \\+& E_{\alpha\beta}\Lambda_{\beta\gamma} + \Lambda_{\alpha\beta}F_{\beta\gamma}
      \\+& E_{\alpha}\Lambda_{\alpha\beta\gamma} - \Lambda_{\alpha\beta\gamma}F_\gamma
      \\+& \Lambda_{\alpha\gamma}
      \end{aligned}
    \]
    which we can rewrite as
    \[
      \begin{aligned}
        \partial\Lambda_{\alpha\beta\gamma}
        =\phantom{+}& E_{\alpha\beta\gamma}\Lambda_\gamma - \Lambda_\alpha F_{\alpha\beta\gamma}
      \\+& E_{\alpha\beta}\Lambda_{\beta\gamma} + \Lambda_{\alpha\beta}F_{\beta\gamma}
      \\+& \Lambda_{\alpha\gamma}.
      \end{aligned}
      \tag{$\star_2$}
    \]
\end{itemize}

Now we look at the first three terms of a $1$-simplex in $\Tot\Twist(\cechnerve\cover_\bullet)$.
To do this, we need to understand the simplicial structure of $\Delta[p]\times\Delta[1]$ for $p\in\mathbb{N}$.
Fortunately there is a simple combinatorial description of all sub-simplices, given by considering certain paths in a two-dimensional grid: a non-degenerate $q$-simplex in $\Delta[p]\times\Delta[1]$ is given by a strictly increasing sequence of $(q+1)$ pairs $(i,j)$, for $0\leq i\leq p$ and $0\leq j\leq 1$, where ``strictly increasing'' means that at least one of $i$ and $j$ increases with each successive element.
We write such a sequence as $\simplexpath{i_0&i_1&\ldots&i_q\\j_0&j_1&\ldots&j_q}$.
For example, in $\Delta[1]\times\Delta[1]$,
\begin{itemize}
  \item the $0$-simplices are $\simplexpath{0\\0}$, $\simplexpath{1\\0}$, $\simplexpath{0\\1}$, and $\simplexpath{1\\1}$;
  \item the $1$-simplices are $\simplexpath{0&1\\0&0}$, $\simplexpath{0&0\\0&1}$, $\simplexpath{1&1\\0&1}$, $\simplexpath{0&1\\1&1}$, and $\simplexpath{0&1\\0&1}$;
  \item the $2$-simplices are $\simplexpath{0&1&1\\0&0&1}$ and $\simplexpath{0&0&1\\0&1&1}$.
\end{itemize}
Thinking of $\simplexpath{i\\j}$ as a coordinate, $\simplexpath{i_0&i_1\\j_0&j_1}$ as the path from $\simplexpath{i_0\\j_0}$ to $\simplexpath{i_1\\j_1}$, and $\simplexpath{i_0&i_1&i_2\\j_0&j_1&j_2}$ as the triangle given by the vertices $\simplexpath{i_0\\j_0}$, $\simplexpath{i_1\\j_1}$, and $\simplexpath{i_2\\j_2}$, we can draw $\Delta[1]\times\Delta[1]$ as a square with diagonal, as in \cref{figure:abstract-diagonal-square}.

\begin{figure}[ht!]
  \centering
  \begin{tikzpicture}[scale=3.5]
    \node (v00) at (0,0) {$\simplexpath{0\\0}$};
    \node (v10) at (1,0) {$\simplexpath{1\\0}$};
    \node (v01) at (0,1) {$\simplexpath{0\\1}$};
    \node (v11) at (1,1) {$\simplexpath{1\\1}$};
    \node (centre) at (0.5,0.5) {};
    \draw[thick,-latex] (v00) to node[fill=white]{\scriptsize$\simplexpath{0&0\\0&1}$} (v01);
    \draw[thick,-latex] (v10) to node[fill=white]{\scriptsize$\simplexpath{1&1\\0&1}$} (v11);
    \draw[thick,-latex] (v01) to node[fill=white]{\scriptsize$\simplexpath{0&1\\1&1}$} (v11);
    \draw[thick,-latex] (v00) to node[fill=white]{\scriptsize$\simplexpath{0&1\\0&0}$} (v10);
    \draw[thick,-latex] (v00) to node[fill=white]{\scriptsize$\simplexpath{0&1\\0&1}$} (v11);
    \draw[white] (centre) to node[black]{\scriptsize$\simplexpath{0&0&1\\0&1&1}$} (v01);
    \draw[white] (centre) to node[black]{\scriptsize$\simplexpath{0&1&1\\0&0&1}$} (v10);
  \end{tikzpicture}
  \caption{The canonical simplicial structure of $\Delta[1]\times\Delta[1]$.}
  \label{figure:abstract-diagonal-square}
\end{figure}

The story is entirely analogous in higher dimensions: see \cref{figure:abstract-and-labelled-triangular-prism} for the case of $\Delta[2]\times\Delta[1]$, and \cite[Appendix~B.2]{GMTZ2022a} for the general description of $\Delta[p]\times\Delta[q]$ (though here we are only concerned with $\Delta[p]\times\Delta[1]$).

\begin{figure}[ht!]
  \centering
  \begin{tikzpicture}[scale=2.5]
    \node (v00) at (0,0) {$\simplexpath{0\\0}$};
    \node (v10) at (1,0.5) {$\simplexpath{1\\0}$};
    \node (v20) at (2,0) {$\simplexpath{2\\0}$};
    \node (v01) at (0,2) {$\simplexpath{0\\1}$};
    \node (v11) at (1,2.5) {$\simplexpath{1\\1}$};
    \node (v21) at (2,2) {$\simplexpath{2\\1}$};
    \draw[thick,-stealth] (v00) to node[fill=white]{\scriptsize$\simplexpath{0&0\\0&1}$} (v01);
    \draw[thick,dashed,-stealth] (v10) to node[fill=white]{\scriptsize$\simplexpath{1&1\\0&1}$} (v11);
    \draw[thick,-stealth] (v20) to node[fill=white]{\scriptsize$\simplexpath{2&2\\0&1}$} (v21);
    \draw[thick,dashed,-stealth] (v00) to node[fill=white]{\scriptsize$\simplexpath{0&1\\0&0}$} (v10);
    \draw[thick,dashed,-stealth] (v10) to node[fill=white]{\scriptsize$\simplexpath{1&2\\0&0}$} (v20);
    \draw[thick,-stealth] (v00) to node[fill=white]{\scriptsize$\simplexpath{0&2\\0&0}$} (v20);
    \draw[thick,-stealth] (v01) to node[fill=white]{\scriptsize$\simplexpath{0&1\\1&1}$} (v11);
    \draw[thick,-stealth] (v11) to node[fill=white]{\scriptsize$\simplexpath{1&2\\1&1}$} (v21);
    \draw[thick,-stealth] (v01) to node[fill=white,near end]{\scriptsize$\simplexpath{0&2\\1&1}$} (v21);
    \draw[thick,dashed,-stealth] (v00) to node[fill=white]{\scriptsize$\simplexpath{0&1\\0&1}$} (v11);
    \draw[thick,dashed,-stealth] (v10) to node[fill=white]{\scriptsize$\simplexpath{1&2\\0&1}$} (v21);
    \draw[thick,-stealth] (v00) to node[fill=white,near start]{\scriptsize$\simplexpath{0&2\\0&1}$} (v21);
  \end{tikzpicture}
  \qquad\qquad
  \begin{tikzpicture}[scale=2.5]
    \node (v00) at (0,0) {$E_\alpha^\bullet$};
    \node (v10) at (1,0.5) {$E_\beta^\bullet$};
    \node (v20) at (2,0) {$E_\gamma^\bullet$};
    \node (v01) at (0,2) {$F_\alpha^\bullet$};
    \node (v11) at (1,2.5) {$F_\beta^\bullet$};
    \node (v21) at (2,2) {$F_\gamma^\bullet$};
    \draw[thick,-stealth] (v00) to node[fill=white]{\scriptsize$\lambda_\alpha$} (v01);
    \draw[thick,dashed,-stealth] (v10) to node[fill=white]{\scriptsize$\lambda_\beta$} (v11);
    \draw[thick,-stealth] (v20) to node[fill=white]{\scriptsize$\lambda_\gamma$} (v21);
    \draw[thick,dashed,-stealth] (v00) to node[fill=white]{\scriptsize$E_{\alpha\beta}$} (v10);
    \draw[thick,dashed,-stealth] (v10) to node[fill=white]{\scriptsize$E_{\beta\gamma}$} (v20);
    \draw[thick,-stealth] (v00) to node[fill=white]{\scriptsize$E_{\alpha\gamma}$} (v20);
    \draw[thick,-stealth] (v01) to node[fill=white]{\scriptsize$F_{\alpha\beta}$} (v11);
    \draw[thick,-stealth] (v11) to node[fill=white]{\scriptsize$F_{\beta\gamma}$} (v21);
    \draw[thick,-stealth] (v01) to node[fill=white,near end]{\scriptsize$F_{\alpha\gamma}$} (v21);
    \draw[thick,dashed,-stealth] (v00) to node[fill=white]{\scriptsize$d_{\alpha\beta}$} (v11);
    \draw[thick,dashed,-stealth] (v10) to node[fill=white]{\scriptsize$d_{\beta\gamma}$} (v21);
    \draw[thick,-stealth] (v00) to node[fill=white,near start]{\scriptsize$d_{\alpha\gamma}$} (v21);
  \end{tikzpicture}
  \caption{\emph{Left:} The canonical simplicial structure of $\Delta[2]\times\Delta[1]$. \emph{Right:} The data of a morphism $\Delta[2]\times\Delta[1]\to\Twist(\cechnerve\cover_2)$. Note that the $2$- and $3$-simplices are not labelled in either figure, for the sake of legibility, but the three constituent $3$-simplices are drawn in \cref{figure:three-abstract-tetrahedra}.}
  \label{figure:abstract-and-labelled-triangular-prism}
\end{figure}

\begin{itemize}
  \item Since $\Delta[0]\times\Delta[1]\cong\Delta[1]$, the degree-$0$ component of a $1$-simplex in $\Tot\Twist(\cechnerve\cover_\bullet)$ is given by
    \[
      \begin{tikzpicture}
        \node (F) at (0,2) {$F_\alpha^\bullet$};
        \node (E) at (0,0) {$E_\alpha^\bullet$};
        \draw[thick,-latex] (E) to node[fill=white]{\scriptsize$\lambda_\alpha$} (F);
      \end{tikzpicture}
    \]
    which, by the definition of the dg-nerve, is exactly the data of a quasi-isomorphism\footnote{Recall that morphisms in the dg-nerve go in the ``backwards'' direction, i.e. from $x_{i_k}$ to $x_{i_0}$ (\cref{definition:dg-nerve}).} $\lambda_\alpha\colon F_\alpha^\bullet\simto E_\alpha^\bullet$.
    This means that $\lambda_\alpha$ satisfies
    \[
      E_\alpha\circ\lambda_\alpha - \lambda_\alpha\circ F_\alpha = 0
      \tag{$\ast_0$}
    \]
    since quasi-isomorphisms are, in particular, chain maps.

  \item As described above (and in \cref{figure:abstract-diagonal-square}), the product $\Delta[1]\times\Delta[1]$ is a square with diagonal.
    This means that the degree-$1$ component is of the form
    \[
      \begin{tikzpicture}[scale=3]
        \node (v00) at (0,0) {$E_\alpha^\bullet$};
        \node (v10) at (1,0) {$E_\beta^\bullet$};
        \node (v01) at (0,1) {$F_\alpha^\bullet$};
        \node (v11) at (1,1) {$F_\beta^\bullet$};
        \node (centre) at (0.5,0.5) {};
        \draw[thick,-latex] (v00) to node[fill=white]{\scriptsize$\lambda_\alpha$} (v01);
        \draw[thick,-latex] (v10) to node[fill=white]{\scriptsize$\lambda_\beta$} (v11);
        \draw[thick,-latex] (v01) to node[fill=white]{\scriptsize$F_{\alpha\beta}$} (v11);
        \draw[thick,-latex] (v00) to node[fill=white]{\scriptsize$E_{\alpha\beta}$} (v10);
        \draw[thick,-latex] (v00) to node[fill=white,near start]{\scriptsize$d_{\alpha\beta}$} (v11);
        \draw[thick,double,double distance=2pt,-{Implies}] (centre) to node[fill=white]{\scriptsize$h_{\alpha\beta}^F$} (v01);
        \draw[thick,double,double distance=2pt,-{Implies}] (centre) to node[fill=white]{\scriptsize$h_{\alpha\beta}^E$} (v10);
      \end{tikzpicture}
    \]
    where $\lambda_\alpha$ and $\lambda_\beta$ are the degree-$0$ components from above, $E_{\alpha\beta}$ and $F_{\alpha\beta}$ are the degree-$(1,0)$ components of the twisting cochains, $d_{\alpha\beta}$ is some quasi-isomorphism, and $h_{\alpha\beta}^E$ and $h_{\alpha\beta}^F$ are homotopies satisfying
    \[
      \begin{aligned}
        \partial h_{\alpha\beta}^E
        &= d_{\alpha\beta} - E_{\alpha\beta}\lambda_\beta
      \\\partial h_{\alpha\beta}^F
        &= d_{\alpha\beta} - \lambda_\alpha F_{\alpha\beta}.
      \end{aligned}
    \]
    As mentioned in \cref{remark:1-simplices-in-twist-and-cubical}, this is where we need to ``forget'' that we are working simplicially and construct something cubical.
    To do this, we define
    \[
      \lambda_{\alpha\beta}
      \coloneqq h_{\alpha\beta}^F - h_{\alpha\beta}^E
    \]
    giving us the diagram
    \[
      \begin{tikzpicture}[scale=3]
        \node (v00) at (0,0) {$E_\alpha^\bullet$};
        \node (v10) at (1,0) {$E_\beta^\bullet$};
        \node (v01) at (0,1) {$F_\alpha^\bullet$};
        \node (v11) at (1,1) {$F_\beta^\bullet$};
        \draw[thick,-latex] (v00) to node[fill=white]{\scriptsize$\lambda_\alpha$} (v01);
        \draw[thick,-latex] (v10) to node[fill=white]{\scriptsize$\lambda_\beta$} (v11);
        \draw[thick,-latex] (v01) to node[fill=white]{\scriptsize$F_{\alpha\beta}$} (v11);
        \draw[thick,-latex] (v00) to node[fill=white]{\scriptsize$E_{\alpha\beta}$} (v10);
        \draw[thick,double,double distance=2pt,-{Implies}] (v10) to node[fill=white]{\scriptsize$\lambda_{\alpha\beta}$} (v01);
      \end{tikzpicture}
    \]
    where $\lambda_{\alpha\beta}$ satisfies
    \[
      \partial \lambda_{\alpha\beta}
      = E_{\alpha\beta}\lambda_\beta - \lambda_\alpha F_{\alpha\beta}
      \tag{$\ast_1$}
    \]
    by linearity of $\partial$.

  \item The degree-$2$ component is given by a labelling of the canonical triangulation of $\Delta[2]\times\Delta[1]$, as shown in \cref{figure:abstract-and-labelled-triangular-prism}.
    There are three $3$-simplices (i.e. tetrahedra) that make up this triangular prism, each one labelled with a $3$-simplex in the dg-nerve: the equations that these three homotopies satisfy are shown in \cref{figure:three-abstract-tetrahedra}.
    Applying the linearity of $\partial$ to the equations in \cref{figure:three-abstract-tetrahedra}, we see that
    \[
      \begin{aligned}
        \partial\left(
          f_{\simplexpath{0&1&2&2\\0&0&0&1}}
          - f_{\simplexpath{0&1&1&2\\0&0&1&1}}
          + f_{\simplexpath{0&0&1&2\\0&1&1&1}}
        \right)
        =\phantom{+}& f_{\simplexpath{0&2&2\\0&0&1}}
          - f_{\simplexpath{0&0&2\\0&1&1}}
        \\+& f_{\simplexpath{0&1\\0&0}}\left(f_{\simplexpath{1&1&2\\0&1&1}} - f_{\simplexpath{1&2&2\\0&0&1}}\right)
        \\+& \left(f_{\simplexpath{0&0&1\\0&1&1}} - f_{\simplexpath{0&1&1\\0&0&1}}\right)f_{\simplexpath{1&2\\1&1}}
        \\-& f_{\simplexpath{0&0\\0&1}}f_{\simplexpath{0&1&2\\1&1&1}} + f_{\simplexpath{0&1&2\\0&0&0}}f_{\simplexpath{2&2\\0&1}}.
      \end{aligned}
    \]
    This means that, if we define $\lambda_{\alpha\beta\gamma}$ to be the alternating sum of these three homotopies
    \[
      \lambda_{\alpha\beta\gamma}
      \coloneqq f_{\simplexpath{0&1&2&2\\0&0&0&1}}
      - f_{\simplexpath{0&1&1&2\\0&0&1&1}}
      + f_{\simplexpath{0&0&1&2\\0&1&1&1}}.
    \]
    then, using \cref{figure:abstract-and-labelled-triangular-prism}, the above translates to
    \[
      \begin{aligned}
        \partial\lambda_{\alpha\beta\gamma}
        =\phantom{+}& h_{\alpha\gamma}^F - h_{\alpha\gamma}^E
      \\+& E_{\alpha\beta}\left(h_{\beta\gamma}^F - h_{\beta\gamma}^E\right)
      \\+& \left(h_{\alpha\beta}^F - h_{\alpha\beta}^E\right)F_{\beta\gamma}
      \\-& \lambda_\alpha F_{\alpha\beta\gamma} + E_{\alpha\beta\gamma}\lambda_\gamma
      \end{aligned}
    \]
    which rearranges to give
    \[
      \begin{aligned}
        \partial\lambda_{\alpha\beta\gamma}
        =\phantom{+}& E_{\alpha\beta\gamma}\lambda_\gamma - \lambda_\alpha F_{\alpha\beta\gamma}
      \\+& E_{\alpha\beta}\lambda_{\beta\gamma} + \lambda_{\alpha\beta}F_{\beta\gamma}
      \\+& \lambda_{\alpha\gamma}.
      \end{aligned}
      \tag{$\ast_2$}
    \]
\end{itemize}

But then equations~$(\star_i)$ and $(\ast_i)$ are identical for $i=0,1,2$.
In other words, \emph{up to degree~$2$}, a $1$-simplex $(\lambda_\alpha,\lambda_{\alpha\beta},\lambda_{\alpha\beta\gamma})$ in the totalisation defines exactly a weak equivalence $(\Lambda_\alpha,\Lambda_{\alpha\beta},\Lambda_{\alpha\beta\gamma})$ of twisting cochains.
It remains only to give a general argument for arbitrary degree.

\begin{figure}[ht!]
  \centering
  \[
  \begin{tikzpicture}[scale=2.3,baseline=60pt]
    \node (v00) at (0,0) {$\simplexpath{0\\0}$};
    \node (v10) at (1,0.5) {$\simplexpath{1\\0}$};
    \node (v20) at (2,0) {$\simplexpath{2\\0}$};
    \node (v21) at (2,2) {$\simplexpath{2\\1}$};
    \draw[thick,-stealth] (v20) to node[fill=white]{\scriptsize$\simplexpath{2&2\\0&1}$} (v21);
    \draw[thick,dashed,-stealth] (v00) to node[fill=white]{\scriptsize$\simplexpath{0&1\\0&0}$} (v10);
    \draw[thick,dashed,-stealth] (v10) to node[fill=white]{\scriptsize$\simplexpath{1&2\\0&0}$} (v20);
    \draw[thick,-stealth] (v00) to node[fill=white]{\scriptsize$\simplexpath{0&2\\0&0}$} (v20);
    \draw[thick,dashed,-stealth] (v10) to node[fill=white]{\scriptsize$\simplexpath{1&2\\0&1}$} (v21);
    \draw[thick,-stealth] (v00) to node[fill=white]{\scriptsize$\simplexpath{0&2\\0&1}$} (v21);
  \end{tikzpicture}
  \qquad\qquad
  \begin{aligned}
    \partial f_{\simplexpath{0&1&2&2\\0&0&0&1}}
    =\phantom{+}& f_{\simplexpath{0&2&2\\0&0&1}} - f_{\simplexpath{0&1&2\\0&0&1}}
  \\-& f_{\simplexpath{0&1\\0&0}}f_{\simplexpath{1&2&2\\0&0&1}} + f_{\simplexpath{0&1&2\\0&0&0}}f_{\simplexpath{2&2\\0&1}}
  \end{aligned}
  \]
  \[
  \begin{tikzpicture}[scale=2.3,baseline=80pt]
    \node (v00) at (0,0) {$\simplexpath{0\\0}$};
    \node (v10) at (1,0.5) {$\simplexpath{1\\0}$};
    \node (v11) at (1,2.5) {$\simplexpath{1\\1}$};
    \node (v21) at (2,2) {$\simplexpath{2\\1}$};
    \draw[thick,dashed,-stealth] (v10) to node[fill=white]{\scriptsize$\simplexpath{1&1\\0&1}$} (v11);
    \draw[thick,-stealth] (v00) to node[fill=white]{\scriptsize$\simplexpath{0&1\\0&0}$} (v10);
    \draw[thick,-stealth] (v11) to node[fill=white]{\scriptsize$\simplexpath{1&2\\1&1}$} (v21);
    \draw[thick,-stealth] (v00) to node[fill=white]{\scriptsize$\simplexpath{0&1\\0&1}$} (v11);
    \draw[thick,-stealth] (v10) to node[fill=white]{\scriptsize$\simplexpath{1&2\\0&1}$} (v21);
    \draw[thick,-stealth] (v00) to node[fill=white,near start]{\scriptsize$\simplexpath{0&2\\0&1}$} (v21);
  \end{tikzpicture}
  \qquad\qquad
  \begin{aligned}
    \partial f_{\simplexpath{0&1&1&2\\0&0&1&1}}
    =\phantom{+}& f_{\simplexpath{0&1&2\\0&1&1}} - f_{\simplexpath{0&1&2\\0&0&1}}
  \\-& f_{\simplexpath{0&1\\0&0}}f_{\simplexpath{1&1&2\\0&1&1}} + f_{\simplexpath{0&1&1\\0&0&1}}f_{\simplexpath{1&2\\1&1}}
  \end{aligned}
  \]
  \[
  \begin{tikzpicture}[scale=2.3,baseline=60pt]
    \node (v00) at (0,0) {$\simplexpath{0\\0}$};
    \node (v01) at (0,2) {$\simplexpath{0\\1}$};
    \node (v11) at (1,2.5) {$\simplexpath{1\\1}$};
    \node (v21) at (2,2) {$\simplexpath{2\\1}$};
    \draw[thick,-stealth] (v00) to node[fill=white]{\scriptsize$\simplexpath{0&0\\0&1}$} (v01);
    \draw[thick,-stealth] (v01) to node[fill=white]{\scriptsize$\simplexpath{0&1\\1&1}$} (v11);
    \draw[thick,-stealth] (v11) to node[fill=white]{\scriptsize$\simplexpath{1&2\\1&1}$} (v21);
    \draw[thick,-stealth] (v01) to node[fill=white]{\scriptsize$\simplexpath{0&2\\1&1}$} (v21);
    \draw[thick,dashed,-stealth] (v00) to node[fill=white]{\scriptsize$\simplexpath{0&1\\0&1}$} (v11);
    \draw[thick,-stealth] (v00) to node[fill=white]{\scriptsize$\simplexpath{0&2\\0&1}$} (v21);
  \end{tikzpicture}
  \qquad\qquad
  \begin{aligned}
    \partial f_{\simplexpath{0&0&1&2\\0&1&1&1}}
    =\phantom{+}& f_{\simplexpath{0&1&2\\0&1&1}} - f_{\simplexpath{0&0&2\\0&1&1}}
  \\-& f_{\simplexpath{0&0\\0&1}}f_{\simplexpath{0&1&2\\1&1&1}} + f_{\simplexpath{0&0&1\\0&1&1}}f_{\simplexpath{1&2\\1&1}}
  \end{aligned}
  \]
  \caption{\emph{Left:} The three $3$-simplices of $\Delta[2]\times\Delta[1]$. \emph{Right:} The equations satisfied by the corresponding homotopy in the dg-nerve, where $f_{[-]}$ is the morphism labelling the simplex $[-]$.}
  \label{figure:three-abstract-tetrahedra}
\end{figure}

\subsubsection{Full proof}

Now we give the argument for arbitrary degree.
For $0\leq m\leq p$, we define (cf. \cref{figure:delta-m-p})
\[
  \Delta_m^{p+1}
  \coloneqq \simplexpath{0&1&2&\ldots&m&m&m+1&\ldots&p-1&p\\0&0&0&\ldots&0&1&1&\ldots&1&1}
\]
i.e. the non-degenerate $(p+1)$-simplex of $\Delta[p]\times\Delta[1]$ given by travelling along the bottom $p$-simplex (corresponding to $\Delta[p]\times\{0\}$) for $m$ steps (between the first $m+1$ vertices), then travelling straight up to the top $p$-simplex (corresponding to $\Delta[p]\times\{1\}$), before continuing on along the remaining vertices.
Note that these are \emph{all} of the non-degenerate $(p+1)$-simplices of $\Delta[p]\times\Delta[1]$.

\begin{figure}[ht!]
  \centering
  \begin{tikzpicture}[scale=0.75]
    \begin{scope}
      \foreach \x in {0,1,2,3} {
        \node (\x0) at (\x,0) {$\bullet$};
        \node (\x1) at (\x,1) {$\bullet$};
      }
      \draw[thick] (00) to (01) to (11) to (21) to (31);
      \node at (1.5,-0.75) {$\Delta_0^3$};
    \end{scope}
    \begin{scope}[shift={(5,0)}]
      \foreach \x in {0,1,2,3} {
        \node (\x0) at (\x,0) {$\bullet$};
        \node (\x1) at (\x,1) {$\bullet$};
      }
      \draw[thick] (00) to (10) to (11) to (21) to (31);
      \node at (1.5,-0.75) {$\Delta_1^3$};
    \end{scope}
    \begin{scope}[shift={(10,0)}]
      \foreach \x in {0,1,2,3} {
        \node (\x0) at (\x,0) {$\bullet$};
        \node (\x1) at (\x,1) {$\bullet$};
      }
      \draw[thick] (00) to (10) to (20) to (21) to (31);
      \node at (1.5,-0.75) {$\Delta_2^3$};
    \end{scope}
    \begin{scope}[shift={(15,0)}]
      \foreach \x in {0,1,2,3} {
        \node (\x0) at (\x,0) {$\bullet$};
        \node (\x1) at (\x,1) {$\bullet$};
      }
      \draw[thick] (00) to (10) to (20) to (30) to (31);
      \node at (1.5,-0.75) {$\Delta_3^3$};
    \end{scope}
  \end{tikzpicture}
  \caption{The four non-degenerate $4$-simplices of $\Delta[3]\times\Delta[1]$, represented as paths between vertices: in each diagram, the bottom row consists of the vertices $\simplexpath{0\\0}$ to $\simplexpath{3\\0}$ and the top row of the vertices $\simplexpath{0\\1}$ to $\simplexpath{3\\1}$.}
  \label{figure:delta-m-p}
\end{figure}

We can think of a morphism $\Delta[p]\times\Delta[1]\to\Twist(\cechnerve\cover_p)$ as a labelling of the ``prism'' $\Delta[p]\times\Delta[1]$, where the vertices are labelled with objects $x_{i_j}\in\Twist(\cechnerve\cover_p)^0$, and each simplex $I=\{{i_0}<\ldots<{i_k}\}$ with $k\geq2$ is labelled by $f_I\in\Hom^{k-1}(x_{i_k},x_{i_0})$ satisfying the equation \cref{equation:dg-nerve-definition} defining the dg-nerve, which here is exactly
\[
  \partial f_I
  + \sum_{j=1}^{k-1} (-1)^j f_{I\setminus\{i_j\}}
  + \sum_{j=1}^{k-1} (-1)^{k(j-1)} f_{\{i_0<\ldots<i_j\}}f_{\{i_j<\ldots<i_k\}}
  = 0.
  \tag{$\star$}
\]
Generalising the notation of \cref{figure:abstract-and-labelled-triangular-prism}, we write $E_{\alpha_i}^\bullet$ for the object labelling the vertex $\simplexpath{i\\0}$, and $F_{\alpha_i}^\bullet$ for the object labelling the vertex $\simplexpath{i\\1}$;
given $I=\{i_0<\ldots<i_k\}\subseteq[p]$ with $k\geq1$, we write $E_{\alpha_{i_0}\ldots\alpha_{i_k}}$ for the morphism $f_{\simplexpath{I\\0}}$, and $F_{\alpha_{i_0}\ldots\alpha_{i_k}}$ for the morphism $f_{\simplexpath{I\\1}}$, where $\simplexpath{I\\j}=\simplexpath{i_0&i_1&\ldots&i_k\\j&j&\ldots&j}$.

Given any simplex $(\alpha_0\ldots\alpha_p)$, we define
\[
  \lambda_{\alpha_0\ldots\alpha_p}
  \coloneqq \sum_{m=0}^p (-1)^m f_{\Delta_m^{p+1}}.
\]
By $(\star)$, we know that
\begin{align}
  &\sum_{m=0}^p (-1)^m \partial f_{\Delta_m^{p+1}}
  \tag{$\circledast_\partial$}
\\+& \sum_{m=0}^p \sum_{j=1}^p (-1)^m(-1)^j f_{\Delta_m^{p+1}\setminus\ver_j}
  \tag{$\circledast_{\hat{\delta}}$}
\\+& \sum_{m=0}^p \sum_{j=1}^p (-1)^m(-1)^{(p+1)(j-1)} f_{\Delta_m^{p+1}(0,j)} f_{\Delta_m^{p+1}(j,p+1)}
  \tag{$\circledast_\circ$}
\\=& \,0
  \nonumber
\end{align}
where we introduce two notational shorthands: given a simplex $\sigma$ we write $\sigma\setminus\ver_j$ to mean $\sigma\setminus\ver_j\sigma$, and $\sigma(i,j)$ to mean $\{\ver_i\sigma<\ldots<\ver_j\sigma\}$.
We will now examine each of $(\circledast_\partial)$, $(\circledast_{\hat{\delta}})$, and $(\circledast_\circ)$ in turn.

\bigskip

Firstly, $(\circledast_\partial)$.
By the definition of $\partial$, along with its linearity, and the definition of $\lambda_{\alpha_0\ldots\alpha_p}$, we see that this is exactly
\begin{align}
  \sum_{m=0}^p (-1)^m \partial f_{\Delta_m^{p+1}}
  &= \sum_{m=0}^p (-1)^m \left(f_{\Delta_m^{p+1}}\circ F_{\alpha_p} + (-1)^{p+1} E_{\alpha_0}\circ f_{\Delta_m^{p+1}}\right)
  \nonumber
\\&= \left(\sum_{m=0}^p (-1)^m f_{\Delta_m^{p+1}}\right)\circ F_{\alpha_p} + (-1)^{p+1} E_{\alpha_0}\circ\left(\sum_{m=0}^p (-1)^m f_{\Delta_m^{p+1}}\right)
  \nonumber
\\&= \lambda_{\alpha_0\ldots\alpha_p}F_{\alpha_p} + (-1)^{p+1} E_{\alpha_0}\lambda_{\alpha_0\ldots\alpha_p}.
  \tag{$\boxtimes_\partial$}
\end{align}

\bigskip

Next, $(\circledast_{\hat{\delta}})$.
A table showing all terms in the case $p=3$ is given in \cref{figure:hatdelta-components-for-delta-m-4}, and might prove useful in visualising some of the combinatorics.
First of all, note that $\Delta_m^{p+1}\setminus\ver_m=\Delta_{m-1}^{p+1}\setminus\ver_m$ for all $m=1,\ldots,p$, and the corresponding morphisms in $(\circledast_{\hat{\delta}})$ have opposite signs, since $j=m+1$ is fixed for both, and so they cancel.
Now the remaining morphisms can all\footnote{Indeed, the only morphisms labelling simplices that pass through one of $\simplexpath{j\\0}$ and $\simplexpath{j\\1}$ but not the other are exactly those with diagonal paths, but $\Delta_m^{p+1}$ consists entirely of horizontal and vertical paths (cf. \cref{figure:delta-m-p}) and so the only terms that contain diagonal paths are $\Delta_m^{p+1}\setminus\ver_m$ and $\Delta_m^{p+1}\setminus\ver_{m+1}$, which are exactly the previously mentioned terms that cancel each other out.} be grouped together: those labelling simplices that skip over both $\simplexpath{1\\0}$ and $\simplexpath{1\\1}$; those labelling simplices that skip over both $\simplexpath{2\\0}$ and $\simplexpath{2\\1}$; and so on, up to those that skip over both $\simplexpath{p-1\\0}$ and $\simplexpath{p-1\\1}$.
But the set of all simplices that skip over both $\simplexpath{j\\0}$ and $\simplexpath{j\\1}$ is exactly the set of all $\Delta_m^p$ for $m=0,\ldots,p-1$ where we label the top and bottom $p-1$-simplices in $\Delta[p-1]\times\Delta[1]$ by $\{0<1<\ldots<\hat{j}<\ldots<p\}$.
For example, $\Delta_0^4\setminus\ver_2$, $\Delta_2^4\setminus\ver_1$, and $\Delta_3^4\setminus\ver_1$ are exactly $\Delta_0^3$, $\Delta_1^3$, and $\Delta_2^3$ (respectively) on the copy of $\Delta[2]\times\Delta[1]$ given by $\{0<2<3\}\times\{0<1\}$ (see \cref{figure:hatdelta-components-for-delta-m-4}).
We denote these copies of $\Delta_m^p$ on $\{0<\ldots<\hat{j}<\ldots<p\}\times\{0<1\}$ by $\Delta_m^{p+1\setminus j}$.

In the bottom half of \cref{figure:hatdelta-components-for-delta-m-4} we represent how the 12 terms in the case $p=3$ either cancel out or group together, and this pattern generalises: by the above, $(\circledast_{\hat{\delta}})$ is exactly\footnote{To check that all terms are indeed recovered, note that the right-hand side partitions the set of indices $\{(m,j) \mid 0\leq m\leq p,1\leq j\leq p\}$ into three sets: the ``thick diagonal'' $\{(m,j) \mid m=j\}\cup\{(m,j) \mid m=j-1\}$, the lower triangle $\{(m,j) \mid m<j-1\}$, and the upper triangle $\{(m,j) \mid m>j\}$ (with this last partition being slightly obscured, since there is a shift $j\mapsto j-1$ in the indexing of the removed vertex).}
\[
  \begin{aligned}
    \sum_{m=0}^p \sum_{j=1}^p (-1)^m(-1)^j f_{\Delta_m^{p+1}\setminus\ver_j}
    &= \sum_{m=1}^p (-1)^m \left( f_{\Delta_m^{p+1}\setminus\ver_m} - f_{\Delta_{m-1}^{p+1}\setminus\ver_m} \right)
  \\&+ \sum_{j=2}^p(-1)^j \left( \sum_{m=0}^{j-2} (-1)^m f_{\Delta_m^{p+1}\setminus\ver_j} + \sum_{m=j}^{p}(-1)^{m+1} f_{\Delta_m^{p+1}\setminus\ver_{j-1}} \right)
  \end{aligned}
\]
but the first sum on the right vanishes (since $\Delta_m^{p+1}\setminus\ver_m=\Delta_{m-1}^{p+1}\setminus\ver_m$) and the second simplifies (following the pattern of \cref{figure:hatdelta-components-for-delta-m-4}) to
\[
  \sum_{j=1}^{p-1} \sum_{m=0}^{p-1} (-1)^{j-1} (-1)^m f_{\Delta_m^{p+1\setminus j}}
  = \sum_{j=1}^{p-1} (-1)^{j-1} \left( \sum_{m=0}^{p-1} f_{\Delta_m^{p+1\setminus j}} \right)
\]
which is
\[
  -\sum_{j=1}^{p-1}(-1)^j \lambda_{\alpha_0\ldots\widehat{\alpha_j}\ldots\alpha_p}
  \tag{$\boxtimes_{\hat{\delta}}$}
\]
precisely by our definition of $\lambda$.

\begin{figure}[ht!]
  \centering
  \begin{tikzpicture}[scale=0.65] 
    \begin{scope} 
      \foreach \x in {0,1,2,3} {
        \node (\x0) at (\x,0) {$\bullet$};
        \node (\x1) at (\x,1) {$\bullet$};
      }
      \draw[thick] (00) to (01) to (11) to (21) to (31);
      \begin{scope}[shift={(0,-5)}]
        \begin{scope}[shift={(0,2)}] 
          \node (00) at (0,0) {$\bullet$};
          \node (11) at (1,1) {$\bullet$};
          \node (21) at (2,1) {$\bullet$};
          \node (31) at (3,1) {$\bullet$};
          \draw[thick] (00) to (11) to (21) to (31);
        \end{scope}
        \begin{scope}[shift={(0,0)}] 
          \node (00) at (0,0) {$\bullet$};
          \node (01) at (0,1) {$\bullet$};
          \node (21) at (2,1) {$\bullet$};
          \node (31) at (3,1) {$\bullet$};
          \draw[thick] (00) to (01) to (21) to (31);
        \end{scope}
        \begin{scope}[shift={(0,-2)}] 
          \node (00) at (0,0) {$\bullet$};
          \node (01) at (0,1) {$\bullet$};
          \node (11) at (1,1) {$\bullet$};
          \node (31) at (3,1) {$\bullet$};
          \draw[thick] (00) to (01) to (11) to (31);
        \end{scope}
      \end{scope}
    \end{scope} 
    \begin{scope}[shift={(5,0)}] 
      \foreach \x in {0,1,2,3} {
        \node (\x0) at (\x,0) {$\bullet$};
        \node (\x1) at (\x,1) {$\bullet$};
      }
      \draw[thick] (00) to (10) to (11) to (21) to (31);
      \begin{scope}[shift={(0,-5)}]
        \begin{scope}[shift={(0,2)}] 
          \node (00) at (0,0) {$\bullet$};
          \node (11) at (1,1) {$\bullet$};
          \node (21) at (2,1) {$\bullet$};
          \node (31) at (3,1) {$\bullet$};
          \draw[thick] (00) to (11) to (21) to (31);
        \end{scope}
        \begin{scope}[shift={(0,0)}] 
          \node (00) at (0,0) {$\bullet$};
          \node (10) at (1,0) {$\bullet$};
          \node (21) at (2,1) {$\bullet$};
          \node (31) at (3,1) {$\bullet$};
          \draw[thick] (00) to (10) to (21) to (31);
        \end{scope}
        \begin{scope}[shift={(0,-2)}] 
          \node (00) at (0,0) {$\bullet$};
          \node (10) at (1,0) {$\bullet$};
          \node (11) at (1,1) {$\bullet$};
          \node (31) at (3,1) {$\bullet$};
          \draw[thick] (00) to (10) to (11) to (31);
        \end{scope}
      \end{scope}
    \end{scope} 
    \begin{scope}[shift={(10,0)}] 
      \foreach \x in {0,1,2,3} {
        \node (\x0) at (\x,0) {$\bullet$};
        \node (\x1) at (\x,1) {$\bullet$};
      }
      \draw[thick] (00) to (10) to (20) to (21) to (31);
      \begin{scope}[shift={(0,-5)}]
        \begin{scope}[shift={(0,2)}] 
          \node (00) at (0,0) {$\bullet$};
          \node (20) at (2,0) {$\bullet$};
          \node (21) at (2,1) {$\bullet$};
          \node (31) at (3,1) {$\bullet$};
          \draw[thick] (00) to (20) to (21) to (31);
        \end{scope}
        \begin{scope}[shift={(0,0)}] 
          \node (00) at (0,0) {$\bullet$};
          \node (10) at (1,0) {$\bullet$};
          \node (21) at (2,1) {$\bullet$};
          \node (31) at (3,1) {$\bullet$};
          \draw[thick] (00) to (10) to (21) to (31);
        \end{scope}
        \begin{scope}[shift={(0,-2)}] 
          \node (00) at (0,0) {$\bullet$};
          \node (10) at (1,0) {$\bullet$};
          \node (20) at (2,0) {$\bullet$};
          \node (31) at (3,1) {$\bullet$};
          \draw[thick] (00) to (10) to (20) to (31);
        \end{scope}
      \end{scope}
    \end{scope} 
    \begin{scope}[shift={(15,0)}] 
      \foreach \x in {0,1,2,3} {
        \node (\x0) at (\x,0) {$\bullet$};
        \node (\x1) at (\x,1) {$\bullet$};
      }
      \draw[thick] (00) to (10) to (20) to (30) to (31);
      \begin{scope}[shift={(0,-5)}]
        \begin{scope}[shift={(0,2)}] 
          \node (00) at (0,0) {$\bullet$};
          \node (20) at (2,0) {$\bullet$};
          \node (30) at (3,0) {$\bullet$};
          \node (31) at (3,1) {$\bullet$};
          \draw[thick] (00) to (20) to (30) to (31);
        \end{scope}
        \begin{scope}[shift={(0,0)}] 
          \node (00) at (0,0) {$\bullet$};
          \node (10) at (1,0) {$\bullet$};
          \node (30) at (3,0) {$\bullet$};
          \node (31) at (3,1) {$\bullet$};
          \draw[thick] (00) to (10) to (30) to (31);
        \end{scope}
        \begin{scope}[shift={(0,-2)}] 
          \node (00) at (0,0) {$\bullet$};
          \node (10) at (1,0) {$\bullet$};
          \node (20) at (2,0) {$\bullet$};
          \node (31) at (3,1) {$\bullet$};
          \draw[thick] (00) to (10) to (20) to (31);
        \end{scope}
      \end{scope}
    \end{scope} 
    \draw[thick] (-3,-1) to (19,-1);
    \draw[thick] (-1,-7.5) to (-1,1.5);
    \node at (-2.5,0.5) {\scriptsize$\Delta_m^4$};
    \node at (-2.5,-2.5) {\scriptsize$\Delta_m^4\setminus\ver_1$};
    \node at (-2.5,-4.5) {\scriptsize$\Delta_m^4\setminus\ver_2$};
    \node at (-2.5,-6.5) {\scriptsize$\Delta_m^4\setminus\ver_3$};
    \node at (1.5,2) {\scriptsize$m=0$};
    \node at (6.5,2) {\scriptsize$m=1$};
    \node at (11.5,2) {\scriptsize$m=2$};
    \node at (16.5,2) {\scriptsize$m=3$};
  \end{tikzpicture}

  \bigskip
  
  \begin{tikzpicture}[xscale=1.5,yscale=1.25] 
    \node at (0,3) {$0$};
    \node at (1,3) {$1$};
    \node at (2,3) {$2$};
    \node at (3,3) {$3$};
    \node at (-1,2) {$1$};
    \node at (-1,1) {$2$};
    \node at (-1,0) {$3$};
    \node at (-1,3) {${}_j\setminus^m$};
    \draw[thick] (-1.5,2.5) to (3.5,2.5);
    \draw[thick] (-0.5,3.5) to (-0.5,-0.5);
    \node (m0j1) at (0,2) {$\times$};
    \node (m0j2) at (0,1) {$\Delta_0^{4\setminus1}$};
    \node (m0j3) at (0,0) {$\Delta_0^{4\setminus2}$};
    \node (m1j1) at (1,2) {$\times$};
    \node (m1j2) at (1,1) {$\times$};
    \node (m1j3) at (1,0) {$\Delta_1^{4\setminus2}$};
    \node (m2j1) at (2,2) {$\Delta_1^{4\setminus1}$};
    \node (m2j2) at (2,1) {$\times$};
    \node (m2j3) at (2,0) {$\times$};
    \node (m3j1) at (3,2) {$\Delta_2^{4\setminus1}$};
    \node (m3j2) at (3,1) {$\Delta_2^{4\setminus2}$};
    \node (m3j3) at (3,0) {$\times$};
    \draw[thick] (m0j1) to (m1j1);
    \draw[thick] (m1j2) to (m2j2);
    \draw[thick] (m2j3) to (m3j3);
    \draw[thick] (m0j2) to (m2j1) to (m3j1);
    \draw[thick] (m0j3) to (m1j3) to (m3j2);
  \end{tikzpicture}
  \caption{\emph{Top:} The 12 terms of $(\circledast_{\hat{\delta}})$ in the case $p=3$. \emph{Bottom:} The $(j,m)$ entry corresponds to $\Delta_m^4\setminus\ver_j$. The entries marked with a $\times$ cancel out with the other $\times$ to which it is connected by a line; the remaining entries are labelled with $\Delta_m^{4\setminus j}$, denoting the copy of $\Delta_m^3$ on $\{0<\ldots<\hat{j}<\ldots<3\}\times\{0<1\}$, and the ones joined with lines are the ones that add to give the $\lambda_{\alpha_0\ldots\widehat{\alpha_j}\ldots\alpha_3}$ term. Note that entries that are horizontally or vertically adjacent to one another in this table have opposite signs in $(\circledast_{\hat{\delta}})$.}
  \label{figure:hatdelta-components-for-delta-m-4}
\end{figure}

\bigskip

Finally, $(\circledast_\circ)$.
A table showing all terms in the case $p=3$ is given in \cref{figure:circ-components-for-delta-m-4}.
To start, we split up the sum in $(\circledast_\circ)$ into an ``upper triangular'' part and a ``lower triangular'' part:
\[
  \begin{gathered}
    \sum_{m=0}^p \sum_{j=1}^p (-1)^m (-1)^{(p+1)(j-1)} f_{\Delta_m^{p+1}(0,j)}f_{\Delta_m^{p+1}(j,p+1)}
  \\= \sum_{j=1}^p \sum_{m=0}^{j-1} (-1)^m (-1)^{(p+1)(j-1)} f_{\Delta_m^{p+1}(0,j)}f_{\Delta_m^{p+1}(j,p+1)}
  \\+ \sum_{j=1}^p \sum_{m=j}^{p} (-1)^m (-1)^{(p+1)(j-1)} f_{\Delta_m^{p+1}(0,j)}f_{\Delta_m^{p+1}(j,p+1)}.
  \end{gathered}
\]
Now, in the first sum on the right-hand side, the $f_{\Delta_m^{p+1}(j,p+1)}$ term will be constant over all values of $m$, since if $m<j$ then $\Delta_m^{p+1}(j,p+1)$ is simply $\simplexpath{j-1&j&\ldots&p\\1&1&\ldots&1}$, and similarly for the $f_{\Delta_m^{p+1}(0,j)}$ term in the second sum.
In other words, we can write $(\circledast_\circ)$ as
\[
  \begin{aligned}
    & \sum_{j=1}^p \left( \sum_{m=0}^{j-1} (-1)^m (-1)^{(p+1)(j-1)} f_{\Delta_m^j(0,j)}\right) f_{\simplexpath{j-1&j&\ldots&p\\1&1&\ldots&1}}
  \\+& \sum_{j=1}^p  f_{\simplexpath{0&1&\ldots&j\\0&0&\ldots&0}} \left( \sum_{m=j}^{p} (-1)^m (-1)^{(p+1)(j-1)}f_{\Delta_{m-j}^{p+1-j}+\simplexpath{j\\0}} \right)
  \end{aligned}
\]
where we write $\Delta_{m-j}^{p+1-j}+\simplexpath{j\\0}$ to mean repeated vertex-wise addition, i.e.
\[
  \begin{aligned}
    \Delta_{n}^{q} + \simplexpath{i\\0}
    &= \simplexpath{0&1&2&\ldots&n&n&n+1&\ldots&q-1&q\\0&0&0&\ldots&0&1&1&\ldots&1&1} + \simplexpath{i\\0}
  \\&\coloneqq \simplexpath{i&1+i&2+i&\ldots&n+i&n+i&n+1+i&\ldots&q-1+i&q+i\\0&0&0&\ldots&0&1&1&\ldots&1&1}.
  \end{aligned}
\]
But this is exactly
\[
  \sum_{j=1}^p (-1)^{(p+1)(j-1)} \lambda_{\alpha_0\ldots\alpha_{j-1}} F_{\alpha_{j-1}\ldots\alpha_p} + \sum_{j=1}^p (-1)^{(p+1)(j-1)} (-1)^j E_{\alpha_0\ldots\alpha_j}\lambda_{\alpha_j\ldots\alpha_p}
\]
and shifting the first sum from $\sum_{j=1}^p$ to $\sum_{j=0}^{p-1}$ and the second sum from $\sum_{m=j}^p$ to $\sum_{m=0}^{p-j}$ gives
\[
  \sum_{j=0}^{p-1} (-1)^{(p+1)j}\lambda_{\alpha_0\ldots\alpha_j}F_{\alpha_j\ldots\alpha_p} + \sum_{j=1}^p (-1)^{(p+1)(j-1)} (-1)^j E_{\alpha_0\ldots\alpha_j}\lambda_{\alpha_j\ldots\alpha_p}.
  \tag{$\boxtimes_\circ$}
\]

\begin{figure}[ht!]
  \centering
  \begin{tikzpicture}[scale=0.65] 
    \begin{scope} 
      \foreach \x in {0,1,2,3} {
        \node (\x0) at (\x,0) {$\bullet$};
        \node (\x1) at (\x,1) {$\bullet$};
      }
      \draw[thick] (00) to (01) to (11) to (21) to (31);
      \begin{scope}[shift={(0,-5)}]
        \begin{scope}[shift={(-0.5,2)}] 
          \node (00) at (0,0) {\scriptsize$0$};
          \node (01) at (0,1) {\scriptsize$0$};
          \draw[thick] (00) to (01);
          \draw[thick,dashed] (0.5,-0.25) to (0.5,1.25);
        \end{scope}
        \begin{scope}[shift={(0.5,2)}]
          \node (01) at (0,1) {\scriptsize$0$};
          \node (11) at (1,1) {\scriptsize$1$};
          \node (21) at (2,1) {\scriptsize$2$};
          \node (31) at (3,1) {\scriptsize$3$};
          \draw[thick] (01) to (11) to (21) to (31);
        \end{scope}
        \begin{scope}[shift={(-0.5,0)}] 
          \node (00) at (0,0) {\scriptsize$0$};
          \node (01) at (0,1) {\scriptsize$0$};
          \node (11) at (1,1) {\scriptsize$1$};
          \draw[thick] (00) to (01) to (11);
          \draw[thick,dashed] (1.5,-0.25) to (1.5,1.25);
        \end{scope}
        \begin{scope}[shift={(0.5,0)}]
          \node (11) at (1,1) {\scriptsize$1$};
          \node (21) at (2,1) {\scriptsize$2$};
          \node (31) at (3,1) {\scriptsize$3$};
          \draw[thick] (11) to (21) to (31);
        \end{scope}
        \begin{scope}[shift={(-0.5,-2)}] 
          \node (00) at (0,0) {\scriptsize$0$};
          \node (01) at (0,1) {\scriptsize$0$};
          \node (11) at (1,1) {\scriptsize$1$};
          \node (21) at (2,1) {\scriptsize$2$};
          \draw[thick] (00) to (01) to (11) to (21);
          \draw[thick,dashed] (2.5,-0.25) to (2.5,1.25);
        \end{scope}
        \begin{scope}[shift={(0.5,-2)}]
          \node (21) at (2,1) {\scriptsize$2$};
          \node (31) at (3,1) {\scriptsize$3$};
          \draw[thick] (21) to (31);
        \end{scope}
      \end{scope}
    \end{scope} 
    \begin{scope}[shift={(5,0)}] 
      \foreach \x in {0,1,2,3} {
        \node (\x0) at (\x,0) {$\bullet$};
        \node (\x1) at (\x,1) {$\bullet$};
      }
      \draw[thick] (00) to (10) to (11) to (21) to (31);
      \begin{scope}[shift={(0,-5)}]
        \begin{scope}[shift={(-0.5,2)}] 
          \node (00) at (0,0) {\scriptsize$0$};
          \node (10) at (1,0) {\scriptsize$1$};
          \draw[thick] (00) to (10);
          \draw[thick,dashed] (1.5,-0.25) to (1.5,1.25);
        \end{scope}
        \begin{scope}[shift={(0.5,2)}]
          \node (10) at (1,0) {\scriptsize$1$};
          \node (11) at (1,1) {\scriptsize$1$};
          \node (21) at (2,1) {\scriptsize$2$};
          \node (31) at (3,1) {\scriptsize$3$};
          \draw[thick] (10) to (11) to (21) to (31);
        \end{scope}
        \begin{scope}[shift={(-0.5,0)}] 
          \node (00) at (0,0) {\scriptsize$0$};
          \node (10) at (1,0) {\scriptsize$1$};
          \node (11) at (1,1) {\scriptsize$1$};
          \draw[thick] (00) to (10) to (11);
          \draw[thick,dashed] (1.5,-0.25) to (1.5,1.25);
        \end{scope}
        \begin{scope}[shift={(0.5,0)}]
          \node (11) at (1,1) {\scriptsize$1$};
          \node (21) at (2,1) {\scriptsize$2$};
          \node (31) at (3,1) {\scriptsize$3$};
          \draw[thick] (11) to (21) to (31);
        \end{scope}
        \begin{scope}[shift={(-0.5,-2)}] 
          \node (00) at (0,0) {\scriptsize$0$};
          \node (10) at (1,0) {\scriptsize$1$};
          \node (11) at (1,1) {\scriptsize$1$};
          \node (21) at (2,1) {\scriptsize$2$};
          \draw[thick] (00) to (10) to (11) to (21);
          \draw[thick,dashed] (2.5,-0.25) to (2.5,1.25);
        \end{scope}
        \begin{scope}[shift={(0.5,-2)}]
          \node (21) at (2,1) {\scriptsize$2$};
          \node (31) at (3,1) {\scriptsize$3$};
          \draw[thick] (21) to (31);
        \end{scope}
      \end{scope}
    \end{scope} 
    \begin{scope}[shift={(10,0)}] 
      \foreach \x in {0,1,2,3} {
        \node (\x0) at (\x,0) {$\bullet$};
        \node (\x1) at (\x,1) {$\bullet$};
      }
      \draw[thick] (00) to (10) to (20) to (21) to (31);
      \begin{scope}[shift={(0,-5)}]
        \begin{scope}[shift={(-0.5,2)}] 
          \node (00) at (0,0) {\scriptsize$0$};
          \node (10) at (1,0) {\scriptsize$1$};
          \draw[thick] (00) to (10);
          \draw[thick,dashed] (1.5,-0.25) to (1.5,1.25);
        \end{scope}
        \begin{scope}[shift={(0.5,2)}]
          \node (10) at (1,0) {\scriptsize$1$};
          \node (20) at (2,0) {\scriptsize$2$};
          \node (21) at (2,1) {\scriptsize$2$};
          \node (31) at (3,1) {\scriptsize$3$};
          \draw[thick] (10) to (20) to (21) to (31);
        \end{scope}
        \begin{scope}[shift={(-0.5,0)}] 
          \node (00) at (0,0) {\scriptsize$0$};
          \node (10) at (1,0) {\scriptsize$1$};
          \node (20) at (2,0) {\scriptsize$2$};
          \draw[thick] (00) to (10) to (20);
          \draw[thick,dashed] (2.5,-0.25) to (2.5,1.25);
        \end{scope}
        \begin{scope}[shift={(0.5,0)}]
          \node (20) at (2,0) {\scriptsize$2$};
          \node (21) at (2,1) {\scriptsize$2$};
          \node (31) at (3,1) {\scriptsize$3$};
          \draw[thick] (20) to (21) to (31);
        \end{scope}
        \begin{scope}[shift={(-0.5,-2)}] 
          \node (00) at (0,0) {\scriptsize$0$};
          \node (10) at (1,0) {\scriptsize$1$};
          \node (20) at (2,0) {\scriptsize$2$};
          \node (21) at (2,1) {\scriptsize$2$};
          \draw[thick] (00) to (10) to (20) to (21);
          \draw[thick,dashed] (2.5,-0.25) to (2.5,1.25);
        \end{scope}
        \begin{scope}[shift={(0.5,-2)}]
          \node (21) at (2,1) {\scriptsize$2$};
          \node (31) at (3,1) {\scriptsize$3$};
          \draw[thick] (21) to (31);
        \end{scope}
      \end{scope}
    \end{scope} 
    \begin{scope}[shift={(15,0)}] 
      \foreach \x in {0,1,2,3} {
        \node (\x0) at (\x,0) {$\bullet$};
        \node (\x1) at (\x,1) {$\bullet$};
      }
      \draw[thick] (00) to (10) to (20) to (30) to (31);
      \begin{scope}[shift={(0,-5)}]
        \begin{scope}[shift={(-0.5,2)}] 
          \node (00) at (0,0) {\scriptsize$0$};
          \node (10) at (1,0) {\scriptsize$1$};
          \draw[thick] (00) to (10);
          \draw[thick,dashed] (1.5,-0.25) to (1.5,1.25);
        \end{scope}
        \begin{scope}[shift={(0.5,2)}]
          \node (10) at (1,0) {\scriptsize$1$};
          \node (20) at (2,0) {\scriptsize$2$};
          \node (30) at (3,0) {\scriptsize$3$};
          \node (31) at (3,1) {\scriptsize$3$};
          \draw[thick] (10) to (20) to (30) to (31);
        \end{scope}
        \begin{scope}[shift={(-0.5,0)}] 
          \node (00) at (0,0) {\scriptsize$0$};
          \node (10) at (1,0) {\scriptsize$1$};
          \node (20) at (2,0) {\scriptsize$2$};
          \draw[thick] (00) to (10) to (20);
          \draw[thick,dashed] (2.5,-0.25) to (2.5,1.25);
        \end{scope}
        \begin{scope}[shift={(0.5,0)}]
          \node (20) at (2,0) {\scriptsize$2$};
          \node (30) at (3,0) {\scriptsize$3$};
          \node (31) at (3,1) {\scriptsize$3$};
          \draw[thick] (20) to (30) to (31);
        \end{scope}
        \begin{scope}[shift={(-0.5,-2)}] 
          \node (00) at (0,0) {\scriptsize$0$};
          \node (10) at (1,0) {\scriptsize$1$};
          \node (20) at (2,0) {\scriptsize$2$};
          \node (30) at (3,0) {\scriptsize$3$};
          \draw[thick] (00) to (10) to (20) to (30);
          \draw[thick,dashed] (3.5,-0.25) to (3.5,1.25);
        \end{scope}
        \begin{scope}[shift={(0.5,-2)}]
          \node (30) at (3,0) {\scriptsize$3$};
          \node (31) at (3,1) {\scriptsize$3$};
          \draw[thick] (30) to (31);
        \end{scope}
      \end{scope}
    \end{scope} 
    \draw[thick] (-3,-1) to (19,-1);
    \draw[thick] (-1,-7.5) to (-1,1.5);
    \node at (-2,0.5) {\scriptsize$\Delta_m^4$};
    \node at (-2.5,-2.5) {\scriptsize$\Delta_m^4(0,1)(1,4)$};
    \node at (-2.5,-4.5) {\scriptsize$\Delta_m^4(0,2)(2,4)$};
    \node at (-2.5,-6.5) {\scriptsize$\Delta_m^4(0,3)(3,4)$};
    \node at (1.5,2) {\scriptsize$m=0$};
    \node at (6.5,2) {\scriptsize$m=1$};
    \node at (11.5,2) {\scriptsize$m=2$};
    \node at (16.5,2) {\scriptsize$m=3$};
  \end{tikzpicture}

  \bigskip
  
  \begin{tikzpicture}[xscale=3.25,yscale=1.25] 
    \node at (0,3) {$0$};
    \node at (1,3) {$1$};
    \node at (2,3) {$2$};
    \node at (3,3) {$3$};
    \node at (-0.75,2) {$1$};
    \node at (-0.75,1) {$2$};
    \node at (-0.75,0) {$3$};
    \node at (-0.75,3) {${}_j\setminus^m$};
    \draw[thick] (-0.96,2.5) to (3.5,2.5);
    \draw[thick] (-0.5,3.5) to (-0.5,-0.5);
    \node (m0j1) at (0,2) {$\lambda_0\circ F_{0123}$};
    \node (m0j2) at (0,1) {$\Delta_0^2\circ F_{123}$};
    \node (m0j3) at (0,0) {$\Delta_0^3\circ F_{23}$};
    \node (m1j1) at (1,2) {$E_{01}\circ(\Delta_0^3+\simplexpath{1\\0})$};
    \node (m1j2) at (1,1) {$\Delta_1^2\circ F_{123}$};
    \node (m1j3) at (1,0) {$\Delta_1^3\circ F_{23}$};
    \node (m2j1) at (2,2) {$E_{01}\circ(\Delta_1^3+\simplexpath{1\\0})$};
    \node (m2j2) at (2,1) {$E_{012}\circ(\Delta_0^2+\simplexpath{2\\0})$};
    \node (m2j3) at (2,0) {$\Delta_2^3\circ F_{23}$};
    \node (m3j1) at (3,2) {$E_{01}\circ(\Delta_2^3+\simplexpath{1\\0})$};
    \node (m3j2) at (3,1) {$E_{012}\circ(\Delta_1^2+\simplexpath{2\\0})$};
    \node (m3j3) at (3,0) {$E_{0123}\circ \lambda_3$};
    \draw[thick] (m1j1) to (m2j1) to (m3j1);
    \draw[thick] (m0j2) to (m1j2);
    \draw[thick] (m2j2) to (m3j2);
    \draw[thick] (m0j3) to (m1j3) to (m2j3);
  \end{tikzpicture}
  \caption{\emph{Top:} The 12 terms of $(\circledast_\circ)$ in the case $p=3$, where we write $\Delta_m^4(0,j)(j,4)$ as shorthand for $\Delta_m^4(0,j)$ followed by $\Delta_m^4(j,4)$. \emph{Bottom:} The entries are labelled with $f_{\Delta_m^4(0,j)}f_{\Delta_m^4(j,4)}$; those joined with a line are those that add to give a $\lambda_{\alpha_0\ldots\alpha_j}$ (or $\lambda_{\alpha_j\ldots\alpha_p}$) term. Note that the entries above (or on) the diagonal and on the same line all share the same first component, and those below the diagonal on the same line all share the same second component.}
  \label{figure:circ-components-for-delta-m-4}
\end{figure}

\bigskip

Putting this all together, we know that $(\boxtimes_\partial)+(\boxtimes_{\hat{\delta}})+(\boxtimes_\circ)=0$, and so
\[
  \begin{aligned}
    0
    &= -(\boxtimes_{\hat{\delta}}) -(\boxtimes_\partial) -(\boxtimes_\circ)
  \\&= \sum_{j=1}^{p-1} (-1)^j \lambda_{\alpha_0\ldots\widehat{\alpha_j}\ldots\alpha_p}
  \\&- \lambda_{\alpha_0\ldots\alpha_p}F_{\alpha_p} - (-1)^{p+1}E_{\alpha_0}\lambda_{\alpha_0\ldots\alpha_p}
  \\&- \sum_{j=0}^{p-1} (-1)^{(p+1)j} \lambda_{\alpha_0\ldots\alpha_j}F_{\alpha_j\ldots\alpha_p} - \sum_{j=1}^p (-1)^{(p+1)(j-1)} (-1)^j E_{\alpha_0\ldots\alpha_j}\lambda_{\alpha_j\ldots\alpha_p}
  \end{aligned}
\]
but the first sum is exactly $(\hat{\delta}\lambda)_{\alpha_0\ldots\alpha_p}$, so it remains only to show that $-(\boxtimes_\partial)-(\boxtimes_\circ)=E\lambda-\lambda F$, since then $\lambda$ satisfies
\[
  \hat{\delta}\lambda + E\lambda - \lambda F = 0
\]
which is exactly the condition necessary in order for $\lambda$ to be a weak equivalence (since the $\lambda_\alpha$ terms live in the maximal Kan complex of the dg-nerve, and are thus quasi-isomorphisms by \cref{lemma:nerve-inside-dg-nerve}).
We can merge the two terms of $(\boxtimes_\partial)$ into the two sums of $(\boxtimes_\circ)$ to obtain
\[
  \begin{aligned}
    -(\boxtimes_\partial) -(\boxtimes_\circ)
    &= -\sum_{j=0}^p (-1)^{(p+1)(j-1)}(-1)^j E_{\alpha_0\ldots\alpha_j}\lambda_{\alpha_j\ldots\alpha_p} -\sum_{j=0}^{p} (-1)^{(p+1)j} \lambda_{\alpha_0\ldots\alpha_j}F_{\alpha_j\ldots\alpha_p}
  \\&= \sum_{j=0}^p \left[ (-1)^{(p+1)(j-1)+j+1} E_{\alpha_0\ldots\alpha_j}\lambda_{\alpha_j\ldots\alpha_p} - (-1)^{(p+1)j}\lambda_{\alpha_0\ldots\alpha_j}F_{\alpha_j\ldots\alpha_p} \right].
  \end{aligned}
\]
But $(p+1)(j-1)+j+1\equiv(1-j)(p-j)\mod2$ and $(p+1)j\equiv-j(p-j)\mod2$, so we can write the above as
\[
  \sum_{j=0}^p \left[ (-1)^{(1-j)(p-j)} E_{\alpha_0\ldots\alpha_j}\lambda_{\alpha_j\ldots\alpha_p} - (-1)^{-j(p-j)}\lambda_{\alpha_0\ldots\alpha_j}F_{\alpha_j\ldots\alpha_p} \right]
\]
and this is exactly $E\lambda-\lambda F$.


\printbibliography[heading=bibintoc]

\end{document}